\documentclass[a4paper,11pt]{article}
\usepackage{amsmath}
\usepackage{amssymb}
\usepackage{amsthm}
\usepackage{MnSymbol}
\usepackage{mathrsfs}
\usepackage{graphicx,subfigure}
\usepackage{paralist} 
\usepackage{enumerate}
\usepackage{hyperref,url}

\usepackage{bm} 
\usepackage{placeins} 
\usepackage{color}
\usepackage[margin=30mm]{geometry}

\newcommand{\N}{\mathbb{N}}
\newcommand{\R}{\mathbb{R}}
\renewcommand{\div}{\mathrm{div} \, }
\newcommand{\dHaus}{\, d\mathcal{H}^{d-1}}
\newcommand{\dx}{\, d \mathcal{L}^{d}}
\newcommand{\ds}{\, ds}
\newcommand{\dt}{\, dt}
\newcommand{\dd}{\, d}
\newcommand{\dz}{\, dz}
\newcommand{\nut}{nu}
\newcommand{\tox}{tx}
\newcommand{\ptq}{pq}
\newcommand{\qtp}{qp}
\newcommand{\qtn}{qn}

\newcommand{\pd}{\partial}
\newcommand{\abs}[1]{\left| #1 \right|}

\newcommand{\eps}{\varepsilon}
\newcommand{\Laplace}{\Delta}
\newcommand{\surf}{\nabla_{\Gamma}}

\newcommand{\md}{\pd^{\bullet}_{t}}

\newcommand{\tr}[1]{\text{tr} \left ( #1 \right )}
\newcommand{\jump}[3]{\left [#1 \right ]_{#2}^{#3}}

\newcommand{\Gibbs}{\mathrm{G}}
\newcommand{\TangGibbs}{\mathrm{TG}}
\newcommand{\HGibbs}{\mathrm{HG}}
\newcommand{\unit}{\bm{1}}
\newcommand{\Proj}[1]{\mathbb{P} #1}
\newcommand{\TT}[1]{\underline{\underline{#1}}\rule{0pt}{0pt}}
\newcommand{\id}{\TT{\mathrm{I}}}
\newcommand{\bvarphi}{\bm\varphi}

\theoremstyle{plain}
\newtheorem{thm}{Theorem}[section]
\newtheorem{remark}{Remark}[section]
\newtheorem{assump}{Assumption}[section]
\numberwithin{equation}{section}

\title{A multiphase Cahn--Hilliard--Darcy model for tumour growth with necrosis}

\author{Harald Garcke \footnotemark[1] \and Kei Fong Lam\footnotemark[1] \and Robert N\"{u}rnberg \footnotemark[2] \and Emanuel Sitka \footnotemark[3] }

\date{ }

\begin{document}

\maketitle

\renewcommand{\thefootnote}{\fnsymbol{footnote}}
\footnotetext[1]{Fakult\"at f\"ur Mathematik, Universit\"at Regensburg, 93040 Regensburg, Germany
({\tt \{Harald.Garcke, Kei-Fong.Lam\}@mathematik.uni-regensburg.de}).}
\footnotetext[2]{Department of Mathematics, Imperial College London, London, SW7 2AZ, UK ({\tt robert.nurnberg@imperial.ac.uk}).}
\footnotetext[3]{Fakult\"at f\"ur Medizin, Universit\"at Regensburg, 93040 Regensburg, Germany ({\tt Emanuel.Sitka@stud.uni-regensburg.de}).}

\begin{abstract}
We derive a Cahn--Hilliard--Darcy model to describe multiphase tumour growth taking interactions with multiple chemical species into account as well as the simultaneous occurrence of proliferating, quiescent and necrotic regions.  Via a coupling of the Cahn--Hilliard--Darcy equations to a system of reaction-diffusion equations a multitude of phenomena such as nutrient diffusion and consumption, angiogenesis, hypoxia, blood vessel growth, and inhibition by toxic agents, which are released for example by the necrotic cells, can be included.  A new feature of the modelling approach is that a volume-averaged velocity is used, which dramatically simplifies the resulting equations.  With the help of formally matched asymptotic analysis we develop new sharp interface models.  Finite element numerical computations are performed and in particular the effects of necrosis on tumour growth is investigated numerically.

\end{abstract}



\renewcommand{\thefootnote}{\arabic{footnote}}

\section{Introduction}\label{sec:intro}
The morphological evolution of cancer cells, driven by chemical and biological mechanisms, is still poorly understood even in the simplest case of avascular tumour growth.  It is well-known that in the avascular stage, initially homogeneous tumour cells will eventually develop heterogeneity in their growth behaviour.  For example, quiescent cells appear when the tumour reaches a diffusion-limited size, where levels of nutrients, such as oxygen, are too low to support cell proliferation, and necrotic cells develop when the nutrient density drops further.  It is expected that angiogenic factors are secreted by the quiescent tumour cells to induce the development of a capillary network towards the tumour and deliver much required nutrients for proliferation \cite{Rey}.  But it has also been observed (experimentally \cite{Pennacchietti} and in numerical simulations \cite{CLLW,CLNie}), that the tumour exhibits morphological instabilities, driven by a combination of chemotactic gradients and inhomogeneous proliferation, which allows the interior tumour cells to access nutrients by increasing the surface area of the tumour interface.

In this paper, we propose a multi-component diffuse interface model for modelling heterogeneous tumour growth.  We consider $L$ types of cells, with $M$ chemical species.  Similar in spirit to Ambrosi and Preziosi \cite{Ambrosi} (see also \cite{Beysens,Foty,Ranft}), we model each of the $L$ different cell types as inertia-less fluids, and each of the $M$ chemical species can freely diffuse and may be subject to additional mechanisms such as chemotaxis and active transport.  In the diffuse interface methodology, interfaces between different components are modelled as thin transition layers, in which the macroscopically distinct components are allowed to mix microscopically.  This is in contrast to the sharp interface approach, where the interfaces are modelled as idealised moving hypersurfaces.  The treatment of cells as viscous inertia-less fluids naturally leads to a notion of an averaged velocity for the fluid mixture, and we will use a volume-averaged velocity, which is also considered in \cite{AGG,GLSS}.

From basic conservation laws, we will derive the following multi-component model:
\begin{subequations}\label{Intro:multiphase}
\begin{alignat}{3}
\div \vec{v} & = \unit \cdot \bm{U}(\bm{\varphi}, \bm{\sigma}), \label{Intro:div} \\
\vec{v} & = - K \nabla p + K (\nabla \bm{\varphi})^{\top}(\bm{\mu} - \bm{N}_{,\bm{\varphi}}(\bm{\varphi}, \bm{\sigma})), \label{Intro:Darcy} \\
\pd_{t} \bm{\varphi} + \div (\bm{\varphi} \otimes \vec{v}) & = \div (\TT{C}(\bm{\varphi}, \bm{\sigma}) \nabla \bm{\mu}) + \bm{U}(\bm{\varphi}, \bm{\sigma}), \label{Intro:varphi} \\
\bm{\mu} & = - \beta \eps \Laplace \bm{\varphi} + \beta \eps^{-1} \Psi_{,\bm{\varphi}}(\bvarphi) + \bm{N}_{,\bm{\varphi}}(\bm{\varphi}, \bm{\sigma}), \label{Intro:mu} \\
\pd_{t} \bm{\sigma} + \div (\bm{\sigma} \otimes \vec{v}) & = \div (\TT{D}(\bm{\varphi}, \bm{\sigma}) \nabla \bm{N}_{,\bm{\sigma}}(\bm{\varphi}, \bm{\sigma})) + \bm{S}(\bm{\varphi}, \bm{\sigma}), \label{Intro:sigma}
\end{alignat}
\end{subequations}
for a vector $\bm{\varphi} = (\varphi_{1}, \dots, \varphi_{L})^{\top}$ of volume fractions, i.e., $\sum_{i=1}^{L} \varphi_{i} = 1$ and $\varphi_{i} \geq 0$ for $1 \leq i \leq L$, where $\varphi_{i}$ represents the volume fraction of the $i$th cell type, and for a vector $\bm{\sigma} = (\sigma_{1}, \dots, \sigma_{M})^{\top}$, with $\sigma_{j}$ representing the density of the $j$th chemical species.  The velocity $\vec{v}$ is the volume-averaged velocity, $p$ is the pressure, $\bm{\mu} = (\mu_{1}, \dots, \mu_{L})^{\top}$ is the vector of chemical potentials associated to $\bm{\varphi}$, and $\bm{N}_{,\bm{\varphi}} \in \R^{L}$ and $\bm{N}_{,\bm{\sigma}} \in \R^{M}$ denote the partial derivatives of the chemical free energy density $N$ with respect to $\bm{\varphi}$ and $\bm{\sigma}$, respectively.

The system \eqref{Intro:multiphase} can be seen as the multi-component variant of the Cahn--Hilliard--Darcy system derived in Garcke et al.\ \cite{GLSS}.  Equation \eqref{Intro:sigma} can be viewed as a convection-reaction-diffusion system with a vector of source terms $\bm{S} \in \R^{M}$, where for vectors $\bm{a} \in \R^{k}$ and $\bm{b} \in \R^{l}$, the tensor product $\bm{a} \otimes \bm{b} \in \R^{k \times l}$ is defined as $(\bm{a} \otimes \bm{b})_{ij} = a_{i} b_{j}$ for $1 \leq i \leq k$ and $1 \leq j \leq l$.  The positive semi-definite mobility tensor $\TT{D}(\bm{\varphi}, \bm{\sigma})$ can be taken as a second order tensor in $\R^{M \times M}$, or even as a fourth order tensor in $\R^{M \times d \times M \times d}$, where $d$ is the spatial dimension.

Equations \eqref{Intro:varphi} and \eqref{Intro:mu} constitute a multi-component convective Cahn--Hilliard system with a vector of source terms $\bm{U} \in \R^{L}$ and a mobility tensor $\TT{C}(\bm{\varphi}, \bm{\sigma})$, which we take to be either a second order tensor in $\R^{L \times L}$ or a fourth order tensor in $\R^{L \times d \times L \times d}$.  Furthermore, we ask that $\sum_{i=1}^{L} \TT{C}_{ij} = 0$  in the former case and $\sum_{i=1}^{L} \TT{C}_{imjl} = 0$ in the latter case for any $1 \leq j \leq L$ and $1 \leq m,l \leq d$.  These conditions ensure that $\sum_{i=1}^{L} \varphi_{i}(t) = 1$ for $t > 0$ if $\sum_{i=1}^{L} \varphi_{i}(0) = 1$.  One example of such a second order mobility tensor is $\TT{C}_{ij}(\bm{\varphi},\bm\sigma) = m_{i}(\varphi_{i}) (\delta_{ij} - m_{j}(\varphi_{j})/ \sum_{k=1}^{L} m_{k}(\varphi_{k}))$ for $1 \leq i, j \leq L$ and so-called bare mobilities $m_{i}(\varphi_{i})$, see \cite{ElliottGarckeMulti}.  The vector $\Psi_{,\bm{\varphi}}$ is the vector of partial derivatives of a multi-well potential $\Psi$ with $L$ equal minima at the points $\bm{e}_{l}$, $l = 1, \dots, L$, where $\bm{e}_{l}$ is the $l$th unit vector in $\R^{L}$.  

Equation \eqref{Intro:Darcy} is a generalised Darcy's law (with permeability $K > 0$) relating the volume-averaged velocity $\vec{v}$ and the pressure $p$, while in equation \eqref{Intro:div}, $\unit = (1, \dots, 1)^{\top} \in \R^{L}$ and $\unit \cdot \bm{U}$ is the sum of the components of the vector of source terms $\bm{U}$ in \eqref{Intro:varphi}, and \eqref{Intro:div} relates the gain or loss of volume from the vector of source terms $\bm{U}$ to the changes of mass balance.  

Lastly, $\beta > 0$ and $\eps  > 0$ are parameters related to the surface tension and the interfacial thickness, respectively.
In fact, associated with \eqref{Intro:multiphase} is the free energy 
\begin{align}\label{Intro:energy}
\mathcal{E}(\bm{\varphi}, 
\bm{\sigma}) = \int_\Omega\frac{\beta \eps}{2} \sum_{i=1}^{L} \abs{\nabla \varphi_{i}}^{2} + \frac{\beta}{\eps} \Psi(\bm{\varphi}) + N(\bm{\varphi}, \bm{\sigma}) \dx,
\end{align}
where $\dx$ denotes integration with respect to the $d$ dimensional Lebesgue measure. The first two terms in the integral account for the interfacial energy (and by extension the adhesive properties of the different cell types), and the last term accounts for the free energy of the chemical species and their interaction with the cells.  

As a special case, we consider $L=3$ and $M=1$, so that we have three cell types; host cells $(\varphi_{1})$, proliferating tumour cells $(\varphi_{2})$ and necrotic cells $(\varphi_{3})$, along with one chemical species $(\sigma)$ acting as nutrient, for example oxygen.  Then, \eqref{Intro:sigma} becomes a scalar equation, with mobility $\TT{D}(\bm{\varphi},\sigma)$ chosen as a scalar function $D(\bm{\varphi},\sigma)$, and the vector $\bm{S}(\bm{\varphi}, \sigma)$ becomes a scalar function $S(\bm{\varphi}, \sigma)$.  In this case, one can consider a chemical free energy density of the form
\begin{align}
\label{Intro:nutrient:eg}
N(\bm{\varphi}, \sigma) = \frac{\chi_{\sigma}}{2} \abs{\sigma}^{2} - \chi_{\varphi} \sigma \varphi_{2} - \chi_{n} f(\sigma) \varphi_{3},
\end{align}
where $\chi_{\sigma} > 0, \chi_{\varphi}, \chi_{n} \geq 0$ are constants and $f: [0,\infty) \to [0,\infty)$ is a monotonically decreasing function such that $f(s) = 0$ for $s \geq c_{*} > 0$.  The first term of \eqref{Intro:nutrient:eg} will lead to diffusion of the nutrients, and the second term models the chemotaxis mechanism that drives the proliferating tumour cells to regions of high nutrient, which was similarly considered in \cite{CLLW, book:CristiniLowengrub, GLSS, HawkinsZeeOden12}.  The third term shows that it is energetically favourable to be in the necrotic phase when the nutrient density is below $c_{*}$.  Indeed, when $\sigma < c_{*}$, $f(\sigma)$ is positive, and so the term $-\chi_{n} f(\sigma) \varphi_{3}$ is negative when $\varphi_{3} = 1$.
Overall we obtain from \eqref{Intro:multiphase} the three-component model
\begin{subequations}\label{Intro:3component}
\begin{alignat}{3}
\div \vec{v} & = \unit \cdot \bm{U}(\bm{\varphi}, \sigma), \quad \vec{v} = - K \nabla p + K (\nabla \bm{\varphi})^{\top}( \bm{\mu} - \bm{N}_{,\bm{\varphi}}(\bm{\varphi}, \sigma)), \\
\pd_{t} \bm{\varphi} + \div (\bm{\varphi} \otimes \vec{v}) & = \div (\TT{C}(\bm{\varphi}, \sigma) \nabla \bm{\mu}) + \bm{U}(\bm{\varphi}, \sigma),  \\
\bm{\mu} & = - \beta \eps \Laplace \bm{\varphi} + \beta \eps^{-1} \Psi_{,\bm{\varphi}}(\bvarphi) + \bm{N}_{,\bm{\varphi}}(\bm{\varphi}, \sigma), \label{Intro:3component:mu} \\
\pd_{t}\sigma + \div (\sigma \vec{v}) & = \div (D(\bm{\varphi}, \sigma) \nabla ( \chi_{\sigma} \sigma - \chi_{\varphi} \varphi_{2} - \chi_{n} f'(\sigma) \varphi_{3})) + S(\bm{\varphi}, \sigma),  \label{Intro:3component:sigma} \\
\bm{N}_{,\bm{\varphi}}(\bvarphi,\sigma) & = (0, \; - \chi_{\varphi} \sigma, \; - \chi_{n} f(\sigma))^{\top}.
\end{alignat}
\end{subequations}
Similar to \cite{GLSS}, we define
\begin{align} \label{eq:lambda}
\lambda = \frac{\chi_{\varphi}}{\chi_{\sigma}},\quad \theta = \frac{\chi_{n}}{\chi_{\sigma}}, \quad d(\bm{\varphi}, \sigma)  = D(\bm{\varphi}, \sigma) \chi_{\sigma},
\end{align}
so that \eqref{Intro:3component:sigma} becomes
\begin{align}
\label{Intro:sigma:decoupled}
\pd_{t}\sigma + \div (\sigma \vec{v}) = \div (d(\bm{\varphi}, \sigma) \nabla (\sigma  - \lambda \varphi_{2} - \theta f'(\sigma) \varphi_{3})) + S(\bm{\varphi}, \sigma),
\end{align}
which allows us to decouple the chemotaxis mechanism that was appearing in \eqref{Intro:3component:mu} and \eqref{Intro:3component:sigma}.  We point out that it is possible to neglect the effects of fluid flow by sending $K \to 0$ in the case $\unit \cdot \bm{U} = 0$.  By Darcy's law and $\div \vec{v} = 0$, we obtain $\vec{v} \to \vec{0}$ as $K \to 0$, and the above system \eqref{Intro:3component} with source terms satisfying $\unit \cdot \bm{U}(\bm{\varphi}, \sigma) = 0$ transforms into  
\begin{equation*}
\begin{aligned}
\pd_{t} \bm{\varphi} & = \div (\TT{C}(\bm{\varphi}, \sigma) \nabla \bm{\mu}) + \bm{U}(\bm{\varphi}, \sigma), \\
\bm{\mu} & = - \beta \eps \Laplace \bm{\varphi} + \beta \eps^{-1} \Psi_{,\bm{\varphi}}(\bvarphi) + \bm{N}_{,\bm{\varphi}}(\bm{\varphi}, \sigma), \\
\pd_{t} \sigma & = \div (D(\bm{\varphi}, \sigma) \nabla (\chi_{\sigma} \sigma - \chi_{\varphi} \varphi_{2} - \chi_{n} f'(\sigma) \varphi_{3})) + S(\bm{\varphi}, \sigma).
\end{aligned}
\end{equation*}

We now consider the case that tumour cells prefer to adhere to each other instead of the host cells, for general $L \geq 2$ and $M \geq 1$.  More precisely, let $\varphi_{1}$ denote the volume fraction of the host cells, and $\varphi_{T} = 1 - \varphi_{1} = \sum_{i=2}^{L} \varphi_{i}$ is the total volume fraction of the $L-1$ types of tumour cells.  Then the following choice of interfacial energy is considered:
\begin{align}\label{deg:GL}
E(\bvarphi) = \int_{\Omega} \frac{\beta \eps}{2} \abs{\sum_{i=2}^{L} \nabla \varphi_{i}}^{2} + \frac{\beta}{\eps} W \left ( \sum_{i=2}^{L} \varphi_{i} \right ) \dx,
\end{align}
where $W$ is a potential with equal minima at $0$ and $1$.  Note that \eqref{deg:GL} can be viewed as a function of $\varphi_{T}$, i.e., $E(\bvarphi) = \hat{E}(\varphi_{T}) = \int_{\Omega}\frac{\beta \eps}{2} \abs{\nabla \varphi_{T}}^{2} + \frac{\beta}{\eps} W(\varphi_{T}) \dx$, and it is energetically favourable to have $\varphi_{T} = 0$ (representing the host tissues) or $\varphi_{T} = 1$ (representing the tumour as a whole).  It holds that the first variation of $E$ with respect to $\varphi_{i}$, $2 \leq i \leq L$, satisfies
\begin{align*}
\frac{\delta E}{\delta \varphi_{i}} = \frac{\delta \hat{E}}{\delta \varphi_{T}} = - \beta \eps \Laplace \varphi_{T} + \beta \eps^{-1} W'(\varphi_{T}) =: \mu_{T},
\end{align*}
and so, if the chemical free energy density $N$ is independent of $\bm{\varphi}$, the corresponding equations for the chemical potentials for the tumour phases now read as
\begin{align*}
\mu_{1} = 0, \quad \mu_{i} = - \beta \eps \Laplace \varphi_{T} + \beta \eps^{-1} W'(\varphi_{T}) = \mu_{T} \text{ for } 2 \leq i \leq L.
\end{align*}
Then, choosing a second order mobility tensor $\TT{C}(\bm{\varphi}, \bm\sigma)$ 
such that $\sum_{j=2}^{L} \TT{C}_{ij}(\bm{\varphi}, \bm\sigma) = M(\varphi_{i})$, 
for $2 \leq i \leq L$, and
$\sum_{j=2}^{L} \TT{C}_{1j}(\bm{\varphi}, \bm\sigma) = - \sum_{j=2}^{L} M(\varphi_{j})$ 
for a non-negative mobility $M$, 
the equations for $\varphi_{i}$ take the form
\begin{align*}
\pd_{t}\varphi_{i} + \div (\varphi_{i} \vec{v}) & = \div (M(\varphi_{i}) \nabla \mu_{T}) + U_i(\bvarphi,\bm\sigma), \quad 2 \leq i \leq L,
\end{align*}
which resemble the system of equations studied in \cite{ChenLowen,CWSL,FJCWLC,WLFC,YLLL}.  Note in particular that only $\mu_{T}$ is needed to drive the evolution of $\varphi_{i}$, $2 \leq i \leq L$.  However, the mathematical treatment of these types of models is difficult due to the fact that the equation for $\varphi_{i}$ is now a transport equation with a high order source term $\div (M(\varphi_{i}) \nabla \mu_{T})$, and the natural energy identity of the model does not appear to yield useful a priori estimates for $\varphi_{i}$.  In the case that the mobility $M$ is a constant, the existence of a weak solution for the model of \cite{CWSL} has been studied by Dai et al.\ in \cite{Dai}.

The specific forms of the source terms $\bm{U}(\bvarphi,\bm\sigma)$ and $\bm{S}(\bvarphi,\bm\sigma)$ will depend on the specific situation we want to model.  In our numerical investigations, we will primarily focus on a three-component model consisting of host cells $(\varphi_{1})$, proliferating tumour cells $(\varphi_{2})$ and necrotic cells $(\varphi_{3})$ in the presence of a quasi-static nutrient $(\sigma)$, i.e., $L=3$ and $M=1$.  Of biological relevance are the following choices:
\begin{subequations}
\begin{alignat}{3}
S(\bm{\varphi}, \sigma) & = - \mathcal{C} \varphi_{2} \sigma \label{Intro:source:sigma}, \\
\bm{U}_{A}(\bm{\varphi}, \sigma) & = (0, \varphi_{2}(\mathcal{P} \sigma - \mathcal{A}), \mathcal{A} \varphi_{2} - D_{N} \varphi_{3})^{\top}, \label{Intro:source:1} \\
\bm{U}_{B}(\bm{\varphi}, \sigma) & = (- \varphi_{2} \mathcal{P} \sigma, \varphi_{2}(\mathcal{P} \sigma - \mathcal{A}), \mathcal{A} \varphi_{2} - D_{N} \varphi_{3})^{\top}, \label{Intro:source:1:alt} \\
\bm{U}_{C}(\bm{\varphi}, \sigma) & = (0, \eps^{-1} \varphi_{2}^{2}(1-\varphi_{2})^{2} (\mathcal{P} \sigma - \mathcal{A}), \eps^{-1} \varphi_{3}^{2}(1-\varphi_{3})^{2}(\mathcal{A} - D_{N}))^{\top}. \label{Intro:source:2}
\end{alignat}
\end{subequations}
The source term \eqref{Intro:source:sigma} models the consumption of nutrients by the proliferating cells at a constant rate $\mathcal{C} > 0$.  The choice \eqref{Intro:source:1} models the proliferation of tumour cells at a constant rate $\mathcal{P} > 0$ by consuming the nutrient, the apoptosis of the tumour cells at a constant rate $\mathcal{A} \geq 0$, which can be considered as a source term for the necrotic cells, and we assume that the necrotic cells degrade at constant rate $D_{N}$.  Meanwhile, in \eqref{Intro:source:1:alt}, any mass gain for the proliferating tumour equals the mass loss by the host cells, and vice versa for the necrotic and proliferating cells.  In \eqref{Intro:source:2}, the functions $\varphi_{2}^{2}(1-\varphi_{2})^{2}$ and $\varphi_{3}^{2}(1-\varphi_{3})^{2}$ are zero except near the vicinity of the interfacial layers.  The scaling with $\eps^{-1}$ is chosen similarly as in \cite{Kampmann}, which allows the source terms to influence the evolution of the interfaces, see Section \ref{sec:SIM:3component} below for more details.

In \eqref{Intro:multiphase}, the parameter $\eps$ is related to the thickness of the interfacial layers, and hence it is natural to ask if a sharp interface model will emerge in the limit $\eps \to 0$.  Due to the multi-component nature of \eqref{Intro:multiphase}, the sharp interface model consists of equations posed on time-dependent regions $\Omega_{i} = \{ \varphi_{i} = 1\}$ for $1 \leq i \leq L$ and on the free boundaries $\Gamma_{ij} = \pd \Omega_{i} \cap \pd \Omega_{j}$ for $1 \leq i < j \leq L$.  We refer the reader to Section \ref{sec:SIM} below for the multi-component sharp interface limit of \eqref{Intro:multiphase}, which is too complex to state here.  

Instead, we consider the system \eqref{Intro:3component} with a quasi-static nutrient (neglecting the left-hand side of \eqref{Intro:3component:sigma}), $\chi_{\varphi} = \chi_{n} = 0$ (so that $\bm{N}_{,\bm{\varphi}}(\bvarphi,\sigma) = \bm{0}$), $D(\bm{\varphi}, \sigma) = 1$, a mobility tensor $\TT{C}(\bm{\varphi}, \sigma) = (\delta_{ij} - \frac{1}{3})_{i,j=1}^3$, and the source term $S(\bm{\varphi}, \sigma) =  - \mathcal{C} \varphi_{2} \sigma$.  Then, \eqref{Intro:3component} simplifies to
\begin{subequations}\label{Intro:simplify:3compo}
\begin{alignat}{3}
\div \vec{v} & = \unit \cdot \bm{U}(\bm{\varphi}, \sigma), \quad \vec{v} = -K \nabla p + K(\nabla \bm{\varphi})^{\top} \bm{\mu}, \\
\pd_{t} \varphi_{1} + \div (\varphi_{1} \vec{v}) & = \Laplace y + U_1(\bm{\varphi}, \sigma), \\
\pd_{t} \varphi_{2} + \div (\varphi_{2} \vec{v}) &  = \Laplace z + U_2(\bm{\varphi}, \sigma), \\
\pd_{t}\varphi_{3} + \div(\varphi_{3} \vec{v}) & = - \Laplace (y+z) + U_3(\bm{\varphi}, \sigma), \\
\mu_{k} & = - \beta \eps \Laplace \varphi_{k} + \beta \eps^{-1} \Psi_{,\varphi_{k}}(\bvarphi), \quad k = 1,2,3,\\
0 & = \Laplace \sigma - \mathcal{C} \varphi_{2} \sigma,
\end{alignat}
\end{subequations}
where $\bm{U} = (U_{1}, U_{2}, U_{3})^{\top}$,
\begin{align*}
3y = (\mu_{1} - \mu_{2}) + (\mu_{1} - \mu_{3}), \quad 3z = - (\mu_{1} - \mu_{2}) + (\mu_{2} - \mu_{3}),
\end{align*} 
and we note that $3(y + z) = (\mu_{1} - \mu_{3}) + (\mu_{2} - \mu_{3})$ and hence diffusion is governed by the difference of chemical potentials, see also \cite{BronsardGarckeStoth98}.  Let us denote $\Omega_{H} = \{ \varphi_{1} = 1, \varphi_{2} = \varphi_{3} = 0 \}$, $\Omega_{P} = \{ \varphi_{2} = 1, \varphi_{1} = \varphi_{3} = 0\}$, $\Omega_{N} = \{ \varphi_{3} = 1, \varphi_{1} = \varphi_{2} = 0\}$ as the regions of host cells, proliferating tumour and necrotic cells, respectively, along with interfaces $\Gamma_{PN} = \pd \Omega_{P} \cap \pd \Omega_{N}$ and $\Gamma_{HP} = \pd \Omega_{H} \cap \pd \Omega_{P}$.  Note that it makes no sense for the host cells to share a boundary with the necrotic cells, and thus $\Gamma_{HN} = \emptyset$.  Then, the sharp interface limit of \eqref{Intro:simplify:3compo} reads as (see Section \ref{sec:SIM:3component} for a derivation)
\begin{subequations}\label{Intro:SIM}
\begin{alignat}{3}
\Laplace \sigma  = \begin{cases}
0 & \text{ in } \Omega_{H} \cup \Omega_{N}, \\
\mathcal{C} \sigma & \text{ in } \Omega_{P},
\end{cases} \quad \begin{cases}
-\Laplace y = U_1(\bm{\varphi}, \sigma) - (\unit \cdot \bm{U}(\bm{\varphi}, \sigma)) \varphi_{1} , \\
-\Laplace z = U_2(\bm{\varphi}, \sigma) - (\unit \cdot \bm{U}(\bm{\varphi}, \sigma)) \varphi_{2} , \\
-K \Laplace p = \unit \cdot \bm{U}(\bm{\varphi}, \sigma) \end{cases} & \text{ in } \Omega_{H} \cup \Omega_{P} \cup \Omega_{N}, \label{Intro:SIM:bulk} \\
[y] = [z] = [\sigma] = [\nabla \sigma] \cdot \vec{\nu} = [\nabla p] \cdot \vec{\nu} = 0 & \text{ on } \Gamma_{PN} \cup \Gamma_{HP}, \\
\jump{p}{N}{P} = \beta \gamma_{PN} \kappa = 2 y - z, \, -\mathcal{V} + K \nabla p \cdot \vec{\nu} = \jump{\nabla z}{N}{P} \cdot \vec{\nu}, \, 0 = \jump{\nabla y}{N}{P} \cdot \vec{\nu} & \text{ on } \Gamma_{PN}, \\
\jump{p}{P}{H} = \beta \gamma_{HP} \kappa = y - z , \, - \mathcal{V} + K \nabla p \cdot \vec{\nu} = \jump{\nabla y}{P}{H} \cdot \vec{\nu} = - \jump{\nabla z}{P}{H} \cdot \vec{\nu} & \text{ on } \Gamma_{HP}.
\end{alignat}
\end{subequations}
In the above $\vec{\nu}$ denotes the unit normal on $\Gamma_{PN}$ pointing into $\Omega_{P}$ or the unit normal on $\Gamma_{HP}$ pointing into $\Omega_{H}$, $\kappa$ is the mean curvature, $\gamma_{PN}$ and $\gamma_{HP}$ are positive constants related the potential $\Psi$, $\mathcal{V}$ denotes the normal velocity of $\Gamma_{PN}$ or $\Gamma_{HP}$, and $[\cdot]$ denotes the jump across the interfaces.  Let us point out that for the choice \eqref{Intro:source:1} of $\bm{U}(\bvarphi,\sigma)$, equation \eqref{Intro:SIM:bulk} becomes
\begin{align*}
- \Laplace y = 0 \text{ in } \Omega_{H} \cup \Omega_{P} \cup \Omega_{N}, \quad - \Laplace z = \begin{cases}
0 \text{ in } \Omega_{H} \cup \Omega_{N}, \\
-\mathcal{A} \text{ in } \Omega_{P},
\end{cases} \, -K \Laplace p = \begin{cases}
0 \text{ in } \Omega_{H}, \\
\mathcal{P} \sigma \text{ in } \Omega_{P}, \\
-D_{N} \text{ in } \Omega_{N},
\end{cases}
\end{align*}
and for the choice \eqref{Intro:source:1:alt} of $\bm{U}(\bm{\varphi}, \sigma)$, equation \eqref{Intro:SIM:bulk} becomes
\begin{align*}
-\Laplace y = \begin{cases}
-\mathcal{P} \sigma \text{ in } \Omega_{P},  \\
0 \text{ in } \Omega_{H} \cup \Omega_{N},
\end{cases} \, -\Laplace z = \begin{cases}
\mathcal{P} \sigma -\mathcal{A} \text{ in } \Omega_{P}, \\
0 \text{ in } \Omega_{H} \cup \Omega_{H}, 
\end{cases} \, -K \Laplace p = \begin{cases}
0 \text{ in } \Omega_{H} \cup \Omega_{P}, \\
-D_{N} \text{ in } \Omega_{N}.
\end{cases}
\end{align*}
Note that the overall gain or loss in mass is reflected in the equation for $p$, compare \cite[\S 4.6]{book:CristiniLowengrub}.

In contrast, multi-component models obtained from a degenerate interfacial energy such as \eqref{deg:GL} have simpler sharp interface limits.  Due to the fact that \eqref{deg:GL} is a function only of $\varphi_{T}$, the asymptotic analysis leads to a sharp interface limit which is defined on two time-dependent regions $\Omega_{T} = \{\varphi_{T} = 1\}$ (tumour) and $\Omega_{H} = \Omega \setminus \Omega_{T}$ (host), and one free boundary $\Gamma = \pd \Omega_{T}$.  In particular, differentiation between the different types of tumour cells is based on the local density of nutrients \cite{MacklinLowengurb06, MacklinLowengrub07, MacklinLowengurb08, ZhengWC}, unlike in \eqref{Intro:SIM} where an evolution law for the interface $\Gamma_{PN}$ between the proliferating and necrotic cells is stated.  We refer the reader also to Section \ref{sec:SIM:degGL} below for the sharp interface limit of a model with degenerate interfacial energy.

Let us now give a non-exhaustive comparison between the multi-component diffuse interface models in the literature and the model \eqref{Intro:multiphase} we propose in this work.

\paragraph{Interfacial energy/cellular adhesion.}
In \cite{ChenLowen, CWSL, FJCWLC, LimaStochastic, LimaM3AS, WLFC, YLLL}, it is assumed that the different types of tumour cells prefer to adhere to one another instead of the host cells, and thus the degenerate interfacial energy density \eqref{deg:GL} is considered.  This is in contrast to Oden et al.\ \cite{OHP} and our present work, where the adhesive properties of different cell types are distinct and the total energy \eqref{Intro:energy} is considered.  Furthermore, we point out that the model of Xu et al.\ \cite{Xu} can be seen as a two-phase model (tumour and host cells), which uses an interfacial energy similar to \eqref{deg:GL}.  But they use a non-conserved phase field equation of Allen--Cahn type, rather than a Cahn--Hilliard equation, to describe the tumour evolution.

\paragraph{Mixture velocity.} In \cite{ChenLowen,CWSL,FJCWLC,OHP,WLFC,YLLL} a mass-averaged velocity is used instead of the volume-averaged velocity considered in our present approach and also in \cite{PreziosiTosin}.  Meanwhile, in \cite{LimaStochastic,LimaM3AS} the velocities of the cell components are assumed to be negligible.

\paragraph{Source terms.}  Aside from mitosis proportional to the local density of nutrients, and constant apoptosis for the tumour cells, certain sink terms for one cell type become source terms for another, for example the term $\mathcal{A} \varphi_{2}$ in \eqref{Intro:source:1}.  It is commonly assumed that the host cells are homeostatic \cite{ChenLowen,CWSL,FJCWLC,YLLL,WLFC}, and so the source term for the host cells is zero.  In \cite{LimaStochastic,LimaM3AS}, where quiescent cells are also considered, a two-sided exchange between the proliferating cells and the quiescent cells, and a one-sided exchange from quiescent cells to necrotic cells based on local nutrient concentration are included.   However, to the best of our knowledge, source terms of the form \eqref{Intro:source:2} have not yet been considered in the multi-component setting.

\paragraph{Sharp interface limit.}  Out of the aforementioned references, only Wise et al.\ \cite{WLFC} state a sharp interface limit for a multi-component diffuse interface model with degenerate interfacial energy \eqref{deg:GL}. 

\bigskip

The remainder of this paper is organised as follows:  In Section \ref{sec:derive} we derive the diffuse interface model \eqref{Intro:multiphase} from thermodynamic principles.  In Section \ref{sec:SIM} we perform a formal asymptotic analysis to derive the sharp interface limit.  In Section \ref{sec:numerics} we present some numerical simulations for the three-component tumour model derived in this paper.

\section{Model derivation}\label{sec:derive}
Let us consider a mixture consisting of $L \geq 2$ cell components in an open, bounded domain $\Omega \subset \R^{d}$, $d = 1,2,3$.  Moreover, we allow for the presence of $M \geq 1$ chemical species in $\Omega$.  Let $\rho_{i}$, $i = 1, \dots, L$, denote the actual mass of the matter of the $i$th component per volume in the mixture, and let $\bar{\rho}_{i}, i = 1, \dots, L$, be the mass density of a pure component $i$.  Then the sum $\rho = \sum_{i=1}^{L} \rho_{i}$ denotes the mixture density (which is not necessarily constant), and we define the volume fraction of component $i$ as
\begin{align}\label{defn:volumefrac}
\varphi_{i} = \frac{\rho_{i}}{\bar{\rho}_{i}}.
\end{align}
We expect that physically, $\rho_{i} \in [0, \bar{\rho}_{i}]$ and thus $\varphi_{i} \in [0,1]$.  Furthermore we allow for mass exchange between the components, but there is no external volume compartment besides the $L$ components, i.e.,
\begin{align}
\label{varphisum1}
\sum_{i=1}^{L} \varphi_{i} = 1.
\end{align}
For the mixture velocity we consider the volume-averaged velocity
\begin{align}
\label{defn:volavervelo}
\vec{v} = \sum_{i=1}^{L} \varphi_{i} \vec{v}_{\varphi_{i}},
\end{align}
where $\vec{v}_{\varphi_{i}}$ is the individual velocity of component $i$, and we denote the density of the $j$th chemical species as $\sigma_{j}$, $j = 1, \dots, M$, where each chemical species is transported by the volume-averaged mixture velocity and a flux $\vec {J}_{\sigma_{j}}$, $j = 1, \dots, M$.

\subsection{Balance laws}
The balance law for the mass of each component reads as
\begin{align}\label{proto:individualmass}
\pd_{t} \rho_{i} + \div (\rho_{i} \vec{v}_{\varphi_{i}}) = \mathcal{U}_{i}, \quad i = 1, \dots, L,
\end{align}
where $\mathcal{U}_{i}$ denotes a source/sink term for the $i$th component.  Using \eqref{defn:volumefrac} we have
\begin{align}\label{proto:varphiiequ}
\pd_{t} \varphi_{i} + \div (\varphi_{i} \vec{v}_{\varphi_{i}}) = \frac{\mathcal{U}_{i}}{\overline{\rho}_{i}}, \quad i = 1, \dots, L.
\end{align}
Upon adding and using \eqref{defn:volavervelo} and \eqref{varphisum1}, we obtain an equation for the volume-averaged velocity:
\begin{align}\label{proto:div}
\div \vec{v} = \sum_{i=1}^{L} \div (\varphi_{i} \vec{v}_{\varphi_{i}}) = \sum_{i=1}^{L} \frac{\mathcal{U}_{i}}{\overline{\rho}_{i}}.
\end{align}
On recalling \eqref{defn:volavervelo}, we introduce the fluxes
\begin{align}
\label{defn:fluxesJ}
\vec{J}_{\varphi_{i}} = \rho_{i}( \vec{v}_{\varphi_{i}} - \vec{v}), \quad i = 1, \dots, L, \quad \vec{\mathcal{J}} = \sum_{i=1}^{L} \vec{J}_{\varphi_{i}},
\end{align}
so that from \eqref{proto:varphiiequ} we obtain
\begin{align}
\label{proto:varphiiequ:volavegedvelo}
\pd_{t} \varphi_{i} + \frac{1}{\overline{\rho}_{i}} \div \vec{J}_{\varphi_{i}} + \div (\varphi_{i} \vec{v}) = \frac{\mathcal{U}_{i}}{\overline{\rho}_{i}}, \quad i = 1, \dots, L.
\end{align}
Rewriting the mass balance \eqref{proto:individualmass} with $\vec{J}_{\varphi_{i}}$ and upon summing we obtain the following equation for the mixture density:
\begin{align}
\label{proto:mixturedensity}
\pd_{t} \rho + \div (\vec{\mathcal{J}} + \rho \vec{v}) = \sum_{i=1}^{L} \mathcal{U}_{i}.
\end{align}
Moreover, by summing \eqref{defn:fluxesJ} we obtain the requirement
\begin{align}
\label{Jioverlinerhoisum0}
\sum_{i=1}^{L} \frac{1}{\overline{\rho}_{i}} \vec{J}_{\varphi_{i}} = \sum_{i=1}^{L} \varphi_{i} (\vec{v}_{\varphi_{i}} - \vec{v}) = \vec{v} - \vec{v} = \vec{0}.
\end{align} 
For $j = 1, \dots, M$, we postulate the following balance law for the $j$th chemical species
\begin{align}
\label{proto:concentrationj}
\pd_{t} \sigma_{j} + \div (\sigma_{j} \vec{v}) + \div \vec J_{\sigma_{j}} = S_{j},
\end{align}
where $S_{j}$ denotes a source/sink term for the $j$th chemical species, $\sigma_{j} \vec{v}$ models the transport by the volume-averaged velocity and $\vec J_{\sigma_{j}}$ accounts for other transport mechanisms.  It is convenient to introduce the vector form of the balance laws \eqref{proto:varphiiequ:volavegedvelo} and \eqref{proto:concentrationj}.  Let 
\begin{equation}\label{defn:vectorform}
\begin{alignedat}{2}
\bm{\varphi} & = (\varphi_{1}, \dots, \varphi_{L})^{\top} \in \R^{L}, \quad  \bm{U} && = \left ( \overline{\rho}_{1}^{-1} \mathcal{U}_{1}, \dots,\overline{\rho}_{L}^{-1} \mathcal{U}_{L} \right )^{\top} \in \R^{L}, \\
 \bm{\sigma} & = (\sigma_{1}, \dots, \sigma_{M})^{\top} \in \R^{M}, \quad \bm{S} && = (S_{1}, \dots, S_{M})^{\top} \in \R^{M},
\end{alignedat}
\end{equation}
and
\begin{align}\label{fluxtensor}
\TT K_{\varphi}^{\top} = \left ( \overline{\rho}_{1}^{-1} \vec{J}_{\varphi_{1}}, \dots, \overline{\rho}_{L}^{-1} \vec{J}_{\varphi_{L}} \right ) \in \R^{d \times L} , \quad \TT K_{\sigma}^{\top} = (\vec J_{\sigma_{1}}, \dots, \vec J_{\sigma_{M}}) \in \R^{d \times M},
\end{align}
i.e., the $l$th row of $\TT K_{\varphi}$ is the flux $\overline{\rho}_{l}^{-1} \vec J_{\varphi_{l}} \in \R^{d}$ and the $l$th row of $\TT K_{\sigma}$ is the flux $\vec J_{\sigma_{l}} \in \R^{d}$.  We recall that the divergence applied to a second order tensor $\TT{A} \in \R^{k \times l}$ results in a vector in $\R^{k}$ whose $i$th component is the divergence of $(A_{ij})_{j=1}^{l}$, that is, $(\div \TT{A})_{i} = \sum_{j=1}^{l} \pd_{x_{j}} A_{ij}$.  Then, \eqref{proto:varphiiequ:volavegedvelo}, \eqref{proto:concentrationj} and \eqref{proto:div} become
\begin{align}
\label{proto:vector:varphi}
& \pd_{t} \bm{\varphi} + \div (\bm{\varphi} \otimes \vec{v}) + \div \TT K_{\varphi}  = \bm{U}, \\
\label{proto:vector:sigma}
& \pd_{t} \bm{\sigma} + \div ( \bm{\sigma} \otimes \vec{v}) + \div \TT K_{\sigma} = \bm{S}, \quad \div \vec{v} = \unit \cdot \bm{U},
\end{align}
respectively, where $\unit = (1, \dots, 1)^{\top} \in \R^{L}$.

\subsection{Energy inequality} \label{sec:22}
For $L \in \N$, $L\geq2$, we define
\begin{align}\label{SetsGibbs}
\HGibbs = \left \{  \bm{\phi} = (\phi_{1}, \dots, \phi_{L})^{\top} \in \R^{L} : \sum_{i=1}^{L} \phi_{i} = 1 \right \}, \quad \Gibbs = \left \{ \bm{\phi} \in \HGibbs : \phi_{i} \geq 0 \; \forall i \right \}.
\end{align}
The latter is also known as the Gibbs simplex.  The corresponding tangent space $T_{\bm{p}}\HGibbs$ can be identified as the space
\begin{align}\label{tangentspaceTG}
T_{\bm{p}} \HGibbs \cong \TangGibbs = \left \{ \bm{\psi} \in \R^{L} : \sum_{i=1}^{L} \psi_{i} = 0 \right \}.
\end{align}
We postulate a general free energy of Ginzburg--Landau form, i.e.,
\begin{align}\label{postulate:energy}
\mathcal{E}(\bm{\varphi}, \bm{\sigma})
= \int_\Omega e(\bm{\varphi}, \nabla \bm{\varphi}, \bm\sigma) \dx
= \int_\Omega A \Psi(\bm{\varphi}) + B a(\bm{\varphi}, \nabla \bm{\varphi})
+ N(\bm{\varphi}, \bm{\sigma}) \dx 
,
\end{align}
where $\bm{\varphi} = (\varphi_{1}, \dots, \varphi_{L})^{\top} \in \mathrm{G}$, $\bm{\sigma} = (\sigma_{1}, \dots, \sigma_{M})^{\top}$ and $\nabla \bm{\varphi} = ( \pd_{x_{k}} \varphi_{i} )_{1 \leq i \leq N, 1 \leq k \leq d}$. 
Here $A, B > 0$ are constants, $a: \Gibbs \times (\TangGibbs)^{d} \to \R$ is a smooth gradient energy density and $\Psi: \Gibbs \to \R_{\geq 0}$ is a smooth multi-well potential with exactly $L$ equal minima at the points $\bm{e}_{l}$, $l = 1, \dots, L$, where $\bm{e}_{l} = ( \delta_{lm} )_{m=1}^{L}$ is the $l$th unit vector in $\R^{L}$.  In particular, the minima of $\Psi$ are the corners of the Gibbs simplex $\Gibbs$. 
The first two terms in the integral in \eqref{postulate:energy} account for interfacial energy and unmixing tendencies, and the term $N(\bm{\varphi}, \bm{\sigma})$ accounts for the chemical energy of the species and any energy contributions resulting from the interactions between the cells and the chemical species.

Recalling the vector $\unit = (1, \dots, 1)^{\top} \in \R^{L}$, we now introduce the projection operator $\Proj{}$ to the tangent space $\TangGibbs$ as follows:
\begin{align}\label{projectionop}
\Proj{\bm{f}} = \bm{f} - \frac{1}{L} (\unit \cdot \bm{f}) \unit
\end{align}
for a vector $\bm{f} \in \R^{L}$.  For a second order tensor $\TT{A} \in \R^{L \times d}$ we define the $(i,j)$th component  of its projection to be
\begin{align*}
(\Proj{\TT{A}})_{ij} = A_{ij} - \frac{1}{L} \sum_{k=1}^{L} A_{kj}.
\end{align*}
We now derive a diffuse interface model based on a dissipation inequality for the balance laws in \eqref{proto:vector:varphi} and \eqref{proto:vector:sigma}.  We point out that balance laws with source terms have been used similarly by Gurtin \cite{Gurtin89,Gurtin96} and Podio-Guidugli \cite{Podio06} to derive phase field and Cahn--Hilliard type equations.  These authors used the second law of thermodynamics which in an isothermal situation is formulated as a free energy inequality.  

The second law of thermodynamics in the isothermal situation requires that for all volumes $V(t) \subset \Omega$, which are transported with the fluid velocity, the following inequality has to hold (see \cite{Gurtin89, Gurtin96,  Podio06} and \cite[Chapter 62]{GurtinFriedAnand}) 
\begin{align*}
\frac{\dd}{\dt} \int_{V(t)} e(\bm{\varphi}, \nabla \bm{\varphi},\bm{\sigma}) \dx \leq - \int_{\pd V(t)} \vec J_{e} \cdot \vec\nu \dHaus + \int_{V(t)} \left ( \bm{c}_{\varphi} + c_{v} \unit \right ) \cdot \bm{U} + \bm{c}_{\sigma} \cdot \bm{S} \dx,
\end{align*}
where $\dHaus$ denotes integration with respect to the $d-1$ dimensional Hausdorff measure, $\vec\nu$ is the outer unit normal to $\pd V(t)$, $\vec J_{e}$ is an energy flux yet to be specified, and we have postulated that the source terms $\bm{U}$ and $\bm{S}$ carry with them a supply of energy described by
\begin{align}
\int_{V(t)} \left ( \bm{c}_{\varphi} + c_{v} \unit \right ) \cdot \bm{U} + \bm{c}_{\sigma} \cdot \bm{S} \dx,
\end{align}
for some $\bm{c}_{\varphi} \in \TangGibbs$, $c_{v} \in \R$ and $\bm{c}_{\sigma} \in \R^{M}$ yet to be determined.  

Applying the transport theorem and the divergence theorem, we obtain the following local form
\begin{align}
\pd_{t} e + \div (e \vec{v}) + \div \vec J_{e} - \left ( \bm{c}_{\varphi} + c_{v} \unit \right ) \cdot \bm{U} - \bm{c}_{\sigma} \cdot \bm{S} \leq 0.
\end{align}
We now use the Lagrange multiplier method of Liu and M\"{u}ller (\cite[Section 2.2]{AGG} and \cite[Chapter 7]{book:Liu}).  Let $\bm{\lambda}_{\varphi} \in \TangGibbs$, $\bm{\lambda}_{\sigma} \in \R^{M}$ and $\lambda_{v} \in \R$ denote the Lagrange multipliers for the equations in \eqref{proto:vector:varphi} and \eqref{proto:vector:sigma}, respectively.  Then, we require that the following inequality holds for arbitrary $\bm{\varphi} \in \Gibbs$, $\pd_{t} \bm{\varphi} \in \TangGibbs$, $\nabla \bm{\varphi} \in (\TangGibbs)^{d}$, $\bm{\sigma}, \pd_{t} \bm{\sigma} \in \R^{M}$, $\nabla \bm{\sigma} \in \R^{M \times d}$, $\vec{v} \in \R^{d}$, $\bm{U} \in \R^{L}$, and $\bm{S} \in \R^{M}$:
\begin{equation}\label{energydissipationineq}
\begin{aligned}
-\mathcal{D} & = \pd_{t} e +  \vec{v} \cdot \nabla e + e \, \div \vec{v} + \div \vec J_{e} - \left ( \bm{c}_{\varphi} + c_{v} \unit \right ) \cdot \bm{U} - \bm{c}_{\sigma} \cdot \bm{S} \\
& \quad - \bm{\lambda}_{\varphi} \cdot ( \pd_{t} \bm{\varphi} + (\nabla \bm{\varphi}) \vec{v} +  (\div \vec{v}) \bm{\varphi} + \div \TT K_{\varphi} - \bm{U}) \\
& \quad - \bm{\lambda}_{\sigma} \cdot ( \pd_{t} \bm{\sigma} + (\nabla \bm{\sigma}) \vec{v} +  (\div \vec{v}) \bm{\sigma} + \div \TT K_{\sigma} - \bm{S}) \\
& \quad - \lambda_{v} (\div \vec{v} - \unit \cdot \bm{U}) \leq 0.
\end{aligned}
\end{equation}
Using the identities
\begin{align*}
[(\nabla \bm{\varphi}) \vec{v}]_{i} =  \sum_{k=1}^{d} \pd_{x_{k}} \varphi_{i} v_{k}, \; \md \bm{\varphi} = \pd_{t} \bm{\varphi} + (\nabla \bm{\varphi}) \vec{v}, \; \bm{\lambda} \cdot \div (\bm{\varphi} \otimes \vec{v}) = \bm{\lambda} \cdot (\nabla \bm{\varphi}) \vec{v} + (\bm{\lambda} \cdot \bm{\varphi}) \div \vec{v},
\end{align*}
and the product rule 
\begin{align}\label{tensorproductrule}
\div ( \TT K^{\top} \bm{\lambda} ) = \TT K : \nabla \bm{\lambda} + (\div \TT K) \cdot \bm{\lambda},
\end{align}
where for two tensors $\TT{A}$ and $\TT{B}$, the product $\TT{A} : \TT{B}$ is defined as $\TT{A} : \TT{B} = \tr{\TT{A}^{\top} \TT{B}}$, we arrive at
\begin{equation}\label{dissipation1}
\begin{aligned}
-\mathcal{D} & = \div \left ( \vec J_{e} - \TT K_{\varphi}^{\top} \bm{\lambda}_{\varphi} - \TT K_{\sigma}^{\top} \bm{\lambda}_{\sigma} \right ) + ( B a_{,\bm{\varphi}} + A \Psi_{,\bm{\varphi}} + \bm{N}_{,\bm{\varphi}} - \bm{\lambda}_{\varphi} )\cdot \md \bm{\varphi} + (\bm{N}_{,\bm{\sigma}} - \bm{\lambda}_{\sigma}) \cdot \md \bm{\sigma}  \\
& \quad + \bm{U} \cdot (\bm{\lambda}_{\varphi} - \bm{c}_{\varphi} + (\lambda_{v} - c_{v}) \unit ) + \bm{S} \cdot (\bm{\lambda}_{\sigma} - \bm{c}_{\sigma}) + \TT K_{\varphi} : \nabla \bm{\lambda}_{\varphi} + \TT K_{\sigma} : \nabla \bm{\lambda}_{\sigma}\\
& \quad + B \sum_{i=1}^{L} \sum_{j,k=1}^{d} (a_{,\pd_{k} \varphi_{i}} ) [ \pd_{t} \pd_{x_{k}} \varphi_{i} + v_{j} \pd_{x_{j}} \pd_{x_{k}} \varphi_{i}] + (\div \vec{v}) (e  - \bm{\lambda}_{\varphi} \cdot \bm{\varphi} - \bm{\lambda}_{\sigma} \cdot \bm{\sigma} - \lambda_{v}),
\end{aligned}
\end{equation}
where
\begin{align*}
\bm{N}_{,\bm{\varphi}} = \left ( \frac{\pd N}{\pd \varphi_{1}}, \dots, \frac{\pd N}{\pd \varphi_{L}} \right )^{\top} \in \R^{L}, \quad \bm{N}_{,\bm{\sigma}} = \left ( \frac{\pd N}{\pd \sigma_{1}}, \dots, \frac{\pd N}{\pd \sigma_{M}} \right )^{\top} \in \R^{M}.
\end{align*}
We can rewrite the term involving $(a_{,\nabla \bm{\varphi}})_{ik} = \frac{\pd a}{\pd (\pd_{k} \varphi_{i})} $ as follows (using the notation $\md \varphi_{i} = \pd_{t}\varphi_{i} + \vec{v} \cdot \nabla \varphi_{i}$):
\begin{align*}
& \sum_{i=1}^{L} \sum_{k=1}^{d} (a_{,\pd_{k} \varphi_{i}} ) \left [ \pd_{t} \pd_{x_{k}} \varphi_{i} + \sum_{j=1}^{d} v_{j} \pd_{x_{j}} \pd_{x_{k}} \varphi_{i} \right ]  \\
& \; = \sum_{k=1}^{d} \sum_{i=1}^{L} \left (  \pd_{x_{k}} \left ( a_{,\pd_{k} \varphi_{i}} \pd_{t}\varphi_{i} \right ) - \pd_{x_{k}} a_{,\pd_{k} \varphi_{i}}  \pd_{t}\varphi_{i} + \sum_{j=1}^{d} \left ( v_{j} \pd_{x_{k}} \left ( a_{,\pd_{x} \varphi_{i}} \pd_{x_{j}} \varphi_{i} \right ) - v_{j} \pd_{x_{j}} \varphi_{i} \pd_{x_{k}} a_{,\pd_{k} \varphi_{i}} \right ) \right ) \\
 & \; = \div ( (a_{, \nabla \bm{\varphi}})^{\top} \pd_{t} \bm{\varphi} ) - \div (a_{,\nabla \bm{\varphi}}) \cdot \md \bm{\varphi} +  \vec{v} \cdot \div ( (\nabla \bm{\varphi})^{\top} (a_{,\nabla \bm{\varphi}})) .
\end{align*}
Applying the product rule on the term involving $\div \vec{v}$ we get
\begin{align*}
(\div \vec{v})(e - \bm{\lambda}_{\varphi} \cdot \bm{\varphi} - \bm{\lambda}_{\sigma} \cdot \bm{\sigma} - \lambda_{v}) & = \div ( (e - \bm{\lambda}_{\varphi} \cdot \bm{\varphi} - \bm{\lambda}_{\sigma} \cdot \bm{\sigma} - \lambda_{v}) \vec{v}) \\
& \quad - \vec{v} \cdot \nabla (e - \bm{\lambda}_{\varphi} \cdot \bm{\varphi} - \bm{\lambda}_{\sigma} \cdot \bm{\sigma} - \lambda_{v}).
\end{align*}
Thus, substituting the above into the expression \eqref{dissipation1} we obtain
\begin{equation}\label{dissipation2}
\begin{aligned}
-\mathcal{D} & = \div \left ( \vec J_{e} - \TT K_{\varphi}^{\top} \bm{\lambda}_{\varphi} - \TT K_{\sigma}^{\top} \bm{\lambda}_{\sigma} + B(a_{,\nabla \bm{\varphi}})^{\top} \pd_{t} \bm{\varphi} + (e - \bm{\lambda}_{\varphi} \cdot \bm{\varphi} - \bm{\lambda}_{\sigma} \cdot \bm{\sigma} - \lambda_{v}) \vec{v} \right ) \\
& \quad - \vec{v} \cdot \left [  \nabla (e - \bm{\lambda}_{\varphi} \cdot \bm{\varphi} - \bm{\lambda}_{\sigma} \cdot \bm{\sigma} - \lambda_{v}) - B \div (( \nabla \bm{\varphi})^{\top} (a_{,\nabla \bm{\varphi}})) \right ] \\
& \quad + ( B a_{,\bm{\varphi}} - B \div (a_{,\nabla \bm{\varphi}}) + A \Psi_{,\bm{\varphi}} + \bm{N}_{,\bm{\varphi}} - \bm{\lambda}_{\varphi} )\cdot \md \bm{\varphi} + (\bm{N}_{,\bm{\sigma}} - \bm{\lambda}_{\sigma}) \cdot \md \bm{\sigma}  \\
& \quad + \bm{U} \cdot (\bm{\lambda}_{\varphi} - \bm{c}_{\varphi}+  (\lambda_{v} - c_{v}) \unit ) + \bm{S} \cdot (\bm{\lambda}_{\sigma} - \bm{c}_{\sigma}) + \TT K_{\varphi} : \nabla \bm{\lambda}_{\varphi} + \TT K_{\sigma} : \nabla \bm{\lambda}_{\sigma}.
\end{aligned}
\end{equation}

\subsection{Constitutive assumptions and the general model}
We define the vector of chemical potentials $\bm{\mu}$ to be
\begin{align}
\label{chempotential}
\bm{\mu} = B a_{,\bm{\varphi}}(\bvarphi,\nabla\bvarphi) - B \div (a_{,\nabla \bm{\varphi}}(\bvarphi,\nabla\bvarphi)) + A \Psi_{,\bm{\varphi}}(\bvarphi) + \bm{N}_{,\bm{\varphi}}(\bvarphi,\bm\sigma),
\end{align}
and by the definition \eqref{projectionop} of the projection operator $\Proj{}$, we have
\begin{align*}
(\bm{\mu} - \bm{\lambda}_{\varphi}) \cdot \md \bm{\varphi} & = \Proj{(\bm{\mu} - \bm{\lambda}_{\varphi})} \cdot \md \bm{\varphi} + \frac{1}{L} ((\bm{\mu} - \bm{\lambda}_{\varphi}) \cdot \unit) \unit \cdot \md \bm{\varphi} =  \Proj{(\bm{\mu} - \bm{\lambda}_{\varphi})} \cdot \md \bm{\varphi} 
\end{align*}
as $\md \bm{\varphi} \in \TangGibbs$ and $\md \bm{\varphi} \cdot \unit = 0$.  Furthermore, from \eqref{Jioverlinerhoisum0} we find that $\TT{K}_{\varphi} : \nabla \bm{\lambda}_{\varphi} = \TT{K}_{\varphi} : \nabla (\Proj{\bm{\lambda}_{\varphi}})$, and so \eqref{dissipation2} can be simplified to
\begin{equation}\label{dissipation3}
\begin{aligned}
-\mathcal{D} & = \div \left ( \vec J_{e} - \TT K_{\varphi}^{\top} \bm{\lambda}_{\varphi} - \TT K_{\sigma}^{\top} \bm{\lambda}_{\sigma} + B(a_{,\nabla \bm{\varphi}})^{\top} \pd_{t} \bm{\varphi} + (e - \bm{\lambda}_{\varphi} \cdot \bm{\varphi} - \bm{\lambda}_{\sigma} \cdot \bm{\sigma} - \lambda_{v}) \vec{v} \right ) \\
&\quad  + \TT K_{\varphi} : \nabla (\Proj{\bm{\lambda}_{\varphi}}) + \TT K_{\sigma} : \nabla \bm{\lambda}_{\sigma} + \Proj{(\bm{\mu} - \bm{\lambda}_{\varphi})} \cdot \md \bm{\varphi} \\
& \quad + (\bm{N}_{,\bm{\sigma}} - \bm{\lambda}_{\sigma}) \cdot \md \bm{\sigma} + \bm{U} \cdot (\bm{\lambda}_{\varphi} - \bm{c}_{\varphi}+  (\lambda_{v} - c_{v}) \unit ) + \bm{S} \cdot (\bm{\lambda}_{\sigma} - \bm{c}_{\sigma}) \\
& \quad - \vec{v} \cdot \left [  \nabla (e - \bm{\lambda}_{\varphi} \cdot \bm{\varphi} - \bm{\lambda}_{\sigma} \cdot \bm{\sigma} - \lambda_{v}) - B \div (( \nabla \bm{\varphi})^{\top} (a_{,\nabla \bm{\varphi}})) \right ].
\end{aligned}
\end{equation}
Based on \eqref{dissipation3} we make the following constitutive assumptions,
\begin{subequations}\label{assump:constitutive}
\begin{align}
\vec J_{e} & =  \TT K_{\varphi}^{\top} \bm{\lambda}_{\varphi}  + \TT K_{\sigma}^{\top} \bm{\lambda}_{\sigma} - B (a_{,\nabla \bm{\varphi}}(\bm{\varphi}, \nabla \bm{\varphi}))^{\top}  \pd_{t} \bm{\varphi} \label{constitutive:Je} \\
& \notag \quad - (e(\bm{\varphi}, \nabla \bm{\varphi}, \bm{\sigma}) - \bm{\lambda}_{\varphi} \cdot \bm{\varphi} - \bm{\lambda}_{\sigma} \cdot \bm{\sigma} - \lambda_{v} ) \vec{v} ,  \\
\bm{c}_{\sigma} & = \bm{\lambda}_{\sigma} = \bm{N}_{,\bm{\sigma}}(\bm{\varphi}, \bm{\sigma}), \quad \bm{c}_{\varphi} = \bm{\lambda}_{\varphi}, \quad \bm{\lambda}_{\varphi} = \Proj{\bm{\mu}}, \quad c_{v} = \lambda_{v}, \label{constitutive:c} \\
\TT K_{\sigma} & = -\TT{D}(\bm{\varphi}, \bm{\sigma}) \nabla \bm{N}_{,\bm{\sigma}}(\bm{\varphi}, \bm{\sigma}), \quad \TT K_{\varphi} = -\TT{C}(\bm{\varphi}, \bm{\sigma}) \nabla (\Proj{\bm{\mu}}),\label{constitutive:fluxes} 
\end{align}
\end{subequations}
where $\TT{C}(\bm{\varphi}, \bm{\sigma}) \in \R^{L \times L}$ and $\TT{D}(\bm{\varphi}, \bm{\sigma}) \in R^{M \times M}$ are non-negative second order mobility tensors such that
\begin{align}
\label{Compatibility:symmetry=0}
\sum_{i=1}^{L} \TT{C}_{ik}(\bm{\varphi}, \bm{\sigma}) = 0 \quad \text{ for all } \bm{\varphi} \in \Gibbs, \; \bm{\sigma} \in \R^{M}, \text{ and } 1 \leq k \leq L.
\end{align}
Here, by a non-negative second order tensor $\TT{A} \in \R^{L \times L}$, we mean that for all $\bm{b} \in \R^{L}$, $\bm{b} \cdot \TT{A} \bm{b} \geq 0$ and $\bm{b} \cdot \TT{A} \bm{b} = 0$ if and only if $\bm{b} = \bm{0}$.  Recalling the definition of $\TT K_{\varphi}$ and $\TT K_{\sigma}$ from \eqref{fluxtensor}, we see that for $1 \leq m \leq d$, the $m$th component of the fluxes $\vec{J}_{\varphi_{i}}$ and $\vec J_{\sigma_{j}}$ are given as
\begin{align*}
\frac{1}{\overline{\rho}_{i}} (\vec{J}_{\varphi_{i}})_{m} = -\sum_{k=1}^{L} \TT{C}_{ik}(\bm{\varphi}, \bm{\sigma}) \pd_{x_{m}} (\Proj{\bm{\mu}})_{k}, \quad (\vec J_{\sigma_{j}})_{m} = -\sum_{k=1}^{M}  \TT{D}_{jk}(\bm{\varphi}, \bm{\sigma}) \pd_{x_{m}} \left (\frac{\pd N}{\pd \sigma_{k}} \right ).
\end{align*}
Then, the constraint \eqref{Jioverlinerhoisum0} requires
\begin{align}\label{mobilityconstraint}
 \sum_{i=1}^{L} \sum_{k=1}^{L}  \TT{C}_{ik}(\bm{\varphi}, \bm{\sigma}) \pd_{x_{m}} (\Proj{\bm{\mu}})_{k} = 0  \quad \forall 1 \leq m \leq d,
\end{align}
which is satisfied when the constitutive assumption \eqref{Compatibility:symmetry=0} is considered.  We point out that one may take $\TT{D}(\bm{\varphi}, \bm{\sigma}) \in \R^{M \times d \times M \times d}$ as a non-negative fourth order mobility tensor, that is, $\TT{D}\, \TT{A} : \TT{A} \geq 0$ and $\TT{D}\, \TT{A} : \TT{A} = 0$ if and only if $\TT{A} = \TT{0}$ for any second order tensors $\TT{A} \in \R^{M \times d}$.  If we also consider $\TT{C}(\bm{\varphi}, \bm{\sigma}) \in \R^{L \times d \times L \times d}$ as a fourth order tensor, then \eqref{Compatibility:symmetry=0} becomes
\begin{align}\label{Compatibility:fourthorder}
\sum_{i=1}^{L} \TT{C}_{imkl}(\bm{\varphi}, \bm{\sigma}) = 0 \quad \text{ for all } \bm{\varphi} \in \Gibbs, \; \bm{\sigma} \in \R^{M}, \text{ and } 1\leq m,l \leq d,  1 \leq k \leq L,
\end{align}
and for $1 \leq m \leq d$, the $m$th component of the fluxes $\vec{J}_{\varphi_{i}}$ and $\vec J_{\sigma_{j}}$ are given as
\begin{align*}
\frac{1}{\overline{\rho}_{i}} (\vec{J}_{\varphi_{i}})_{m} = -\sum_{k=1}^{L} \sum_{l=1}^{d} \TT{C}_{imkl}(\bm{\varphi}, \bm{\sigma}) \pd_{x_{l}} (\Proj{\bm{\mu}})_{k}, \quad (\vec J_{\sigma_{j}})_{m} = -\sum_{k=1}^{M} \sum_{l=1}^{d} \TT{D}_{jmkl}(\bm{\varphi}, \bm{\sigma}) \pd_{x_{l}} \left (\frac{\pd N}{\pd \sigma_{k}} \right ).
\end{align*}
Note that, from \eqref{dissipation3} and the arbitrariness of $\bm{U}$, we require the prefactor $\bm{\lambda}_{\varphi} - \bm{c}_{\varphi} + (\lambda_{v} - c_{v}) \unit$ to vanish.  Since $\bm{\lambda}_{\varphi}, \bm{c}_{\varphi} \in \TangGibbs$ and the vector $(\lambda_{v} - c_{v}) \unit$ is orthogonal to $\TangGibbs$ this leads to the consideration $\bm{\lambda}_{\varphi} = \bm{c}_{\varphi}$ and $\lambda_{v} = c_{v}$ in \eqref{constitutive:c}.  We introduce a pressure-like function $p$ and choose
\begin{align}
\label{constitutive:pressure}
\lambda_{v} = p - B a(\bm{\varphi}, \nabla \bm{\varphi}) - A \Psi(\bm{\varphi}) + e(\bvarphi,\nabla\bvarphi,\bm\sigma) - \Proj{\bm{\mu}} \cdot \bm{\varphi} - \bm{N}_{,\bm{\sigma}}(\bvarphi,\bm\sigma) \cdot \bm{\sigma}, 
\end{align}
and, for a positive constant $K$,
\begin{equation}\label{constutitive:velo:1}
\begin{aligned}
\vec{v} & = K \left ( \nabla (e(\bm{\varphi}, \nabla \bm{\varphi}, \bm{\sigma}) - \Proj{\bm{\mu}} \cdot \bm{\varphi} - \bm{N}_{,\bm{\sigma}}(\bm{\varphi}, \bm{\sigma}) \cdot \bm{\sigma} - \lambda_{v} ) - B \div ((\nabla \bm{\varphi})^{\top} a_{,\nabla \bm{\varphi}}(\bm{\varphi}, \nabla \bm{\varphi})) \right ) \\
& =   K \left ( \nabla (-p + B a(\bm{\varphi}, \nabla \bm{\varphi}) + A \Psi(\bm{\varphi})) - B \div ((\nabla \bm{\varphi})^{\top} a_{,\nabla \bm{\varphi}}(\bm{\varphi}, \nabla \bm{\varphi})) \right ) .
\end{aligned}
\end{equation}
We can further simplify \eqref{constutitive:velo:1} with the identity:
\begin{align*}
\nabla (a(\bm{\varphi}, \nabla \bm{\varphi})) = (\nabla \bm{\varphi})^{\top} a_{,\bm{\varphi}}(\bm{\varphi}, \nabla \bm{\varphi}) + \div ((\nabla \bm{\varphi})^{\top} a_{,\nabla \bm{\varphi}}(\bm{\varphi}, \nabla \bm{\varphi})) - (\nabla \bm{\varphi})^{\top}  \div (a_{,\nabla \bm{\varphi}}(\bm{\varphi}, \nabla \bm{\varphi})),
\end{align*}
and hence, \eqref{constutitive:velo:1} becomes
\begin{align}\label{constitutive:velo:Darcy}
\vec{v} = -K \nabla p + K (\nabla \bm{\varphi})^{\top}  (\bm{\mu} - \bm{N}_{,\bm{\varphi}}(\bvarphi,\bm\sigma)).
\end{align}
Thus, the model equations are 
\begin{subequations}\label{multiphasetumourmodel}
\begin{align}
\div \vec{v} & = \unit \cdot \bm{U}(\bm{\varphi}, \bm{\sigma}), \label{multi:div} \\
\vec{v} & = -K \nabla p + K (\nabla \bm{\varphi})^{\top}  (\bm{\mu} - \bm{N}_{,\bm{\varphi}}(\bm{\varphi}, \bm{\sigma})), \label{multi:darcy} \\
\pd_{t} \bm{\varphi} + \div (\bm{\varphi} \otimes \vec{v}) & = \div ( \TT{C}(\bm{\varphi}, \bm{\sigma}) \nabla (\Proj{ \bm{\mu}})) + \bm{U}(\bm{\varphi}, \bm{\sigma}), \label{multi:varphi} \\
\bm{\mu} & = B a_{,\bm{\varphi}}(\bm{\varphi}, \nabla \bm{\varphi}) - B \div (a_{,\nabla \bm{\varphi}}(\bm{\varphi}, \nabla \bm{\varphi})) + A \Psi_{,\bm{\varphi}}(\bm{\varphi}) + \bm{N}_{,\bm{\varphi}}(\bm{\varphi}, \bm{\sigma}), \label{multi:mu} \\
\pd_{t} \bm{\sigma} + \div (\bm{\sigma} \otimes \vec{v}) & = \div ( \TT{D}(\bm{\varphi}, \bm{\sigma}) \nabla \bm{N}_{,\bm{\sigma}}(\bm{\varphi}, \bm{\sigma})) + \bm{S}(\bm{\varphi}, \bm{\sigma}), \label{multi:sigma}
\end{align}
\end{subequations}
where
\begin{align*}
\bm{U} = \left ( \overline{\rho}_{1}^{-1} \mathcal{U}_{1}, \dots, \overline{\rho}_{L}^{-1} \mathcal{U}_{L} \right )^{\top} \in \R^{L}, \quad  \unit \cdot \bm{U} = \sum_{i=1}^{L} \frac{\mathcal{U}_{i}}{\overline{\rho}_{i}}, \quad \bm{S}= (S_{1}, \dots, S_{M})^{\top} \in \R^{M}.
\end{align*}
The constitutive choices above lead to the following energy identity.

\begin{thm}\label{thm:Energy}
A sufficiently smooth solution to \eqref{multiphasetumourmodel} fulfills
\begin{align*}
\frac{\dd}{\dt} \mathcal{E}(\bm{\varphi}, \bm{\sigma}) = \, &   \frac{\dd}{\dt} \int_{\Omega} \left ( A \Psi(\bm{\varphi}) + B a(\bm{\varphi}, \nabla \bm{\varphi}) + N(\bm{\varphi}, \bm{\sigma}) \right ) \dx \\
= \, & - \int_{\Omega} \TT{C}(\bm{\varphi}, \bm{\sigma}) \nabla (\Proj{\bm{\mu}}) : \nabla (\Proj{\bm{\mu}}) 
+ \TT{D}(\bm{\varphi}, \bm{\sigma}) \nabla N_{,\bm{\sigma}}(\bm{\varphi}, \bm{\sigma}) : \nabla N_{,\bm{\sigma}}(\bm{\varphi}, \bm{\sigma}) + \frac{\abs{\vec{v}}^{2} }{K} \dx \\
& - \int_{\Omega} \bm{S}(\bm{\varphi}, \bm{\sigma}) \cdot \bm{N}_{,\bm{\sigma}}(\bm{\varphi}, \bm{\sigma}) - \bm{U}(\bm{\varphi}, \bm{\sigma}) \cdot \Proj{\bm{\mu}} \dx \\
& - \int_{\Omega} (\unit \cdot \bm{U}(\bm{\varphi}, \bm{\sigma})) (\bm{\varphi} \cdot \Proj{\bm{\mu}} + \bm{\sigma} \cdot \bm{N}_{,\bm{\sigma}}(\bm{\varphi}, \bm{\sigma}) - N(\bm{\varphi}, \bm{\sigma}) - p)\dx \\
& + \int_{\pd \Omega} \TT{C}(\bm{\varphi}, \bm{\sigma}) \nabla (\Proj{\bm{\mu}}) : (\Proj{\bm{\mu}} \otimes \vec\nu) + B \Proj{a_{,\nabla \bm{\varphi}}(\bm{\varphi}, \nabla \bm{\varphi})} : (\pd_{t} \bm{\varphi} \otimes \vec{\nu})  \dHaus \\
& + \int_{\pd \Omega}  \TT{D}(\bm{\varphi}, \bm{\sigma}) \nabla \bm{N}_{,\bm{\sigma}}(\bm{\varphi}, \bm{\sigma}) : (\bm{N}_{,\bm{\sigma}}(\bm{\varphi}, \bm{\sigma}) \otimes \vec\nu)  + (N(\bm{\varphi}, \bm{\sigma}) + p) \vec{v} \cdot \vec{\nu} \dHaus.
\end{align*}
\end{thm}
\begin{proof}
Taking the scalar product of \eqref{multi:varphi} with $\Proj{\bm{\mu}}$ and integrating over $\Omega$ leads to
\begin{equation}
\label{energy:identity:part1}
\begin{aligned}
& \int_{\Omega} \Proj{\bm{\mu}} \cdot \pd_{t} \bm{\varphi} + (\bm{\varphi} \cdot \Proj{\bm{\mu}}) \unit \cdot \bm{U}(\bm{\varphi}, \bm{\sigma}) + \Proj{\bm{\mu}} \cdot (\nabla \bm{\varphi}) \vec{v} \dx \\
& \quad = \int_{\Omega} -\TT{C}(\bm{\varphi}, \bm{\sigma}) \nabla (\Proj{\bm{\mu}}) : \nabla (\Proj{\bm{\mu}}) + \bm{U}(\bm{\varphi}, \bm{\sigma}) \cdot \Proj{\bm{\mu}} \dx \\
& \quad \quad + \int_{\pd \Omega} \TT{C}(\bm{\varphi}, \bm{\sigma}) \nabla (\Proj{\bm{\mu}}) : (\Proj{\bm{\mu}} \otimes \vec{\nu}) \dHaus.
\end{aligned}
\end{equation}
Taking the projection of \eqref{multi:mu} and the scalar product with $\pd_{t} \bm{\varphi}$, and integrating over $\Omega$ leads to
\begin{equation}
\label{energy:identity:part2}
\begin{aligned}
\int_{\Omega} \Proj{\bm{\mu}} \cdot \pd_{t} \bm{\varphi} \dx & = \int_{\Omega} \Proj{\left (Ba_{,\bm{\varphi}}(\bm{\varphi}, \nabla \bm{\varphi}) + A \Psi_{,\bm{\varphi}}(\bm{\varphi}) + \bm{N}_{,\bm{\varphi}}(\bm{\varphi}, \bm{\sigma}) \right )} \cdot \pd_{t} \bm{\varphi} \dx \\
& \quad - \int_{\Omega} B \div (\Proj{a_{,\nabla \bm{\varphi}}(\bm{\varphi}, \nabla \bm{\varphi})}) \cdot \pd_{t} \bm{\varphi} \dx,
\end{aligned}
\end{equation}
where we used the linearity of the projection operator to deduce that $\Proj{\div (a_{,\nabla \bm{\varphi}})} = \div (\Proj{a_{,\nabla \bm{\varphi}}})$.  Integrating by parts on the last term of \eqref{energy:identity:part2} leads to
\begin{equation}\label{energy:identity:part3}
\begin{aligned}
\int_{\Omega} \Proj{\bm{\mu}} \cdot \pd_{t} \bm{\varphi} \dx & = \int_{\Omega} \Proj{\left (Ba_{,\bm{\varphi}}(\bm{\varphi}, \nabla \bm{\varphi}) + A \Psi_{,\bm{\varphi}}(\bm{\varphi}) + \bm{N}_{,\bm{\varphi}}(\bm{\varphi}, \bm{\sigma}) \right )} \cdot \pd_{t} \bm{\varphi} \dx \\
& \quad \quad + \int_{\Omega} B (\Proj{a_{,\nabla \bm{\varphi}}(\bm{\varphi}, \nabla \bm{\varphi})}) : \pd_{t} \nabla \bm{\varphi} \dx \\
& \quad \quad - \int_{\pd \Omega} B (\Proj{a_{,\nabla \bm{\varphi}}(\bm{\varphi}, \nabla \bm{\varphi})}) : (\pd_{t} \bm{\varphi} \otimes \vec{\nu}) \dHaus.
\end{aligned}
\end{equation}
Next, taking the scalar product of \eqref{multi:sigma} with $\bm{N}_{,\bm{\sigma}}$ and integrating over $\Omega$ leads to
\begin{equation}\label{energy:identity:part4}
\begin{aligned}
& \int_{\Omega} \bm{N}_{,\bm{\sigma}}(\bm{\varphi}, \bm{\sigma}) \cdot \pd_{t} \bm{\sigma} + (\bm{\sigma} \cdot \bm{N}_{,\bm{\sigma}}(\bm{\varphi}, \bm{\sigma})) \unit\cdot \bm{U}(\bm{\varphi}, \bm{\sigma}) + \bm{N}_{,\bm{\sigma}}(\bm{\varphi}, \bm{\sigma}) \cdot (\nabla \bm{\sigma}) \vec{v} \dx \\
& \quad  = - \int_{\Omega} \TT{D}(\bm{\varphi}, \bm{\sigma}) \nabla \bm{N}_{,\bm{\sigma}}(\bm{\varphi}, \bm{\sigma}) : \nabla \bm{N}_{,\bm{\sigma}}(\bm{\varphi}, \bm{\sigma}) - \bm{S}(\bm{\varphi}, \bm{\sigma}) \cdot \bm{N}_{,\bm{\sigma}}(\bm{\varphi}, \bm{\sigma}) \dx \\
& \quad \quad + \int_{\pd \Omega} \TT{D}(\bm{\varphi}, \bm{\sigma}) \nabla \bm{N}_{,\bm{\sigma}}(\bm{\varphi}, \bm{\sigma}) : (\bm{N}_{,\bm{\sigma}}(\bm{\varphi}, \bm{\sigma}) \otimes \vec{\nu}) \dHaus,
\end{aligned}
\end{equation}
while taking the scalar product of \eqref{multi:darcy} with $\vec{v}$ and integrating over $\Omega$ gives
\begin{equation}\label{energy:identity:part5}
\begin{aligned} 
\int_{\Omega} \frac{\abs{\vec{v}}^{2}}{K} \dx & = \int_{\Omega} -\nabla p \cdot \vec{v} + (\bm{\mu} - \bm{N}_{,\bm{\varphi}}(\bm{\varphi}, \bm{\sigma})) \cdot (\nabla \bm{\varphi}) \vec{v} \dx \\
& = \int_{\Omega} p \unit\cdot \bm{U}(\bm{\varphi}, \bm{\sigma}) + \Proj{(\bm{\mu} - \bm{N}_{,\bm{\varphi}}(\bm{\varphi}, \bm{\sigma}))} \cdot (\nabla \bm{\varphi}) \vec{v} \dx - \int_{\pd \Omega} p \vec{v} \cdot \vec{\nu} \dHaus,
\end{aligned}
\end{equation}
where we used the projection operator to deduce that $(\bm{\mu} - \bm{N}_{,\bm{\varphi}}) \cdot (\nabla \bm{\varphi}) \vec{v} = \Proj{(\bm{\mu} - \bm{N}_{,\bm{\varphi}})} \cdot (\nabla \bm{\varphi}) \vec{v}$.
Note that
\begin{align*}
& \int_{\Omega} \Proj{\bm{N}_{,\bm{\varphi}}(\bm{\varphi}, \bm{\sigma}) \cdot (\nabla \bm{\varphi}) \vec{v} + \bm{N}_{,\bm{\sigma}}(\bm{\varphi}, \bm{\sigma}) \cdot (\nabla \bm{\sigma}) \vec{v} \dx = \int_{\Omega} \nabla (N(\bm{\varphi}, \bm{\sigma})) \cdot \vec{v}} \dx \\
& \quad = \int_{\Omega} - N(\bm{\varphi}, \bm{\sigma}) \unit\cdot \bm{U}(\bm{\varphi}, \bm{\sigma}) \dx + \int_{\pd \Omega} N(\bm{\varphi}, \bm{\sigma}) \vec{v} \cdot \vec{\nu} \dHaus.
\end{align*}
Furthermore, by the definition of the projection operator and the fact that $\pd_{t} \bm{\varphi} \in \TangGibbs$, $\pd_{t} \nabla \bm{\varphi} \in (\TangGibbs)^{d}$, it holds that
\begin{align*}
\frac{\dd}{\dt} \mathcal{E}(\bm{\varphi}, \bm{\sigma}) & = \int_{\Omega} \Proj{(B a_{,\bm{\varphi}}(\bm{\varphi}, \nabla \bm{\varphi}) + A\Psi_{,\bm{\varphi}}(\bm{\varphi}) + N_{,\bm{\varphi}}(\bm{\varphi}, \bm{\sigma}))} \cdot \pd_{t} \bm{\varphi} \dx \\
& \quad + \int_{\Omega} B \Proj{a_{,\nabla \bm{\varphi}}(\bm{\varphi}, \nabla \bm{\varphi})} : \pd_{t} \nabla \bm{\varphi} + \bm{N}_{,\bm{\sigma}}(\bm{\varphi}, \bm{\sigma}) \cdot \pd_{t} \bm{\sigma} \dx.
\end{align*}
Thus, adding \eqref{energy:identity:part1}, \eqref{energy:identity:part3}, \eqref{energy:identity:part4} and \eqref{energy:identity:part5} gives the energy identity.
\end{proof}

\begin{remark}
It follows from Theorem \ref{thm:Energy} that, under the boundary conditions
\begin{align*}
\vec{v} \cdot \vec{\nu} = 0, \; \left ( \TT{C}(\bm{\varphi}, \bm{\sigma}) \nabla (\Proj{\bm\mu}) \right ) \vec{\nu} = \bm{0}, \; \left ( \TT{D}(\bm{\varphi}, \bm{\sigma}) \nabla \bm{N}_{,\bm{\sigma}}(\bm{\varphi}, \bm{\sigma}) \right ) \vec{\nu} = \bm{0}, \; \left (\Proj{a_{,\nabla \bm{\varphi}}(\bm{\varphi}, \nabla \bm{\varphi})}\right ) \vec{\nu} = \bm{0} 
\end{align*}
on $\pd \Omega$, and in the absence of source terms $\bm{S}(\bm{\varphi}, \bm{\sigma}) = \bm{0}$ and $\bm{U}(\bm{\varphi}, \bm{\sigma}) = \bm{0}$, the total free energy $\mathcal{E}(\bm{\varphi}, \bm{\sigma})$ is non-increasing in time.
\end{remark}

\subsection{Specific models}
\subsubsection{Zero velocity and zero excess of total mass}\label{sec:ZeroVelo:ZeroExcess}
Assuming zero excess of total mass, i.e., $\unit \cdot \bm{U} = \sum_{i=1}^{L} \overline{\rho}_{i}^{-1} \mathcal{U}_{i} = 0$, we obtain from \eqref{multi:div} that $\div \vec{v} = 0$.  Then, sending $K \to 0$ in \eqref{multi:darcy} formally implies that $\vec{v} \to \vec{0}$, see also \cite[\S 6]{GarckeLamDarcy} for a rigorous treatment in the two-component case.  Then \eqref{multiphasetumourmodel}, with source terms satisfying $\unit\cdot \bm{U} = 0$, can be reduced to
\begin{subequations}\label{model:zeroexcesstotalamass}
\begin{align}
\pd_{t} \bm{\varphi} & = \div ( \TT{C}(\bm{\varphi}, \bm{\sigma}) \nabla (\Proj{ \bm{\mu}})) + \bm{U}(\bm{\varphi}, \bm{\sigma}), \\
\bm{\mu} & = B a_{,\bm{\varphi}}(\bvarphi,\nabla\bvarphi) - B \div (a_{,\nabla \bm{\varphi}}(\bvarphi,\nabla\bvarphi)) + A \Psi_{,\bm{\varphi}}(\bvarphi) + \bm{N}_{,\bm{\varphi}}(\bm{\varphi}, \bm{\sigma}),  \\
\pd_{t} \bm{\sigma} & = \div ( \TT{D}(\bm{\varphi}, \bm{\sigma}) \nabla \bm{N}_{,\bm{\sigma}}(\bm{\varphi}, \bm{\sigma})) + \bm{S}(\bm{\varphi}, \bm{\sigma}),
\end{align}
\end{subequations}
which can be seen as the multiphase analogue of the model considered in \cite[\S 2.4.3]{GLSS}.  Note that due to the condition $\unit\cdot \bm{U} = 0$ and \eqref{Compatibility:symmetry=0} (for second order tensors) or \eqref{Compatibility:fourthorder} (for fourth order tensors), we necessarily have that $\bm{\varphi}(t) \in  \Gibbs$ for all $t > 0$ if the initial condition $\bm{\varphi}_{0}$ for $\bm{\varphi}$ belongs to $\Gibbs$.

\subsubsection{Choices for the Ginzburg--Landau energy}
Typical choices for the gradient part of the free energy are the following
\begin{align*}
a( \bm{\eta}, \nabla \bm{\varphi}) = \sum_{i=1}^{L} \frac{1}{2} \abs{\nabla \varphi_{i}}^{2}, \text{ or } a(\bm{\eta}, \nabla \bm{\varphi}) = \sum_{1 \leq i < j \leq L} \frac{1}{2} \beta_{ij}^{2} \abs{\eta_{i} \nabla \varphi_{j} - \eta_{j} \nabla \varphi_{i}}^{2},
\end{align*}
where the constants $\beta_{ij}$, $1 \leq i < j \leq L$ are referred to as the gradient energy coefficient of phases $i$ and $j$ (see \cite{GHaas,GarckeNestlerStothSIAM}).  For the potential part, we may consider the following
\begin{align*}
\Psi(\bm{\varphi}) = k_{\mathrm{B}} \theta  \sum_{i=1}^{L} \varphi_{i} \ln \varphi_{i} - \frac{1}{2} \bm{\varphi} \cdot \TT{\mathcal{W}} \bm{\varphi},
\end{align*}
where $k_{\mathrm{B}}$ denotes the Boltzmann constant, $\theta$ is the absolute temperature, and $\TT{\mathcal{W}} = (w_{ij})_{1 \leq i,j \leq L}$ is a symmetric $L \times L$ matrix with zeros on the diagonal and positive definite on $\mathrm{TG}$.  For example, the choice 
$\TT{\mathcal{W}} =  \id - \unit \otimes \unit$, where $\id$ is the identity matrix, is used in \cite{Blanketal14,GarckeNestlerStothSIAM,Nurnberg09}.  One can check that $\bm{\zeta} \cdot \left (\id - \unit \otimes \unit \right ) \bm{\zeta} = \abs{\bm{\zeta}}^{2}$ for any $\bm{\zeta} \in \mathrm{TG}$.  We can also consider obstacle potentials that penalise the order parameter $\bm{\varphi}$ from straying out of the set $\mathrm{G}$:
\begin{align}\label{Obstacle:Pot}
\Psi(\bm{\varphi}) = I_{\mathrm{G}}(\bm{\varphi}) - \frac{1}{2} \bm{\varphi} \cdot \TT{\mathcal{W}} \bm{\varphi} , \quad I_{\mathrm{G}}(\bm{y}) = \begin{cases} 
0 & \text{ for } \bm{y} \in \mathrm{G}, \\
\infty & \text{ otherwise}.
\end{cases}
\end{align}
Let us also mention potentials of polynomial type, which generalise the quartic double-well potential $(1-y^{2})^{2}$ commonly used in two-phase diffuse interface models.  One example is
\begin{align*}
\Psi(\bm{\varphi}) = \sum_{1 \leq i < j \leq L} \alpha_{ij} \varphi_{i}^{2} \varphi_{j}^{2},
\end{align*}
where $\alpha_{ij}$ are positive constants \cite{GarckeNestlerStoth98}.

\subsection{Degenerate Ginzburg--Landau energy}\label{sec:DegGL}
As described in Section~\ref{sec:22}, 
we may consider a Ginzburg--Landau-type energy of the form
\begin{align*}
\mathcal{E}(\bm{\varphi}, \bm\sigma) = \int_\Omega \frac{B}{2}\abs{\sum_{i=2}^{k}\nabla \varphi_{i}}^{2} + AW \left (\sum_{i=2}^{k}\varphi_{i} \right ) 
+ N(\bvarphi,\bm\sigma) \dx,
\end{align*}
for some $2 \leq k \leq L$, i.e., $\mathcal{E}(\bm{\varphi}, \bm\sigma)$ can be independent of $\varphi_{1}$ and $\varphi_{j}$ for any $j > k$, and $W$ is a scalar potential with equal minima at $0$ and $1$.  In the simplest setting $L = 2$ and if the chemical free energy density $N$ is independent of $\bm{\varphi}$, we obtain from \eqref{multi:mu} that
\begin{align*}
\mu_{1} = 0, \quad \mu_{2} =  - B \Laplace \varphi_{2} + A W'(\varphi_{2}).
\end{align*}
Together with a mobility tensor $\TT{C}(\bm{\varphi}, \bm{\sigma}) \in \R^{2 \times 2}$ 
such that 
$\TT{C}_{22}(\bm{\varphi}, \bm{\sigma}) = - \TT{C}_{21}(\bm{\varphi}, \bm{\sigma}) = m(\varphi_{2})$ 
for some mobility function $m$, we obtain from \eqref{multi:varphi} 
\begin{align*}
\pd_{t}\varphi_{1} + \div (\varphi_{1} \vec{v}) = -\div (m(\varphi_{2}) \nabla \mu_{2}) + \overline{\rho}_{1}^{-1} \mathcal{U}_{1}, \quad \pd_{t}\varphi_{2} + \div (\varphi_{2} \vec{v}) = \div (m(\varphi_{2}) \nabla \mu_{2}) + \overline{\rho}_{2}^{-1} \mathcal{U}_{2}.
\end{align*}
Thus, we obtain a Cahn--Hilliard type equation for $\varphi_{2}$, while for $\varphi_{1}$ we have a transport equation with source terms $\overline{\rho}_{1}^{-1} \mathcal{U}_{1}$ and $\div (m(\varphi_{2}) \nabla \mu_{2})$.  This is similar to the situations encountered in \cite{ChenLowen,CWSL,Dai,FJCWLC,WLFC,YLLL}.

\subsection{Mobility tensor}\label{sec:mob:tensor}
We consider second order mobility tensors $\TT{C}(\bm{\varphi}, \bm{\sigma})$ which fulfill \eqref{Compatibility:symmetry=0}.  For future analysis and numerical implementations, it is advantageous to consider a mobility that is symmetric and positive semi-definite on $\mathrm{TG}$, see for instance \cite{BarrettBG01,ElliottGarckeMulti}.  In most cases $\TT{C}(\bm{\varphi}, \bm{\sigma})$ is expected to mainly depend on $\bm{\varphi}$ and our standard choice will be independent of $\bm{\sigma}$ and of the form
\begin{align}\label{MobilityChoice:Symmetric}
\TT{C}_{ij}(\bm{\varphi}) = m_{i}(\varphi_{i}) \left ( \delta_{ij} - m_{j}(\varphi_{j}) \left ( \sum_{k=1}^{L} m_{k}(\varphi_{k}) \right )^{-1} \right ) \text{ for } 1 \leq i, j \leq L,
\end{align}
where $m_{i}(\varphi_{i}) \geq 0$, $1 \leq i \leq L$, are the so-called bare mobilities.  Here, we assume that the vector $(m_{1}(\varphi_{1}), \dots, m_{L}(\varphi_{L}))^{\top}$ is not identically zero on the Gibbs simplex, so that the reciprocal of the sum $\sum_{k=1}^{L} m_{k}(\varphi_{k})$ is well-defined.  Summing over $1 \leq i \leq L$ in \eqref{MobilityChoice:Symmetric} shows that \eqref{Compatibility:symmetry=0} is satisfied.  Furthermore, for any $\bm{\zeta} \in \R^{L}$, we have (for notational convenience we write $m_{i}$ for $m_{i}(\varphi_{i})$)
\begin{align*}
\bm{\zeta} \cdot \TT{C}(\bvarphi) \bm{\zeta} & = \frac{\left ( \sum_{i=1}^{L} m_{i} \abs{\zeta_{i}}^{2} \sum_{j=1}^{L}m_{j} \right ) - \left ( \sum_{i=1}^{L} m_{i} \zeta_{i} \right )^{2}}{\sum_{j = 1}^{L} m_{j}} \\
& = \frac{1}{\sum_{j=1}^{L} m_{j}} \left ( \sum_{1 \leq i < j \leq L} m_{i} m_{j} \left ( \abs{\zeta_{i}}^{2} + \abs{\zeta_{j}}^{2} - 2 \zeta_{i} \zeta_{j} \right ) \right ) = \frac{\sum_{1 \leq i < j \leq L} m_{i} m_{j} (\zeta_{i} - \zeta_{j})^{2}}{\sum_{j=1}^{L} m_{j}} \geq 0,
\end{align*}
where we have used the relations
\begin{align*}
\left ( \sum_{i=1}^{L} m_{i} \abs{\zeta_{i}}^{2} \right ) \left ( \sum_{j=1}^{L} m_{j} \right ) & = \sum_{i=1}^{L} m_{i}^{2} \abs{\zeta_{i}}^{2} + \sum_{1 \leq i \neq j \leq L} m_{i} m_{j} \abs{\zeta_{i}}^{2} \\
& = \sum_{i=1}^{L} m_{i}^{2} \abs{\zeta_{i}}^{2} + \sum_{1 \leq i < j \leq L} m_{i} m_{j} \left ( \abs{\zeta_{i}}^{2} + \abs{\zeta_{j}}^{2} \right ), \\
\left ( \sum_{i=1}^{L} m_{i} \zeta_{i} \right)^{2} & = \sum_{i=1}^{L} m_{i}^{2} \abs{\zeta_{i}}^{2} + 2 \sum_{1 \leq i < j \leq L} m_{i} m_{j} \zeta_{i} \zeta_{j}.
\end{align*}
In particular, for any $\bm{\varphi} \in \Gibbs$, $\TT{C}(\bm{\varphi})$ is positive semi-definite.

\subsection{Reduction to a two-component tumour model} \label{sec:27}
We assume that the domain $\Omega$ consists of proliferating tumour tissue and host tissue in the presence of a chemical species acting as a nutrient for the tumour.  Let $L = 2$ and $M = 1$, and set\begin{equation} \label{eq:sec27}
\begin{aligned}
\tilde{\varphi} & = \varphi_{2} - \varphi_{1}, \; \tilde{\Psi}(\tilde{\varphi}) =  \Psi \left ( \tfrac{1}{2}(1 - \tilde{\varphi}), \tfrac{1}{2}(1 + \tilde{\varphi}) \right ), \\ 
\tilde{\mu} &= \tfrac{1}{2}(\mu_{2} - \mu_{1}), \; a(\bm{\eta}, \nabla \bm{\varphi}) = \abs{\nabla \varphi_{1}}^{2} + \abs{\nabla \varphi_{2}}^{2}, \\
\tilde{N}(\tilde{\varphi}, \sigma) & = N \left ( \tfrac{1}{2}(1-\tilde{\varphi}), \tfrac{1}{2}(1 + \tilde{\varphi}) , \sigma \right ), \\
 \tilde{S}(\tilde{\varphi}, \sigma) & = S \left ( \tfrac{1}{2}(1-\tilde{\varphi}), \tfrac{1}{2}(1 + \tilde{\varphi}) , \sigma \right ), \\
\tilde{\mathcal{U}}_{i} (\tilde{\varphi}, \sigma) & = \mathcal{U}_{i} \left ( \tfrac{1}{2}(1-\tilde{\varphi}), \tfrac{1}{2}(1 + \tilde{\varphi}) , \sigma \right ) \text{ for } i = 1,2,
\end{aligned}
\end{equation}
together with a scalar mobility 
\begin{equation} \label{eq:Dn}
{D}((\tfrac{1}{2}(1-\tilde{\varphi}), \tfrac{1}{2}(1+\tilde{\varphi})), \sigma) = n(\tilde{\varphi}),
\end{equation}
which we here assume to be independent of $\sigma$, for the nutrient equation and a second order mobility tensor $\TT{C}(\bm{\varphi})$ of the form \eqref{MobilityChoice:Symmetric} with bare mobilities $m_{1}(\varphi_{1})$ and $m_{2}(\varphi_{2})$.  With the help of \eqref{MobilityChoice:Symmetric} the entries of $\TT{C}(\bm{\varphi})$ can be computed as
\begin{align*}
\TT{C}_{11}(\bm{\varphi}) = \TT{C}_{22}(\bm{\varphi}) = - \TT{C}_{12}(\bm{\varphi}) = -\TT{C}_{21}(\bm{\varphi}) = \frac{m_{1}(\varphi_{1}) m_{2}(\varphi_{2})}{m_{1}(\varphi_{1}) + m_{2}(\varphi_{2})}.
\end{align*} 
Then, upon defining a non-negative scalar mobility $m$ that is a function of $\tilde{\varphi}$ as
\begin{align}\label{twophasemodel:mobility}
m(\tilde{\varphi}) = \frac{4 m_{1}(\tfrac{1-\tilde{\varphi}}{2}) m_{2}(\tfrac{1+\tilde{\varphi}}{2})}{m_{1}(\tfrac{1-\tilde{\varphi}}{2}) + m_{2}(\tfrac{1+\tilde{\varphi}}{2})} \, \Rightarrow \, 
\TT{C}(\bm{\varphi}) = \frac{1}{4} \left ( \begin{array}{cc}
m(\tilde{\varphi}) & -m(\tilde{\varphi}) \\
-m(\tilde{\varphi}) & m(\tilde{\varphi}) \end{array} \right ) \in \R^{2 \times 2},
\end{align}
it can be shown that \eqref{multiphasetumourmodel} becomes
\begin{subequations}\label{twophasemodel}
\begin{align}
\div \vec{v} & = \overline{\rho}_{1}^{-1} \tilde{\mathcal{U}}_{1}(\tilde{\varphi}, \sigma) + \overline{\rho}_{2}^{-1} \tilde{\mathcal{U}}_{2}(\tilde{\varphi}, \sigma), \\
\vec{v} & = -K \nabla p + K (\tilde{\mu} - \tilde{N}_{,\tilde{\varphi}}(\tilde{\varphi}, \sigma)) \nabla \tilde{\varphi}, \\
\pd_{t} \tilde{\varphi} + \div (\tilde{\varphi} \vec{v}) & = \div (m(\tilde{\varphi}) \nabla \tilde{\mu}) + \overline{\rho}_{2}^{-1} \tilde{\mathcal{U}}_{2}(\tilde{\varphi}, \sigma) - \overline{\rho}_{1}^{-1} \tilde{\mathcal{U}}_{1}(\tilde{\varphi}, \sigma), \\
\tilde{\mu} & = A \tilde{\Psi}'(\tilde{\varphi}) - B \Laplace \tilde{\varphi} + \tilde{N}_{,\tilde{\varphi}}(\tilde{\varphi}, \sigma) , \\
\pd_{t} \sigma + \div (\sigma \vec{v}) & = \div (n(\tilde{\varphi}) \nabla \tilde{N}_{,\sigma}) + \tilde{S}(\tilde{\varphi}, \sigma) ,
\end{align}
\end{subequations}
which coincides with \cite[Equation (2.25)]{GLSS}.  We refer the reader to \cite{GLSS} for a detailed comparison between \eqref{twophasemodel} with other two-component phase field models of tumour growth in the literature.

\subsection{Tumour with quiescent and necrotic cells}\label{sec:specificmodel:necroticcore}
In this section, we give some examples of source terms for the case where a tumour exhibits a quiescent region and a necrotic region.  Let $L = 4$ and denote the volume fractions of the host tissue, proliferating tumour cells, quiescent tumour cells and necrotic tumour cells by $\varphi_{H}$, $\varphi_{P}$, $\varphi_{Q}$, and $\varphi_{N}$, respectively, i.e., $\bm{\varphi} = (\varphi_{H}, \varphi_{P}, \varphi_{Q}, \varphi_{N})^\top$.

We assume matched densities, i.e., $\overline{\rho}_{H} = \overline{\rho}_{P} = \overline{\rho}_{Q} = \overline{\rho}_{N} = 1$, and that there are two chemical species present in the domain, i.e., $M = 2$. The first is a nutrient whose concentration is denoted as $\sigma_{\nut}$, and is only consumed by the proliferating and quiescent tumour cells, and the second is a toxic intracellular agent, whose concentration is denoted as $\sigma_{\tox}$. Hence $\bm\sigma = (\sigma_{nu}, \sigma_{tx})^\top$.  During necrosis, the cell membrane loses its integrity and toxic agents from the former intracellular compartment flow outwards.  We assume that these toxic agents act as growth inhibitors on the surrounding living cells and degrade at a constant rate.  Furthermore, we denote by $\sigma_{\ptq}^{*}$, $\sigma_{\qtn}^{*}$, $\sigma_{\tox}^{*} > 0$ the critical concentrations such that
\begin{itemize}
\item if $\sigma_{\qtn}^{*} < \sigma_{\nut} < \sigma_{\ptq}^{*}$, then the proliferating tumour cells will turn quiescent,
\item if $\sigma_{\nut} < \sigma_{\qtn}^{*}$, then the quiescent tumour cells will undergo necrosis,
\item if $\sigma_{\tox} \geq \sigma_{\tox}^{*}$, then the toxic agents start to inhibit the growth of the living cells.
\end{itemize}
For the source/sink terms $S_{j}(\bvarphi,\bm\sigma)$, $j \in \{ \nut, \tox \}$, we consider
\begin{subequations}
\begin{align}
S_{\nut}(\bvarphi,\bm\sigma) & = \underbrace{- \sigma_{\nut} \left ( \varphi_{P} \mathcal{C}_{P}  +  \varphi_{Q}\mathcal{C}_{Q}  \right )}_{\text{consumption by living tumour cells}} \label{sourceterm:nutrient1}, \\
S_{\tox}(\bvarphi,\bm\sigma) & = \underbrace{ \varphi_{N} \mathcal{R}_{\tox} }_{\text{release by necrotic cells}} - \underbrace{\mathcal{D}_{\tox} \sigma_{\tox}}_{\text{degradation}} \label{soruceterm:toxic1},
\end{align}
\end{subequations}
with constant consumption rates $\mathcal{C}_{P}$, $\mathcal{C}_{Q} \geq 0$ by the proliferating and quiescent cells, respectively, constant release rate $\mathcal{R}_{\tox} \geq 0$ of toxic agents by the necrotic cells, and constant degradation rate $\mathcal{D}_{\tox} \geq 0$ of the toxic agents.  We consider the following free energy density $N(\bm{\varphi}, \bm{\sigma})$:
\begin{align}\label{samplefreeenergyN1}
N(\bm{\varphi}, \bm{\sigma}) = \frac{D_{\nut}}{2} \abs{\sigma_{\nut}}^{2} + \frac{D_{\tox}}{2} \abs{\sigma_{\tox}}^{2} - \chi_{\nut} \sigma_{\nut} \varphi_{P},
\end{align}
where $D_{\nut}, D_{\tox} > 0$ denote parameters related to the diffusivity of the nutrient and of the toxic agent, respectively, and $\chi_{\nut} \geq 0$ can be viewed as a parameter for transport mechanisms such as chemotaxis and active transport.  Neglecting the toxic agent, the above form for the free energy density $N$ is similar to the one chosen in \cite{GLSS,HawkinsZeeOden12}.  In particular, the first two terms of $N$ lead to diffusion of the nutrient and toxic agent, respectively, while the third term of $N$ will give rise to transport mechanisms that drive the proliferating tumour cells to the regions of high nutrient, and also drive the nutrient to the proliferating tumour cells, see \cite{GLSS} for more details regarding the effects of the third term.

Then, computing $\bm{N}_{,\bm{\sigma}}(\bvarphi,\bm\sigma)$ and considering $\TT{D}(\bm{\varphi}, \bm{\sigma})$ to be the second order identity tensor $\id \in \R^{2 \times 2}$, \eqref{multi:sigma} becomes
\begin{subequations}\label{Nutrient:Toxic:equ}
\begin{align}
\pd_{t}\sigma_{\nut} + \div (\sigma_{\nut} \vec{v}) & = \div \left (D_{\nut} \nabla \sigma_{\nut} - \chi_{\nut} \nabla \varphi_{P} \right ) - \sigma_{\nut} \left ( \varphi_{P} \mathcal{C}_{P}  + \varphi_{Q} \mathcal{C}_{Q}  \right ), \\
\pd_{t}\sigma_{\tox} + \div (\sigma_{\tox} \vec{v}) & = \div \left ( D_{\tox} \nabla \sigma_{\tox} \right ) + \varphi_{N} \mathcal{R}_{\tox}  - \mathcal{D}_{\tox} \sigma_{\tox}.
\end{align} 
\end{subequations}
For the source terms $\mathcal{U}_{H}, \mathcal{U}_{P}, \mathcal{U}_{Q}, \mathcal{U}_{N}$, we assume that
\begin{itemize}
\item the host cells experience apoptosis at a constant rate $\mathcal{A}_{H} \geq 0$ and are inhibited by the toxic agent at a constant rate $\mathcal{A}_{\tox} \geq 0$, leading to
\begin{align*}
\mathcal{U}_{H}(\bvarphi,\bm\sigma) & = \underbrace{- \varphi_{H} \mathcal{A}_{\tox} (\sigma_{\tox} - \sigma_{\tox}^{*})^{+}}_{\text{inhibition by toxic agents}}  - \underbrace{\varphi_{H} \mathcal{A}_{\mathrm{H}}}_{\substack{\text{apoptosis of} \\ \text{host tissue}}},
\end{align*}
where $(f)^{+} = \max (0,f)$ denotes the positive part of $f$.
\item The proliferating tumour cells grow due to nutrient consumption at a constant rate $\mathcal{P} \geq 0$, experience apoptosis at a constant rate $\mathcal{A}_{P} \geq 0$, and are inhibited by the toxic agents at the rate $\mathcal{A}_{\tox}$.  Furthermore, when $\sigma_{\nut}$ falls below $\sigma_{\ptq}^{*}$, there is a transition to the quiescent cells at a constant rate $\mathcal{T}_{\ptq}\geq 0$, but when the nutrient concentration is above $\sigma_{\ptq}^{*}$, there is a transition from the quiescent cells at a constant rate $\mathcal{T}_{\qtp} \geq 0$.  Altogether this yields
\begin{equation*}
\begin{aligned}
\mathcal{U}_{P}(\bvarphi,\bm\sigma) & = \underbrace{\varphi_{P} \mathcal{P} \sigma_{\nut}}_{\substack{\text{growth due to} \\ \text{nutrient consumption}}} - \underbrace{\varphi_{P} \mathcal{A}_{\tox} (\sigma_{\tox} - \sigma_{\tox}^{*})^{+}}_{\substack{\text{inhibition by} \\ \text{toxic agents}}} - \underbrace{ \varphi_{P} \mathcal{A}_{P}}_{\substack{\text{apoptosis of} \\ \text{proliferating cells}}}  \\
\notag & \quad + \underbrace{ \varphi_{Q} \mathcal{T}_{\qtp} (\sigma_{\nut} - \sigma_{\ptq}^{*})^{+}}_{\substack{\text{transition from quiescent} \\ \text{ to proliferating cells}}} - \underbrace{ \varphi_{P} \mathcal{T}_{\ptq} ( \sigma_{\ptq}^{*} - \sigma_{\nut})^{+}}_{\substack{\text{transition from proliferating} \\ \text{ to quiescent cells}}}. 
\end{aligned} 
\end{equation*}
\item The quiescent cells experience apoptosis at a constant rate $\mathcal{A}_{Q} \geq 0$, and are inhibited by the toxic agent at the rate $\mathcal{A}_{\tox}$.  Furthermore, aside from the exchange between the proliferating cells and the quiescent cells when the nutrient concentration falls below or is above the critical concentration $\sigma_{\ptq}^{*}$, there is also a transition to the necrotic cells when $\sigma_{\nut}$ falls below $\sigma_{\qtn}^{*}$.  This occurs at a constant rate $\mathcal{T}_{\qtn} \geq 0$, and we obtain
\begin{equation*}
\begin{aligned}
\mathcal{U}_{Q}(\bvarphi,\bm\sigma) & =  - \underbrace{\varphi_{Q} \mathcal{A}_{\tox} (\sigma_{\tox} - \sigma_{\tox}^{*})^{+}}_{\substack{\text{inhibition by} \\ \text{toxic agents}}} - \underbrace{ \varphi_{Q} \mathcal{A}_{Q}}_{\substack{\text{apoptosis of} \\ \text{quiescent cells}}} - \underbrace{\varphi_{Q} \mathcal{T}_{\qtn} (\sigma_{\qtn}^{*} - \sigma_{\nut} )^{+}}_{\substack{\text{transition from quiescent} \\ \text{ to necrotic cells}}} \\
\notag & \quad - \underbrace{\varphi_{Q} \mathcal{T}_{\qtp} (\sigma_{\nut} - \sigma_{\ptq}^{*})^{+}}_{\substack{\text{transition from quiescent} \\ \text{ to proliferating cells}}} + \underbrace{ \varphi_{P} \mathcal{T}_{\ptq} ( \sigma_{\ptq}^{*} - \sigma_{\nut})^{+}}_{\substack{\text{transition from proliferating} \\ \text{ to quiescent cells}}}.
\end{aligned}
\end{equation*}
\item The necrotic cells degrades at a constant rate $\mathcal{D}_{N} \geq 0$ and there is a transition from the quiescent cells at the rate $\mathcal{T}_{\qtn}$ when $\sigma_{\nut}$ falls below $\sigma_{\qtn}^{*}$.  Furthermore, the apoptosis of the proliferating and quiescent cells is a source term for the necrotic cells.  This yields
\begin{align*}
\mathcal{U}_{N}(\bvarphi,\bm\sigma) & = \underbrace{\varphi_{P} \mathcal{A}_{P} + \varphi_{Q} \mathcal{A}_{Q}}_{\substack{\text{apoptosis of proliferating} \\ \text{and quiescent cells}}} +  \underbrace{\varphi_{Q} \mathcal{T}_{\qtn} (\sigma_{\qtn}^{*} - \sigma_{\nut} )^{+}}_{\substack{\text{transition from quiescent} \\ \text{ to necrotic cells}}} - \underbrace{ \mathcal{D}_{N} \varphi_{N}}_{\text{degradation}}.
\end{align*}
\end{itemize}
In practice, on the time scale considered, $\mathcal{A}_{H}$ is small and will often be neglected.  A unique feature of the necrotic core is reflected in the second term of $\mathcal{U}_{N}$, which describes a spontaneous degradation of the necrotic core.  Physiologically, one would expect that the remains of the necrotic cells are slowly processed by specialised cells, leaving only extracellular liquid behind.  Since we do not account for a pure liquid phase in our systems, we obtain a local mass defect due to the disintegration of the necrotic core.  For the source terms discussed above, the equation \eqref{proto:div} for equal densities $\overline{\rho}_{H} = \overline{\rho}_{P} = \overline{\rho}_{Q} = \overline{\rho}_{N} = 1$ then becomes
\begin{align*}
\div \vec{v} & = \varphi_{P}\mathcal{P} \sigma_{\nut}  - \varphi_{H} \mathcal{A}_{H} - \mathcal{A}_{\tox} (1-\varphi_{N}) (\sigma_{\tox} - \sigma_{\tox}^{*})^{+} - \mathcal{D}_{N} \varphi_{N},
\end{align*}
and the disintegration of the necrotic core leads to a sink term for the divergence of the volume-averaged velocity field.  Hence one could argue that there are two effects resulting from the existence of a necrotic core which could possibly limit the uncontrolled growth of the tumour colony.  On the one hand we have the obvious growth inhibition due to the toxic agents, whereas on the other hand the degradation of the necrotic core draws the growing periphery of the tumour back towards the tumour centre.

\subsection{Blood vessels and angiogenic factors}\label{sec:angio}
We can introduce angiogenic factors into the system by considering two additional chemical species: blood vessels whose density is denoted as $b$ and an angiogenic factor whose concentration is denoted as $a$. 
Hence $\bm\sigma = (\sigma_{nu},\sigma_{tx},a,b)^\top$.  We assume that
\begin{itemize}
\item the blood vessels offer a supply of nutrient $\sigma_{\mathrm{Sup}} \geq 0$ at a constant rate $\mathcal{B}_{\nut} \geq 0$, which leads to the modification
\begin{align*}
S_{\nut}(\bvarphi,\bm\sigma) & = \underbrace{\mathcal{B}_{\nut} b \left ( \sigma_{\mathrm{Sup}} - \sigma_{\nut}\right )}_{\text{nutrient supply from blood vessels}}- \sigma_{\nut} \left ( \varphi_{P} \mathcal{C}_{P}  + \varphi_{Q} \mathcal{C}_{Q}  \right ).
\end{align*}
The new term $\mathcal{B}_{\nut} b \left ( \sigma_{\mathrm{Sup}} - \sigma_{\nut}\right )$ in $S_{\nut}$ models the situation where if the nutrient concentration is below $\sigma_{\mathrm{Sup}}$, then additional nutrient is supplied by the blood vessels at a rate $\mathcal{B}_{\nut}$.  However, if $\sigma_{\nut} \geq \sigma_{\mathrm{Sup}}$, then the nutrient diffuses into the blood vessels and is transported away from the cells.
\item The blood vessels are capable of removing the toxic agents released by the necrotic cells at a constant rate $\mathcal{B}_{\tox} \geq 0$, which leads to the modification
\begin{align*}
S_{\tox}(\bvarphi,\bm\sigma) & = \varphi_{N} \mathcal{R}_{\tox}  - \mathcal{D}_{\tox} \sigma_{\tox} - \underbrace{ \mathcal{B}_{\tox}  \sigma_{\tox} b}_{\text{removal by blood vessels}} .
\end{align*}
\item The angiogenic factor is a chemical species that is released by the quiescent tumour cells at a constant rate $\mathcal{R}_{\mathrm{ang}} \geq 0$ due to the lack of nutrient in their surroundings, and it degrades at a constant rate $\mathcal{D}_{\mathrm{ang}} \geq 0$.  This leads to
\begin{align*}
S_{a}(\bvarphi,\bm\sigma) & = \underbrace{\varphi_{Q} \mathcal{R}_{\mathrm{ang}} }_{\text{release by queiscent cells}}  - \underbrace{\mathcal{D}_{\mathrm{ang}} a}_{\text{degradation}}.
\end{align*}
In our model, tumour cells become quiescent as a consequence of a lack of nutrient. Therefore it makes sense to assume that the cells, which are in most need of a reliable vascularisation, are secreting factors which induce the necessary blood vessel growth. This assumption has already been suggested in \cite{Byrne,Cumsille}.  A very important example for tumour nutrient is oxygen. It is well known that a lack of this nutrient, hypoxia, is an important stimulus for angiogenesis \cite{Rey}.

\item Meanwhile, the angiogenic factor induces angiogenesis and consequently the vessel density around the badly supplied tumour cells increases at a constant rate $\mathcal{G}_{\mathrm{bv}} \geq 0$.  There are two ways in which the blood vessels can degrade.  The first is a natural process which occurs at a constant rate $\mathcal{D}_{\mathrm{bv}} \geq 0$, and the second is through the overexposure of the toxic agent.  That is, the blood vessels degrade at a constant rate $\mathcal{D}_{\mathrm{bv}}$ when the concentration of the toxic agent $\sigma_{\tox}$ is higher than the critical value $\sigma_{\tox}^{*}$.  These considerations lead to
\begin{align*}
S_{b}(\bvarphi,\bm\sigma)  & = \underbrace{\mathcal{G}_{\mathrm{bv}} ab}_{\substack{\text{vessel growth due to} \\ \text{ angiogenic factors}} }  - \underbrace{\mathcal{D}_{\mathrm{bv}} b}_{\substack{\text{natural} \\ \text{degradation}}} - \underbrace{\mathcal{D}_{\mathrm{bv}} (\sigma_{\tox} - \sigma_{\tox}^{*})^{+} b}_{\substack{\text{degradation due to} \\ \text{toxic agents}}}.
\end{align*} 
\end{itemize}
Similar to Section \ref{sec:specificmodel:necroticcore}, for the choice of the free energy density $N(\bm{\varphi}, \bm{\sigma})$, we consider 
\begin{align}\label{samplefreeenergyN2}
N(\bm{\varphi}, \bm{\sigma}) = \frac{D_{\nut}}{2} \abs{\sigma_{\nut}}^{2} + \frac{D_{\mathrm{bv}}}{2} \abs{b}^{2} + \frac{D_{\mathrm{ang}}}{2} \abs{a}^{2} + \frac{D_{\tox}}{2} \abs{\sigma_{\tox}}^{2} - \chi_{\nut} \sigma_{\nut}\varphi_{P}.
\end{align}
The difference between \eqref{samplefreeenergyN1} and \eqref{samplefreeenergyN2} is the addition of the terms $\frac{D_{\mathrm{bv}}}{2} \abs{b}^{2} + \frac{D_{\mathrm{ang}}}{2} \abs{a}^{2}$ to model the diffusion of the blood vessel density and the angiogenesis factor, respectively.  Computing $\bm{N}_{,\bm{\sigma}}(\bvarphi,\bm\sigma)$ and taking $\TT{D}(\bm{\varphi}, \bm{\sigma})$ as the identity tensor in $\R^{4 \times 4}$, we arrive at the following system for the chemical species:
\begin{subequations}
\begin{align}
\pd_{t}\sigma_{\nut} + \div (\sigma_{\nut} \vec{v}) & = \div \left (D_{\nut} \nabla \sigma_{\nut} - \chi_{\nut} \nabla \varphi_{P} \right ) + \mathcal{B}_{\nut} b \left ( \sigma_{\mathrm{Sup}} - \sigma_{\nut}\right ) \\
\notag & \quad - \sigma_{\nut} \left ( \varphi_{P} \mathcal{C}_{P}  + \varphi_{Q} \mathcal{C}_{Q}  \right ), \\
\pd_{t}\sigma_{\tox} + \div (\sigma_{\tox} \vec{v}) & = \div \left ( D_{\tox} \nabla \sigma_{\tox} \right ) + \varphi_{N} \mathcal{R}_{\tox}  - \mathcal{D}_{\tox} \sigma_{\tox} - \mathcal{B}_{\tox} \sigma_{\tox} b, \\
\pd_{t}b + \div (b \vec{v}) & = \div \left ( D_{\mathrm{bv}} \nabla b \right )  + \mathcal{G}_{\mathrm{bv}} ab - \mathcal{D}_{\mathrm{bv}} b \left ( 1 + \left ( \sigma_{\tox} - \sigma_{\tox}^{*} \right)^{+} \right ), \\
\pd_{t}a + \div (a \vec{v}) & = \div \left ( D_{\mathrm{ang}} \nabla  a \right ) + \varphi_{Q} \mathcal{R}_{\mathrm{ang}}  - \mathcal{D}_{\mathrm{ang}} a.
\end{align}
\end{subequations}
We expect that $D_{\mathrm{bv}} = 0$ in practice, however choosing $D_{\mathrm{bv}}$ to be positive is beneficial for the analytical and numerical treatment of the equations.

An alternative way to model angiogenesis is as follows.  One could fix the blood vessel density on the boundary of the domain and assume that blood vessel growth is governed by chemotaxis towards the angiogenic factor, meaning that blood vessels are drawn towards regions with a high concentration of angiogenic factors. In this case, we neglect the first term of $S_{b}$, leading to
\begin{align*}
S_{b}(\bvarphi,\bm\sigma) = -\mathcal{D}_{\mathrm{bv}} b -\mathcal{D}_{\mathrm{bv}}  (\sigma_{\tox} - \sigma_{\tox}^{*})^{+} b.
\end{align*} 
If we consider the free energy density $N$ as in \eqref{samplefreeenergyN2}, with its partial derivative with respect to the vector $\bm{\sigma}$ given as
\begin{align*}
\bm{N}_{,\bm{\sigma}}(\bvarphi,\bm\sigma) = \left ( D_{\nut} \sigma_{\nut} + \chi_{\nut} \varphi_{P}, \; D_{\mathrm{bv}} b, \; D_{\mathrm{ang}} a, \; D_{\tox} \sigma_{\tox} \right )^{\top} \in \R^{4},
\end{align*}
then we may consider a second order mobility tensor $\TT{D}(\bm{\varphi}, \bm{\sigma}) \in \R^{4 \times 4}$ of the form
\begin{align*}
[\TT{D}(\bm{\varphi}, \bm{\sigma})]_{ij} = \begin{cases}
1 & \text{ if } i = j, \\
- \frac{\chi_{\mathrm{ang}}}{D_{\mathrm{ang}}} b & \text{ if } i = 2, j = 3, \\
0 &\text{ otherwise},
\end{cases}
\end{align*}
where $\chi_{\mathrm{ang}} \geq 0$ is a chemotactic sensitivity to the angiogenic factor.  Then, upon computing $\TT{D}(\bm{\varphi}, \bm{\sigma}) \nabla \bm{N}_{,\sigma}(\bm{\varphi}, \bm{\sigma})$, this yields the following convection-reaction-diffusion system for the blood vessel density and the angiogenic factor:
\begin{align*}
\pd_{t}b + \div (b \vec{v}) & = \div (D_{\mathrm{bv}} \nabla b - \chi_{\mathrm{ang}}b  \nabla a) -\mathcal{D}_{\mathrm{bv}} b  - \mathcal{D}_{\mathrm{bv}} (\sigma_{\tox} - \sigma_{\tox}^{*})^{+} b, \\
\pd_{t} a + \div (a \vec{v}) & = \div (D_{\mathrm{ang}} \nabla a) + \varphi_{Q} \mathcal{R}_{\mathrm{ang}} - \mathcal{D}_{\mathrm{ang}} a.
\end{align*}
The term $-\div ( \chi_{\mathrm{ang}} b \nabla a)$ in the equation for $b$ can also be found in the classical models for chemotaxis (see for example \cite{article:HillenPainter, article:Horstmann, article:KellerSegel70, article:KellerSegel71}).
We remark that the above modelling approach is different to that in \cite{FJCWLC} (see also \cite[Section 5.12]{book:CristiniLowengrub}), which utilises a random walk model for angiogenesis.

\subsection{Three phase model with necrotic cells}\label{sec:3phase}
In Section \ref{sec:numerics}, we perform numerical simulations of a three-component model, similar to \eqref{Intro:3component}, consisting of host, proliferating and necrotic cells, along with a single nutrient $\sigma = \sigma_{\nut}$.  Neglecting the quiescent cells $(\varphi_{Q})$ and the toxic intracelluar agent $(\sigma_{\tox})$, as well as the apoptosis of host cells $(\mathcal{A}_{H} = 0)$, the source terms from Section \ref{sec:specificmodel:necroticcore} now become
\begin{align*}
\mathcal{U}_{H}(\bvarphi,\bm\sigma) = 0, \; \mathcal{U}_{P}(\bvarphi,\bm\sigma) = \varphi_{P} (\mathcal{P} \sigma - \mathcal{A}_{P}), \; \mathcal{U}_{N}(\bvarphi,\bm\sigma) = \mathcal{A}_{P} \varphi_{P} - \mathcal{D}_{N} \varphi_{N}, 
\end{align*}
where the mass lost by the proliferating cells through apoptosis is equal to the mass gained by the necrotic cells.  In the case of equal densities $\overline{\rho}_{H} = \overline{\rho}_{P} = \overline{\rho}_{N} = 1$, this yields the vector $\bm{U}_{A}$ in \eqref{Intro:source:1}.  Alternatively, we can consider source terms of the form
\begin{align}\label{3phase:source:2}
\mathcal{U}_{H}(\bvarphi,\bm\sigma)  = 0, \; \mathcal{U}_{P}(\bvarphi,\bm\sigma)  = \eps^{-1} F(\varphi_{P}) \left ( \mathcal{P} \sigma - \mathcal{A} \right ), \; \mathcal{U}_{N}(\bvarphi,\bm\sigma)  = \eps^{-1} F(\varphi_{N}) \left ( \mathcal{A} - \mathcal{D}_{N} \right ),
\end{align}
where $F$ is a non-negative function satisfying $F(0) = F(1) = 0$, $F'(0) = F'(1) = 0$, and $\eps > 0$ is a parameter measuring the thickness of the interfacial layers, recall \eqref{Intro:energy}.  One such example is $F(s) = s^{2}(1-s)^{2}$, which in the case of equal densities $\overline{\rho}_{H} = \overline{\rho}_{P} = \overline{\rho}_{N} = 1$ leads to the vector $\bm{U}_{C}$ in \eqref{Intro:source:2}.  These source terms are chosen in the spirit of \cite{Kampmann} (see also \cite[\S 3.3.2]{GLSS}), where we note that $\varphi_{i}^{2}(1-\varphi_{i})^{2}$ is non-zero only near the vicinity of the interfacial layers, while the scaling with $\frac{1}{\eps}$ and the specific properties of $F$ ensure that these source terms only appear in the equation of motion for the interfaces when we consider the sharp interface limit $\eps \to 0$.

\section{Sharp interface asymptotics}\label{sec:SIM}
Different models have been suggested to describe free boundary problems involving multiphase tumour growth.  In particular, the effect of a necrotic core has been studied in \cite{CuiFriedman,EscherMM}.  In this section, we will perform a formally matched asymptotic analysis for the phase field model \eqref{multiphasetumourmodel} in order to derive new free boundary problems for tumour growth.  We make the following assumptions:
\begin{assump}\label{assump:Asymptotics}
\
\begin{enumerate}
\item $A = \frac{\beta}{\eps}$ and $B = \beta \eps$ for positive constants $\beta$ and $\eps$.
\item The mass exchange terms $\bm{U} \in \R^{L}$ and $\bm{S} \in \R^{M}$ depend only on $\bm{\varphi} \in \R^{L}$ and $\bm{\sigma} \in \R^{M}$, and not on any derivatives.
\item The mobility tensor $\TT{D}(\bvarphi,\bm\sigma) \in \R^{M \times d \times M \times d}$ is a strictly positive and smooth fourth order tensor for all $\bm{\varphi} \in \Gibbs$
and $\bm{\sigma} \in \R^{M}$.  Here by a strictly positive fourth order tensor $\TT{A}$ we mean $\bm{t} : \left ( \TT{A} \bm{t} \right ) > 0$ for all second order tensors $\bm{t} \in \R^{M \times d}$, $\bm{t} \neq \bm{0}$, and $\bm{t} : \left ( \TT{A} \bm{t} \right ) = 0 \Leftrightarrow \bm{t} = \bm{0}$.
\item $\Psi: \mathrm{G} \to \R_{\geq 0}$ is a smooth multi-well potential with $L$ equal minima at the points $\bm{e}_{l}$ satisfying $\Psi(\bm{e}_{l}) = 0$ for $1 \leq l \leq L$.  Furthermore, we assume that there exist constants $c_{1}, c_{2}, c_{3}$ and $p \geq 2$ such that
\begin{align*}
c_{1} \abs{\bm{\varphi}}^{p} \leq \Psi(\bm{\varphi}) \leq c_{2} \abs{\bm{\varphi}}^{p} \quad \text{ for } \abs{\bm{\varphi}} \geq c_{3}.
\end{align*}
\item We choose the gradient energy as
\begin{align*}
a(\bm{\eta}, \nabla \bm{\varphi}) = \frac{1}{2} \sum_{i=1}^{L} \abs{\nabla \varphi_{i}}^{2} = \frac{1}{2} \sum_{i=1}^{L} \sum_{k=1}^{d} \abs{\pd_{x_{k}} \varphi_{i}}^{2} = \frac{1}{2} (\nabla \bm{\varphi} : \nabla \bm{\varphi}).
\end{align*}
\item The mobility tensor $\TT{C}(\bvarphi,\bm\sigma) \in \R^{L \times d \times L \times d}$ for all $\bm{\varphi} \in \Gibbs$ and $\bm{\sigma} \in \R^{M}$
is a smooth fourth order tensor such that \eqref{Compatibility:symmetry=0} is satisfied and also fulfils $\TT{C}(\bvarphi,\bm\sigma) (\bm{a} \otimes \vec{b}) : (\bm{a} \otimes \vec{b}) > 0$ for all $\bm{0} \neq \bm{a} \in \{\unit\}^{\perp}$ and $\vec{0} \neq \vec{b} \in \R^{d}$.
\item For small $\eps$, we assume that the domain $\Omega$ can be divided into $L$ open subdomains $\Omega_{i}(\eps)$, $1 \leq i \leq L$, separated by interfaces $\Gamma_{ij}(\eps)$, $1 \leq i < j \leq L$ that do not intersect with each other or with the boundary $\pd \Omega$.  
\item We assume that there is a family $(\vec{v}_{\eps}, p_{\eps}, \bm{\varphi}_{\eps}, \bm{\sigma}_{\eps}, \bm{\mu}_{\eps})_{\eps > 0}$ of solutions to \eqref{multiphasetumourmodel}, which are sufficiently smooth and have an asymptotic expansion in $\eps$ in the bulk regions away from the interfaces $\{\Gamma_{ij}(\eps)\}_{1 \leq i < j \leq L}$ (the outer expansion), and another expansion in the interfacial regions close to the interfaces (the inner expansion).
\end{enumerate}
\end{assump}

\begin{remark}
In the above assumption, for a domain $\Omega \subset \R^{d}$, $d = 2,3$, we exclude the possibility of triple junction points in $\R^{2}$ and triple junction lines or quadruple junction points in $\R^{3}$.  Although the method of formally matched asymptotic analysis is able to derive certain boundary/angle conditions for the interfaces $\Gamma_{ij}$ at such a triple junction, as the case of junctions is not so relevant for tumour growth we will omit the analysis and refer the reader to \cite{Blanketal14, BronsardGarckeStoth98, GarckeNestlerStinner04, GarckeNestlerStoth98, NovickCohen00}.
\end{remark}

With the above assumptions, \eqref{multiphasetumourmodel} becomes
\begin{subequations}\label{asymp:multiphase}
\begin{align}
\div \vec{v} & = \unit\cdot \bm{U}(\bm{\varphi}, \bm{\sigma}), \label{asym:multi:div} \\
\vec{v} & = -K \nabla p + K (\nabla \bm{\varphi})^{\top}  (\bm{\mu} - \bm{N}_{,\bm{\varphi}}(\bm{\varphi}, \bm{\sigma})), \label{asym:multi:darcy} \\
\pd_{t} \bm{\varphi} + \div (\bm{\varphi} \otimes \vec{v}) & = \div ( \TT{C}(\bm{\varphi}, \bm{\sigma}) \nabla \Proj{\bm{\mu}}) + \bm{U}(\bm{\varphi}, \bm{\sigma}), \label{asym:multi:varphi} \\
\Proj{\bm{\mu}} & = - \beta \eps \Laplace \bm{\varphi} + \beta \eps^{-1} \Proj{\left (\Psi_{,\bm{\varphi}}(\bm{\varphi}) + \bm{N}_{,\bm{\varphi}}(\bm{\varphi}, \bm{\sigma}) \right )}, \label{asym:multi:mu} \\
\pd_{t} \bm{\sigma} + \div (\bm{\sigma} \otimes \vec{v}) & = \div ( \TT{D}(\bm{\varphi}, \bm{\sigma}) \nabla \bm{N}_{,\bm{\sigma}}(\bm{\varphi}, \bm{\sigma})) + \bm{S}(\bm{\varphi}, \bm{\sigma}). \label{asym:multi:sigma}
\end{align}
\end{subequations}
The idea of the method is to plug the outer and inner expansions in the model equations and solve them order by order.  In addition, we have to define a suitable region where these expansions should match up.

We will use the following notation: $\eqref{asym:multi:div}_{O}^{\alpha}$ and $\eqref{asym:multi:div}_{I}^{\alpha}$  denote the terms resulting from the order $\alpha$ outer and inner expansions of \eqref{asym:multi:div}, respectively.  For convenience, we will denote $\bm{N}_{,\bm{\sigma}}(\bm{\varphi}, \bm{\sigma})$ by the variable $\bm{\theta}$.

\subsection{Outer expansions}\label{sec:outer}
We assume that for $f_{\eps} \in \{ \vec{v}_{\eps}, \bm{\varphi}_{\eps}, \bm{\sigma}_{\eps}, \bm{\mu}_{\eps}, p_{\eps}, \bm{\theta}_{\eps} \}$, the following outer expansions hold:
\begin{align*}
f_{\eps} = f_{0} + \eps f_{1} + \eps^{2} f_{2} + \dots,
\end{align*}
where to ensure that the constraint $\bm{\varphi} \in \Gibbs$ is satisfied, we additionally assume that
\begin{align*}
\bm{\varphi}_{0} \in \Gibbs, \quad \bm{\varphi}_{k} \in \TangGibbs \quad \forall k \geq 1.
\end{align*}
Note that we can relate the expansions for $\bm{\theta}_{\eps}$ by means of Taylor's expansion: 
\begin{align}\label{outer:theta}
\bm{\theta}_{0} = \bm{N}_{,\bm{\sigma}}(\bm{\varphi}_{0}, \bm{\sigma}_{0}), \quad \bm{\theta}_{1} = \bm{N}_{,\bm{\sigma} \bm{\varphi}}(\bm{\varphi}_{0}, \bm{\sigma}_{0}) \bm{\varphi}_{1} + \bm{N}_{,\bm{\sigma} \bm{\sigma}}(\bm{\varphi}_{0}, \bm{\sigma}_{0}) \bm{\sigma}_{1}, \quad \dots.
\end{align}
To leading order $\eqref{asym:multi:mu}_{O}^{-1}$ we have 
\begin{align}\label{outer:varphi:-1}
\Proj{\Psi_{,\bm{\varphi}}(\bm{\varphi}_{0})} = \Psi_{,\bm{\varphi}}(\bm{\varphi}_{0}) -  \frac{1}{L} \sum_{i=1}^{L} \frac{\pd \Psi}{\pd \varphi_{i}}(\bm{\varphi}_{0}) \unit = \bm{0}.
\end{align}
The stable solutions to \eqref{outer:varphi:-1} are the minima of $\Psi$, that is, $\bm{\varphi}_{0} = \bm{e}_{i}$, $1 \leq i \leq L$.  Thus, to leading order the domain $\Omega$ is partitioned into $L$ regions corresponding to the stable minima of $\Psi$.  We define
\begin{align*}
\Omega_{i} = \{ \vec{x} \in \Omega : \bm{\varphi}_{0}(\vec{x}) = \bm{e}_{i} \} \quad \text{ for } 1 \leq i \leq L.
\end{align*}
Since $\nabla \bm{\varphi}_{0} = \TT{0} \in \R^{L \times d}$ is the zero tensor in the bulk regions $\Omega_{i}$, $1 \leq i \leq L$, we obtain from \eqref{asym:multi:div}, \eqref{asym:multi:darcy}, \eqref{asym:multi:varphi} and \eqref{asym:multi:sigma} to zeroth order in each bulk region:
\begin{subequations}
\begin{align}
\div \vec{v}_{0} & = \unit\cdot \bm{U}(\bm{\varphi}_{0}, \bm{\sigma}_{0}), \\
\vec{v}_{0} & = -K \nabla p_{0}, \\
-\div (\TT{C}(\bm{\varphi}_{0}, \bm{\sigma}_{0}) \nabla (\Proj{\bm{\mu}_{0}}) ) & = \bm{U}(\bm{\varphi}_{0}, \bm{\sigma}_{0}) - (\unit\cdot \bm{U}(\bm{\varphi}_{0}, \bm{\sigma}_{0}))\bm{\varphi}_{0} , \\
\pd_{t} \bm{\sigma}_{0} + \div (\bm{\sigma}_{0} \otimes \vec{v}_{0}) & = \div (\TT{D}(\bm{\varphi}_{0}, \bm{\sigma}_{0}) \nabla \bm{N}_{,\bm{\sigma}}(\bm{\varphi}_{0}, \bm{\sigma}_{0})) + \bm{S}(\bm{\varphi}_{0},  \bm{\sigma}_{0}).
\end{align}
\end{subequations}

\subsection{Inner expansions and matching conditions}
In this section we fix $1 \leq i < j \leq L$ and construct a solution that makes a transition from $\Omega_{i}$ to $\Omega_{j}$ across a smoothly evolving hypersurface $\Gamma = \Gamma_{ij}$ moving with normal velocity $\mathcal{V}$.  Let $d(\vec{x})$ denote the signed distance function to $\Gamma$, and set $z = \frac{d}{\eps}$ as the rescaled distance variable.  Here we use the convention that $d(\vec{x}) < 0$ in $\Omega_{i}$ and $d(\vec{x}) > 0$ in $\Omega_{j}$.  Thus the gradient $\nabla d$ points from $\Omega_{i}$ to $\Omega_{j}$ and we may use $\nabla d$ on $\Gamma$ as a unit normal $\vec{\nu}$.

Let $g(t,s)$ denote a parameterisation of $\Gamma$ by arclength $s$, and in a tubular neighbourhood of $\Gamma$, for smooth functions $f(\vec{x})$, we have
\begin{align*}
f(\vec{x}) = f(g(t,s) + \eps z \vec{\nu}(g(t,s))) =: F(t,s,z).
\end{align*}
In this new $(t,s,z)$-coordinate system, the following change of variables apply (see
\cite{AGG, GarckeStinner06}):
\begin{align*}
\pd_{t} f = -\frac{1}{\eps} \mathcal{V} \pd_{z} F + \text{ h.o.t.}, \quad \nabla_{x} f = \frac{1}{\eps}\pd_{z} F \vec{\nu} + \nabla_{\Gamma} F + \text{ h.o.t.},
\end{align*}
where $\nabla_{\Gamma}h$ denotes the surface gradient of $h$ on $\Gamma$ and h.o.t. denotes higher order terms with respect to $\eps$.  In particular, we have
\begin{align*}
\Laplace f  = \div_{x} (\nabla_{x} f) & = \frac{1}{\eps^{2}} \pd_{zz}F + \frac{1}{\eps}\underbrace{\div_{\Gamma} (\pd_{z}F \vec{\nu})}_{= - \kappa \pd_{z}F} + \text{ h.o.t.},
\end{align*}
where $\kappa = - \div_{\Gamma} \vec{\nu}$ is the mean curvature of $\Gamma$.  Moreover, if $\bm{f}$ is a vector-valued function with $\bm{F}(t,s,z) = \bm{f}(t,\vec{x})$ for $\vec x$ in a tubular neighbourhood of $\Gamma$, then we obtain
\begin{align*}
\pd_{t} \bm{f} & = -\frac{1}{\eps} \mathcal{V} \pd_{z} \bm{F} + \text{ h.o.t.}, \quad
\nabla_{x} \bm{f} = \frac{1}{\eps} \pd_{z} \bm{F} \otimes \vec{\nu} + \surf \bm{F} + \text{ h.o.t.}, \\
\div_{x} \bm{f} & = \frac{1}{\eps}\pd_{z} \bm{F} \cdot \vec{\nu} + \div_{\Gamma} \bm{F} + \text{ h.o.t.}
\end{align*}
We denote the variables $\bm{\varphi}_{\eps}$, $\vec{v}_{\eps}$, $p_{\eps}$, $\bm{\mu}_{\eps}$, $\bm{\sigma}_{\eps}$, $\bm{\theta}_{\eps}$ in the new coordinate system by $\bm{\Phi}_{\eps}$, $\vec{V}_{\eps}$, $P_{\eps}$, $\bm{\Upsilon}_{\eps}$, $\bm{\Sigma}_{\eps}$, $\bm{\Theta}_{\eps}$, respectively.  We further assume that they have the following inner expansions:
\begin{align*}
F_{\eps}(t, s, z) = F_{0}(t, s, z) + \eps F_{1}(t, s, z) + \eps^{2} F_{2}(t,s,z) + \dots,
\end{align*}
for $F_{\eps} \in \{ \bm{\Phi}_{\eps}, \vec{V}_{\eps}, P_{\eps}, \bm{\Upsilon}_{\eps}, \bm{\Sigma}_{\eps}, \bm{\Theta}_{\eps} \}$ such that
\begin{align*}
\bm{\Phi}_{0} \in \Gibbs, \quad \bm{\Phi}_{k} \in \TangGibbs \quad \forall k \geq 1.
\end{align*}   
to ensure that the constraint $\bm{\varphi} \in \Gibbs$ is satisfied.  Analogous to \eqref{outer:theta}, by Taylor's expansion, we have 
\begin{align}\label{inner:theta:relations}
\bm{\Theta}_{0} = \bm{N}_{,\bm{\sigma}}(\bm{\Phi}_{0}, \bm{\Sigma}_{0}), \quad \bm{\Theta}_{1} = \bm{N}_{,\bm{\sigma} \bm{\varphi}}(\bm{\Phi}_{0}, \bm{\Sigma}_{0}) \bm{\Phi}_{1} + \bm{N}_{,\bm{\sigma} \bm{\sigma}}(\bm{\Phi}_{0}, \bm{\Sigma}_{0}) \bm{\Sigma}_{1}, \quad \dots.
\end{align}
In order to match the inner expansions valid in the interfacial region to the outer expansions of Section \ref{sec:outer}, we employ the matching conditions, see \cite{GarckeStinner06}:
\begin{align}
\label{MatchingCond1}
\lim_{z \to \pm l} F_{0}(t,s,z) &= f_{0}^{\pm}(t,\vec{x}), \\
\label{MatchingCond2}
\lim_{z \to \pm l} \pd_{z}F_{0}(t,s,z) &= 0 ,\\
\label{MatchingCond3}
\lim_{z \to \pm l} \pd_{z} F_{1}(t,s,z) &= \nabla f_{0}^{\pm}(t,\vec{x}) \cdot \vec{\nu},
\end{align}
where $f_{0}^{\pm}(t, \vec{x})= \lim_{\delta \to 0} f_{0}(t, \vec{x} \pm \delta \vec{\nu}(\vec{x}))$ for $\vec{x} \in \Gamma$ and $\delta > 0$.  Here we use the convention that for a vectorial quantity $\bm{f}_{0}$, the right hand side of \eqref{MatchingCond3} reads as $(\nabla \bm{f}_{0}^{\pm}) \vec{\nu}$.  Moreover, we use the following notation:  Let $\delta > 0$ and for $\vec{x} \in \Gamma$ with $x - \delta \vec{\nu}(\vec{x}) \in \Omega_{i}$ and $x + \delta \vec{\nu}(\vec{x}) \in \Omega_{j}$, we denote the jump of a scalar quantity $f$ across the interface by
\begin{align}\label{defn:jump}
\jump{f}{i}{j} = \lim_{\delta \to 0} \left( f(t,\vec{x} + \delta \vec{\nu}(\vec{x})) -  f(t, \vec{x} - \delta \vec{\nu}(\vec{x})) \right ).
\end{align}
For a vectorial quantity $\bm{f} \in \R^{k}$, we define
\begin{align*}
\jump{\bm{f}}{i}{j} = \lim_{\delta \to 0} \left ( \bm{f}(t,\vec{x} + \delta \vec{\nu}) - \bm{f}(t,\vec{x} - \delta \vec{\nu}) \right ) = ( \jump{f_{1}}{i}{j}, \dots \jump{f_{k}}{i}{j})^{\top}.
\end{align*}
It will be useful to compute the expansion for the term $\div (\TT{D}(\bm{\varphi}, \bm{\sigma}) \nabla \bm{N}_{,\bm{\sigma}}(\bm{\varphi}, \bm{\sigma}))$ as follows:  For fixed $1 \leq l,n \leq d$ and $1 \leq k,m \leq M$, we find from the above change of variables formula 
\begin{align*}
\pd_{x_{l}} \left ( \left ( \TT{D}(\bm{\varphi}, \bm{\sigma})\right )_{klmn} \pd_{x_{n}} \theta_{m}\right ) & = \frac{1}{\eps} \pd_{z} \left ( \left ( \TT{D}(\bm{\Phi}, \bm{\Sigma}) \right )_{klmn} \left ( \frac{1}{\eps} \pd_{z} \Theta_{m} \nu_{n} + \underline{D}_{n} \Theta_{m} \right ) \right ) \nu_{l} \\
& \quad + \underline{D}_{l} \left ( \left ( \TT{D}(\bm{\Phi}, \bm{\Sigma}) \right )_{klmn} \left ( \frac{1}{\eps} \pd_{z} \Theta_{m} \nu_{n} + \underline{D}_{l} \Theta_{m} \right ) \right ) + \text{ h.o.t.},
\end{align*}
where $\underline{D}_{n}$ denotes the $n$th component of the surface gradient, i.e., $\surf f = (\underline{D}_{1} f, \dots, \underline{D}_{d} f)^{\top}$.  Plugging in the expansion (where we use $\Theta_{m,q}$ to denote the $q$th term of the inner expansion for the $m$th component of $\bm{\Theta}$) and using Taylor's theorem we have
\begin{equation}\label{divergence:sigma:expansion}
\begin{aligned}
& \pd_{x_{l}} \left ( \left ( \TT{D}(\bm{\varphi}, \bm{\sigma})\right )_{klmn} \pd_{x_{n}} \theta_{m} \right ) \\
& = \frac{1}{\eps^{2}} \pd_{z} \left ( \left ( \TT{D}(\bm{\Phi}_{0}, \bm{\Sigma}_{0}) \right )_{klmn} \pd_{z} \Theta_{m,0} \right ) \nu_{l}\nu_{n}  + \frac{1}{\eps} \pd_{z} \left ( \left ( \TT{D}(\bm{\Phi}_{0}, \bm{\Sigma}_{0}) \right )_{klmn}  \pd_{z} \Theta_{m,1}  \right ) \nu_{l} \nu_{n} \\
& \quad + \frac{1}{\eps} \left [ \pd_{z} \left ( \left ( \TT{D}(\bm{\Phi}_{0}, \bm{\Sigma}_{0}) \right )_{klmn} \underline{D}_{n} \Theta_{m,0} \nu_{l} \right ) + \underline{D}_{l} \left ( \left ( \TT{D}(\bm{\Phi}_{0}, \bm{\Sigma}_{0}) \right )_{klmn} \pd_{z} \Theta_{m,0} \nu_{n} \right )\right ] \\
& \quad  +  \frac{1}{\eps} \pd_{z} \left ( \left ( \sum_{s=1}^{L}  \frac{\pd (\TT{D}(\bm{t}, \bm{w}))_{klmn}}{\pd t_{s}}  \Phi_{s,1} + \sum_{s=1}^{M} \frac{\pd (\TT{D}(\bm{t}, \bm{w}))_{klmn}}{\pd w_{s}}  \Sigma_{s,1} \right ) \left ( \pd_{z} \Theta_{m,0} \nu_{n}  \right ) \right ) \nu_{l} \\
& \quad  + \text{ h.o.t.},
\end{aligned}
\end{equation}
where we evaluate the last term at $\bm{t} = \bm{\Phi}_{0}$ and $\bm{w} = \bm{\Sigma}_{0}$.  Using that
\begin{align*}
\div (\TT{D}(\bm{\varphi}, \bm{\sigma}) \nabla \bm{N}_{,\bm{\sigma}}(\bm{\varphi}, \bm{\sigma})) = \left (\sum_{l,n=1}^{d} \sum_{m=1}^{M} \pd_{x_{l}} \left ( \left ( \TT{D}(\bm{\varphi}, \bm{\sigma})\right )_{klmn} \pd_{x_{n}} \theta_{m} \right ) \right )_{1 \leq k \leq L},
\end{align*}
we obtain the expansion for the term $\div (\TT{D}(\bm{\varphi}, \bm{\sigma}) \nabla \bm{N}_{,\bm{\sigma}}(\bm{\varphi}, \bm{\sigma}))$.  One can also derive a similar expansion for $\div (\TT{C}(\bm{\varphi}, \bm{\sigma}) \nabla (\Proj{\bm{\mu}}))$.
\subsubsection{Expansions to leading order}
To leading order $\eqref{asym:multi:mu}_{I}^{-1}$ we obtain
\begin{align}
\label{inner:mu:-1}
\bm{0} = -\pd_{z} ( \pd_{z} \bm{\Phi}_{0} \otimes \vec{\nu}) \cdot \vec{\nu} + \Proj{\Psi_{,\bm{\varphi}}(\bm{\Phi}_{0})} = -\pd_{zz} \bm{\Phi}_{0}  + \Psi_{,\bm{\varphi}}(\bm{\Phi}_{0}) - \frac{1}{L} (\Psi_{,\bm{\varphi}}(\bm{\Phi}_{0}) \cdot \unit) \unit.
\end{align}
This is a second order differential equation for $\bm{\Phi}_{0}(t,s,\cdot)$, and for each $s$ we solve the above ordinary differential equation (in $z$) with the boundary conditions
\begin{align}\label{ODE:bdycond}
\lim_{z \to \infty}\bm{\Phi}_{0}(t,s,z) = \bm{e}_{j}, \quad \lim_{z \to -\infty} \bm{\Phi}_{0}(t,s,z) = \bm{e}_{i},
\end{align}
which then yields a vector-valued function that connects $\bm{e}_{i}$ to $\bm{e}_{j}$ and hence the values of the phase fields in $\Omega_{i}$ and $\Omega_{j}$.  By the assumptions satisfied by $\Psi$ in Assumption \ref{assump:Asymptotics}, it is shown in Sternberg \cite[Lemma, p. 801]{Sternberg} that for any $\bm{u} \in \Gibbs$, there exists a curve $\gamma_{\bm{u}} : [-1,1] \to \Gibbs$ such that $\gamma_{\bm{u}}(-1) = \bm{e}_{i}$ and $\gamma_{\bm{u}}(1) = \bm{u}$ and the Lipschitz continuous function
\begin{align*}
q(\bm{u}) = \int_{-1}^{1} \sqrt{\Psi(\gamma_{\bm{u}}(t))} \abs{\gamma_{\bm{u}}'(t)} \dt \quad  \text{ satisfies }  \quad \abs{\nabla q(\bm{u})} = \sqrt{\Psi(\bm{u})} \quad \text{ for a.e. } \bm{u} \in \Gibbs.
\end{align*}
Let us define $\beta : (-\infty, \infty) \to (-1,1)$ as the monotone solution of
\begin{align*}
\beta'(z) = \frac{\sqrt{2 \Psi(\gamma_{\bm{e}_{j}}(\beta(z)))}}{\abs{\gamma'_{\bm{e}_{j}}(\beta(z))}}, \quad \beta(0) = 0
\end{align*}
and then set
\begin{align*}
\overline{\bm{\Phi}}(z) = \gamma_{\bm{e}_{j}}(\beta(z)).
\end{align*}
Then, it holds that $\abs{\overline{\bm{\Phi}}'(z)} = \sqrt{2 \Psi(\overline{\bm{\Phi}}(z))}$ and
\begin{equation}\label{defn:surfacetension}
\begin{aligned}
\sqrt{2} q(\bm{e}_{j}) & = \sqrt{2} \int_{-\infty}^{\infty} \sqrt{\Psi(\gamma_{\bm{e}_{j}}(\beta(z)))} \abs{\gamma'_{\bm{e}_{j}}(\beta(z))} \beta'(z) \dz \\
& = 2 \int_{-\infty}^{\infty} \Psi(\overline{\bm{\Phi}}(z)) \dz = \int_{-\infty}^{\infty} \Psi(\overline{\bm{\Phi}}(z)) + \frac{1}{2}\abs{\overline{\bm{\Phi}}'(z)}^{2} \dz.
\end{aligned}
\end{equation}
It follows that $\overline{\bm{\Phi}}$ is a candidate solution to the following problem
\begin{align*}
\inf_{\bm{\zeta} \in \Gibbs, \bm{\zeta}(-\infty) = \bm{e}_{i}, \bm{\zeta}(\infty) = \bm{e}_{j}} \int_{-\infty}^{\infty} \Psi(\bm{\zeta}(\tau)) + \frac{1}{2} \abs{\bm{\zeta}'(\tau)}^{2} \dd \tau.
\end{align*}
Computing its Euler--Lagrange equations (subject to the constraint $\bm{\zeta} \in \Gibbs$) yields that
\begin{align*}
\overline{\bm{\Phi}}''(z) = \Proj{\Psi_{,\bm{\varphi}}(\overline{\bm{\Phi}}(z))} = \Psi_{,\bm{\varphi}}(\overline{\bm{\Phi}}(z)) - \frac{1}{L} (\Psi_{,\bm{\varphi}}(\overline{\bm{\Phi}}(z)) \cdot \unit) \unit,
\end{align*}
and if we consider
\begin{align*}
\bm{\Phi}_{0}(t,s,z) = \overline{\bm{\Phi}}(z),
\end{align*}
then $\bm{\Phi}_{0}(z)$ satisfies \eqref{inner:mu:-1} and \eqref{ODE:bdycond}.  Furthermore, multiplying \eqref{inner:mu:-1} with $\bm{\Phi}'_{0} \in \TangGibbs$, integrating with respect to $z$ and applying the matching condition \eqref{MatchingCond1} to $\bm{\Phi}_{0}$, leads to the so-called equipartition of energy:
\begin{align*}
\Psi(\bm{\Phi}_{0}(z)) = \frac{1}{2} \abs{\bm{\Phi}'_{0}(z)}^{2} \quad \forall z \in \R,
\end{align*}
and we define the surface energy $\gamma_{ij}$ to be
\begin{align}\label{defn:surfaceenergy}
\gamma_{ij} = \sqrt{2} q(\bm{e}_{j}) = \int_{-\infty}^{\infty} \Psi( \bm{\Phi}_{0}(z)) + \frac{1}{2} \abs{\bm{\Phi}'_{0}(z)}^{2} \dz = \int_{-\infty}^{\infty} \abs{\bm{\Phi}'_{0}(z)}^{2} \dz.
\end{align}
Next, $\eqref{asym:multi:div}_{I}^{-1}$ gives
\begin{align}
\label{inner:div:-1}
\pd_{z} \vec{V}_{0} \cdot \vec{\nu} = \pd_{z} (\vec{V}_{0} \cdot \vec{\nu}) = 0.
\end{align}
Integrating with respect to $z$ and using the matching condition (\ref{MatchingCond1}) applied to $\vec{v}_{0}$ leads to
\begin{align}
\label{jump:velo:zero}
\jump{\vec{v}_{0}}{i}{j} \cdot \vec{\nu} = 0.
\end{align}
{From} $\eqref{asym:multi:sigma}_{I}^{-2}$ and \eqref{divergence:sigma:expansion} 
we have
\begin{align}
\label{inner:sigma:-2}
\sum_{l,n=1}^{d} \sum_{m=1}^{M} \pd_{z} \left ( ( \TT{D}(\bm{\Phi}_{0}, \bm{\Sigma}_{0}))_{klmn} \pd_{z} \Theta_{m,0} \right ) \nu_{l} \nu_{n} = \bm{0} \in \R^{M}.
\end{align}
Multiplying by $\Theta_{k,0}$ and summing from $k = 1$ to $M$ and then integrating with respect to $z$, we obtain from integration by parts and the matching condition (\ref{MatchingCond2}) applied to $\bm{\Theta}_{0}$ that
\begin{align*}
0 & = -\sum_{k,m=1}^{M} \sum_{l,n=1}^{d} \int_{-\infty}^{\infty} \left ( \TT{D}(\bm{\Phi}_{0}, \bm{\Sigma}_{0}) \right )_{klmn} \pd_{z} \Theta_{m,0} \nu_{n} \pd_{z} \Theta_{k,0} \nu_{l} \dz \\
& = -\int_{-\infty}^{\infty} \left ( \TT{D}(\bm{\Phi}_{0}, \bm{\Sigma}_{0}) (\pd_{z} \bm{\Theta}_{0} \otimes \vec{\nu}) \right ) : (\pd_{z} \bm{\Theta}_{0} \otimes \vec{\nu}) \dz.
\end{align*}
The strict positivity of $\TT{D}$ (in the sense of Assumption \ref{assump:Asymptotics}) yields
\begin{align}
\label{inner:Theta:constant}
\pd_{z} \bm{\Theta}_{0}(t,s,z) = \bm{0} \in \R^{M} \quad \forall z \in \R,
\end{align}
i.e., $\bm{\Theta}_{0}$ is independent of $z$.  Moreover, integrating (\ref{inner:Theta:constant}) with respect to $z$ and using the matching condition (\ref{MatchingCond1}) applied to $\bm{\Theta}_{0}$ gives
\begin{align}
\label{inner:jump:Nsigma}
\jump{\bm{\theta}_{0}}{i}{j} = \jump{\bm{N}_{,\bm{\sigma}}(\bm{\varphi}_{0}, \bm{\sigma}_{0})}{i}{j} = \bm{0} \in \R^{M}.
\end{align} 
Meanwhile, from $\eqref{asym:multi:varphi}_{I}^{-2}$, we have
\begin{align}
\label{inner:varphi:-2}
\R^{L} \ni \bm{0} = \left (\sum_{l,n=1}^{d} \sum_{m=1}^{L} \pd_{z} \left ( \left ( \TT{C}(\bm{\Phi}_{0}, \bm{\Sigma}_{0}) \right )_{klmn} \pd_{z} \Proj{\Upsilon_{m,0}} \right ) \nu_{l} \nu_{m} \right )_{1 \leq k \leq L}.
\end{align}
Multiplying by $\Proj{\Upsilon_{k,0}} \in \TangGibbs$ and summing from $k = 1$ to $L$, an analysis  similar to the above for $\Theta_{0}$ using the assumptions on $\TT{C}$ yields that
\begin{align}
\label{inner:mu:constant}
\pd_{z} \Proj{\bm{\Upsilon}_{0}}(t,s,z) = \bm{0} \in \R^{L} \quad \forall z \in \R \quad \Rightarrow \quad
\jump{\Proj{\bm{\mu}_{0}}}{i}{j} = \bm{0} \in \R^{L}.
\end{align}
Lastly, $\eqref{asym:multi:darcy}_{I}^{-1}$ yields 
\begin{align}
\label{inner:darcy:-1}
\bm{0} = -\pd_{z} P_{0} \vec{\nu} + \left ( \bm{\Phi}'_{0} \cdot \left ( \bm{\Upsilon}_{0} - \bm{N}_{, \bm{\varphi}}(\bm{\Phi}_{0}, \bm{\Sigma}_{0}) \right ) \right ) \vec{\nu}.
\end{align}
Taking the scalar product with $\vec{\nu}$ and then integrating with respect to $z$ leads to
\begin{align}\label{inner:darcy:-1:intergated}
\jump{p_{0}}{i}{j} = \int_{-\infty}^{\infty}  \bm{\Phi}'_{0} \cdot ( \bm{\Upsilon}_{0} - \bm{N}_{,\bm{\varphi}}(\bm{\Phi}_{0}, \bm{\Sigma}_{0}) ) \dz = \int_{-\infty}^{\infty}  \bm{\Phi}'_{0} \cdot ( \Proj{\bm{\Upsilon}_{0}} - \bm{N}_{,\bm{\varphi}}(\bm{\Phi}_{0}, \bm{\Sigma}_{0}) ) \dz ,
\end{align}
where we used the matching condition (\ref{MatchingCond1}) applied to $P_{0}$ and the fact that $\bm{\Phi}_{0}' \in \TangGibbs$ and so $\bm{\Phi}_{0}' \cdot \bm{\Upsilon}_{0} = \bm{\Phi}_{0}' \cdot \Proj{\bm{\Upsilon}_{0}}$.  Thanks to the fact that $\Proj{\bm{\Upsilon}_{0}}$ is independent of $z$, we find that
\begin{align}\label{intProjMu0Phi0'}
\int_{-\infty}^{\infty} \bm{\Phi}'_{0} \cdot \Proj{\bm{\Upsilon}_{0}} \dz = \int_{-\infty}^{\infty} \pd_{z} \left ( \Proj{\bm{\Upsilon}_{0}} \cdot  \bm{\Phi}_{0} \right ) \dz = \jump{\Proj{\bm{\mu}_{0}} \cdot \bm{\varphi}_{0}}{i}{j}.
\end{align}
Recalling $\bm{\Theta}_{0} = N_{,\bm{\sigma}}(\bm{\Phi}_{0}, \bm{\Sigma}_{0})$ from \eqref{inner:theta:relations}, we have
\begin{equation}\label{intPhi0'Xi0}
\begin{aligned}
\int_{-\infty}^{\infty}  \bm{\Phi}'_{0} \cdot \bm{N}_{,\bm{\varphi}}(\bm{\Phi}_{0}, \bm{\Sigma}_{0})  \dz & = \jump{N(\bm{\Phi}_{0}, \bm{\Sigma}_{0})}{i}{j} - \int_{-\infty}^{\infty} \pd_{z} \bm{\Sigma}_{0} \cdot \bm{\Theta}_{0} \dx \\
& = \jump{N(\bm{\varphi}_{0}, \bm{\sigma}_{0})}{i}{j}  - \jump{ \bm{\sigma}_{0}}{i}{j} \cdot \bm{N}_{,\bm{\sigma}}(\bm{\varphi}_{0}, \bm{\sigma}_{0}), 
\end{aligned}
\end{equation}
and so \eqref{inner:darcy:-1:intergated} becomes
\begin{align}
\label{inner:darcy:-1:int:simplified}
\jump{p_{0} - \Proj{\bm{\mu}_{0}} \cdot \bm{\varphi}_{0} - N(\bm{\varphi}_{0}, \bm{\sigma}_{0}) + \bm{N}_{,\bm{\sigma}}(\bm{\varphi}_{0}, \bm{\sigma}_{0}) \cdot \bm{\sigma}_{0}}{i}{j} = 0.
\end{align}

\subsubsection{Expansions to first order}
To first order, we find from $\eqref{asym:multi:mu}_{I}^{0}$
\begin{equation}
\label{inner:mu:0}
\begin{aligned}
\Proj{\bm{\Upsilon}_{0}} & =  \beta \Proj{\Psi_{,\bm{\varphi} \bm{\varphi}}(\bm{\Phi}_{0}) \bm{\Phi}_{1}} + \Proj{\bm{N}_{,\bm{\varphi}}(\bm{\Phi}_{0}, \bm{\Sigma}_{0})}  - \beta \pd_{zz} \bm{\Phi}_{1} - \beta \div_{\Gamma} \left ( \bm{\Phi}'_{0} \otimes \vec{\nu} \right ),
\end{aligned}
\end{equation}
where we used that $\bm{\Phi}_{0}$ is only a function of $z$ and so $\surf \bm{\Phi}_{0} = \TT{0}$ is the zero tensor.  As $\bm{\Phi}_{0}' \in \TangGibbs$, $\Proj{\bm{f}} \cdot \bm{\Phi}_{0}' = \bm{f} \cdot \bm{\Phi}_{0}'$, and hence after multiplying \eqref{inner:mu:0} with $\bm{\Phi}'_{0}$ and integrating over $z$ we obtain
\begin{equation}\label{inner:mu:0:integrated}
\begin{aligned}
& \int_{-\infty}^{\infty} \frac{1}{\beta} (\Proj{\bm{\Upsilon}_{0}} - \bm{N}_{,\bm{\varphi}}(\bm{\Phi}_{0}, \bm{\Sigma}_{0})) \cdot \bm{\Phi}'_{0} \dz \\
& \quad = \int_{-\infty}^{\infty} (\Psi_{,\bm{\varphi} \bm{\varphi}}(\bm{\Phi}_{0}) \bm{\Phi}_{1}) \cdot \bm{\Phi}'_{0} - \pd_{zz} \bm{\Phi}_{1} \cdot \bm{\Phi}'_{0} + \kappa \abs{\bm{\Phi}'_{0}}^{2} \dz,
\end{aligned}
\end{equation}
where $\kappa = - \div_{\Gamma} \vec{\nu}$ is the mean curvature of $\Gamma$ and we have used that $\div_{\Gamma}(\bm{\Phi}'_{0} \otimes \vec{\nu}) = \bm{\Phi}'_{0} \div_{\Gamma} \vec{\nu} = -\kappa \bm{\Phi}'_{0}$.  Due to the symmetry of the tensor $\Psi_{,\bm{\varphi} \bm{\varphi}}$ it holds that \begin{align}\label{solvability:part:1}
(\Psi_{,\bm{\varphi} \bm{\varphi}} (\bm{\Phi}_{0})) \bm{\Phi}_{1} \cdot \bm{\Phi}'_{0} = (\Psi_{,\bm{\varphi} \bm{\varphi}} (\bm{\Phi}_{0}))  \bm{\Phi}'_{0} \cdot \bm{\Phi}_{1} = \pd_{z} \left ( \Psi_{,\bm{\varphi}}(\bm{\Phi}_{0}) \right ) \cdot \bm{\Phi}_{1}.
\end{align}
Then by integrating by parts we obtain from \eqref{inner:mu:-1} and the matching conditions \eqref{MatchingCond1}, \eqref{MatchingCond2} applied to $\bm{\Phi}_{0}$ that
\begin{align*}
\int_{-\infty}^{\infty} \pd_{z} (\Psi_{, \bm{\varphi}}(\bm{\Phi}_{0})) \cdot \bm{\Phi}_{1} - \pd_{zz} \bm{\Phi}_{1} \cdot \bm{\Phi}'_{0} \dz & = \int_{-\infty}^{\infty} (\Psi_{,\bm{\varphi}}(\bm{\Phi}_{0}) - \bm{\Phi}''_{0}) \cdot \pd_{z} \bm{\Phi}_{1} \dz \\
& \quad + \jump{\Psi_{,\bm{\varphi}}(\bm{\Phi}_{0}) \cdot \bm{\Phi}_{1} + \bm{\Phi}'_{0} \cdot \pd_{z} \bm{\Phi}_{1}}{z=-\infty}{z=\infty} = 0.
\end{align*}
Using \eqref{defn:surfaceenergy}, \eqref{intProjMu0Phi0'}, and \eqref{intPhi0'Xi0}, we obtain from \eqref{inner:mu:0:integrated} the following solvability condition for $\bm{\Phi}_{1}$:
\begin{align}\label{solvabilitycondition}
\beta \gamma_{ij} \kappa  = \jump{\Proj{\bm{\mu}_{0}} \cdot \bm{\varphi}_{0}}{i}{j} - \jump{N(\bm{\varphi}_{0}, \bm{\sigma}_{0})}{i}{j} + \jump{\bm{\sigma}_{0}}{i}{j} \cdot \bm{N}_{,\bm{\sigma}}(\bm{\varphi}_{0}, \bm{\sigma}_{0}) = 
\jump{p_{0}}{i}{j}.
\end{align}
Next, thanks to the fact that $\pd_{z} \Proj{\bm{\Upsilon}_{0}} = \bm{0}$, to first order we obtain from $\eqref{asym:multi:varphi}_{I}^{-1}$
\begin{equation}\label{inner:varphi:-1}
\begin{aligned}
-\mathcal{V} \bm{\Phi}'_{0} + \pd_{z} ( \bm{\Phi}_{0} (\vec{v}_{0} \cdot \vec{\nu}) ) = \pd_{z} \left ( \left ( \TT{C}(\bm{\Phi}_{0}, \bm{\Sigma}_{0}) \right ) \left [ ( \pd_{z} \Proj{\bm{\Upsilon}_{1}} \otimes \vec{\nu}) + \surf \Proj{\bm{\Upsilon}_{0}} \right ]  \right ) \vec{\nu}.
\end{aligned}
\end{equation}
We note that by the matching condition (\ref{MatchingCond3}) applied to $\Proj{\bm{\Upsilon}_{1}}$, we have $\lim_{z \to \pm \infty} \pd_{z} \Proj{\bm{\Upsilon}_{1}} \otimes \vec{\nu} = \nabla (\Proj{\bm{\mu}_{0}^{\pm}}) \vec{\nu}$, and hence
\begin{align*}
( \pd_{z} \Proj{\bm{\Upsilon}_{1}} \otimes \vec{\nu}) + \surf \Proj{\Upsilon_{0}} \to \begin{cases}
 \nabla \Proj{\bm{\mu}_{0}^{+}} & \text{ for } z \to \infty, \\
 \nabla \Proj{\bm{\mu}_{0}^{-}} & \text{ for } z \to -\infty.
 \end{cases}
\end{align*}
{From} (\ref{inner:div:-1}), $\vec{v}_{0} \cdot \vec{\nu}$ is independent of $z$, so integrating (\ref{inner:varphi:-1}) with respect to $z$ and applying the matching condition (\ref{MatchingCond1}) to $\bm{\Phi}_{0}$ and (\ref{MatchingCond3}) to $\pd_{z} \Proj{\bm{\Upsilon}_{1}}$ gives
\begin{equation}
\label{interfacelaw:varphi:equ}
\begin{aligned}
\left ( -\mathcal{V} + \vec{v}_{0} \cdot \vec{\nu} \right ) \jump{\bm{\varphi}_{0}}{i}{j} & =  \jump{ \left ( \TT{C}(\bm{\varphi}_{0}, \bm{\sigma}_{0}) \right ) \nabla \left ( \Proj{\bm{\mu}_{0}} \right ) }{i}{j}\vec{\nu}.
\end{aligned}
\end{equation}
Similar, thanks to the fact that $\pd_{z} \bm{\Theta}_{0} = \bm{0}$, we obtain from $\eqref{asym:multi:sigma}_{I}^{-1}$
\begin{equation}
\label{inner:sigma:-1}
\begin{aligned}
-\mathcal{V} \pd_{z} \bm{\Sigma}_{0} + \pd_{z} (\bm{\Sigma}_{0} (\vec{v}_{0} \cdot \vec{\nu})) = \pd_{z} \left ( \left ( \TT{D}(\bm{\Phi}_{0}, \bm{\Sigma}_{0}) \right )  \left [ \left ( \pd_{z} \bm{\Theta}_{1} \otimes \vec{\nu} \right ) + \surf \bm{\Theta}_{0} \right ] \right ) \vec{\nu},
\end{aligned}
\end{equation}
and upon integrating with respect to $z$ we obtain
\begin{align}
\label{interfacelaw:sigma:equ}
\left ( -\mathcal{V} + \vec{v}_{0} \cdot \vec{\nu} \right ) \jump{\bm{\sigma}_{0}}{i}{j} = \jump{ \TT{D}(\bm{\varphi}_{0}, \bm{\sigma}_{0}) \nabla \bm{N}_{,\bm{\sigma}}(\bm{\varphi}_{0}, \bm{\sigma}_{0}) }{i}{j}\vec{\nu}.
\end{align}
In summary, we obtain the following sharp interface model:  In the bulk domains $\Omega_{k} = \{\bm{\varphi}_{0} = \bm{e}_{k}\}$, $1 \leq k \leq L$,
\begin{subequations}\label{Sharpinterface:bulk}
\begin{align}
\div \vec{v}_{0} & = \unit\cdot \bm{U}(\bm{\varphi}_{0}, \bm{\sigma}_{0}), \\
\vec{v}_{0} & = -K \nabla p_{0}, \\
-\div (\TT{C}(\bm{\varphi}_{0}, \bm{\sigma}_{0}) \nabla (\Proj{\bm{\mu}_{0}})) & = \bm{U}(\bm{\varphi}_{0},  \bm{\sigma}_{0}) -  \unit\cdot \bm{U}(\bm{\varphi}_{0}, \bm{\sigma}_{0})\bm{\varphi}_{0}, \label{SI:mu} \\
\pd_{t} \bm{\sigma}_{0} + \div (\bm{\sigma}_{0} \otimes \vec{v}_{0}) & = \div (\TT{D}(\bm{\varphi}_{0}, \bm{\sigma}_{0}) \nabla \bm{N}_{,\bm{\sigma}}(\bm{\varphi}_{0}, \bm{\sigma}_{0})) + \bm{S}(\bm{\varphi}_{0}, \bm{\sigma}_{0}), \label{SI:sigma}
\end{align}
\end{subequations}
and on the free boundaries $\Gamma_{ij} = \pd \Omega_{i} \cap \pd \Omega_{j}$, $1 \leq i < j \leq L$, with unit normal $\vec{\nu}$ pointing from $\Omega_{i}$ to $\Omega_{j}$, 
\begin{subequations}\label{SharpInterface:freebdy}
\begin{align}
\jump{\vec{v}_{0}}{i}{j} \cdot \vec{\nu} = 0, \quad \jump{p_{0}}{i}{j} & = \beta \gamma_{ij} \kappa, \label{jump:p} \\
\jump{\Proj{\bm{\mu}_{0}}}{i}{j} & = \bm{0}, \label{jump:bmmu} \\
\jump{\bm{N}_{,\bm{\sigma}}(\bm{\varphi}_{0}, \bm{\sigma}_{0})}{i}{j} & = \bm{0}, \label{jump:bmNsigma} \\
\beta \gamma_{ij} \kappa & = \jump{\Proj{\bm{\mu}_{0}} \cdot \bm{\varphi}_{0} - N(\bm{\varphi}_{0}, \bm{\sigma}_{0}) + \bm{\sigma}_{0}\cdot \bm{N}_{,\bm{\sigma}}(\bm{\varphi}_{0}, \bm{\sigma}_{0})}{i}{j}, \label{SI:solvabilitycond} \\
\left ( -\mathcal{V} + \vec{v}_{0} \cdot \vec{\nu} \right ) (\bm{e}_{j} - \bm{e}_{i}) & =  \jump{ \TT{C}(\bm{\varphi}_{0}, \bm{\sigma}_{0}) \nabla (\Proj{\bm{\mu}_{0}})}{i}{j}\vec{\nu}, \label{Velo:mu} \\
\left ( -\mathcal{V} + \vec{v}_{0} \cdot \vec{\nu} \right ) \jump{\bm{\sigma}_{0}}{i}{j} & = \jump{ \TT{D}(\bm{\varphi}_{0}, \bm{\sigma}_{0}) \nabla \bm{N}_{,\bm{\sigma}}(\bm{\varphi}_{0}, \bm{\sigma}_{0})  }{i}{j}\vec{\nu}, \label{Velo:sigma}
\end{align}
\end{subequations}
where the surface energy $\gamma_{ij}$ is defined in \eqref{defn:surfaceenergy}.

\subsection{Sharp interface limit for a two-component tumour model}
We now sketch the argument to recover the sharp interface model \cite[Equation (3.49)]{GLSS} from \eqref{Sharpinterface:bulk}-\eqref{SharpInterface:freebdy} for a two-component model of host cells and tumour cells, along with a single nutrient species, recall also Section~\ref{sec:27}. 
Dropping the subscript $0$ from \eqref{Sharpinterface:bulk}-\eqref{SharpInterface:freebdy}, we consider \eqref{eq:sec27} with
\begin{align*}
N(\bm{\varphi}, \sigma) = \frac{\chi_{\sigma}}{2} \abs{\sigma}^{2} + 2 \chi_{\varphi} \sigma \varphi_{1},\quad
\Psi(\bvarphi) = 4 \varphi_{1}^{2} \varphi_{2}^{2} 
\quad\Rightarrow\quad
\tilde{\Psi}(\tilde{\varphi})= \tfrac{1}{4}(1-\tilde{\varphi}^{2})^{2}\,,
\end{align*}
along with the scalar mobility \eqref{eq:Dn}, and 
the second order mobility tensor defined in \eqref{twophasemodel:mobility}.  Setting $\Omega_{H} = \Omega_{1} = \{ \bm{\varphi} = (1,0)^{\top} \} = \{\tilde{\varphi} = -1\}$ and $\Omega_{T} = \Omega_{2} = \{ \bm{\varphi} = (0,1)^{\top} \} = \{ \tilde{\varphi} = 1 \}$, we obtain
\begin{subequations}
\begin{alignat}{3}
\vec{v} = -K \nabla p, \quad \div \vec{v} & = \overline{\rho}_{1}^{-1} \tilde{\mathcal{U}}_{1}(\tilde{\varphi}, \sigma) + \overline{\rho}_{2}^{-1} \tilde{\mathcal{U}}_{2}(\tilde{\varphi}, \sigma) && \text{ in } \Omega_{H} \cup \Omega_{T}, \\
\pd_{t} \sigma + \div (\sigma \vec{v}) & = \chi_{\sigma}\, \div(n(\tilde\varphi) \nabla \sigma) + \tilde{S}(\tilde\varphi,\sigma) && \text{ in } \Omega_{H} \cup \Omega_{T}, \\
m(+1) \Laplace \mu & = 2 \overline{\rho}_{1}^{-1} \tilde{\mathcal{U}}_{1}(+1, \sigma) && \text{ in } \Omega_{T}, \\
-m(-1) \Laplace \mu & = 2 \overline{\rho}_{2}^{-1} \tilde{\mathcal{U}}_{2}(-1, \sigma) && \text{ in } \Omega_{H}.
\end{alignat}
\end{subequations}

To obtain the free  boundary conditions on $\Gamma = \Gamma_{12}$, let $\overline{\bm{\Phi}}(z)$ denote the solution  to \eqref{inner:mu:-1} that connects $(0,1)^{\top}$ to $(1,0)^{\top}$, and let $\phi = \overline{\bm{\Phi}}_{2} - \overline{\bm{\Phi}}_{1}$ denote the difference between the second and first component of $\overline{\bm{\Phi}}$.  Then, it holds that $\overline{\bm{\Phi}}_{1} = \frac{1 - \phi}{2}$ and $\overline{\bm{\Phi}}_{2} = \frac{1 + \phi}{2}$ with $\lim_{z \to - \infty} \phi(z) = -1 $ and $\lim_{z \to + \infty} \phi(z) = 1$.  We now derive the ODE which is satisfied by $\phi$.  Taking the difference between the second and first component of \eqref{inner:mu:-1} (noting that the last term on the right-hand side of \eqref{inner:mu:-1} will not contribute to this difference), it holds that $\phi$ satisfies
\begin{align*}
-\phi''(z) + \phi^{3}(z) - \phi(z) = - \phi''(z) + \tilde{\Psi}'(\phi(z)) = 0.
\end{align*}
The unique solution $\phi$ to the above ODE with the boundary conditions $\lim_{z \to \pm \infty} \phi(z) = \pm 1$ and satisfying $\phi(0) = 0$ is the function $\phi(z) = \tanh(z/\sqrt{2})$.  With the help of \eqref{defn:surfaceenergy}, we compute the surface energy $\gamma = \gamma_{12}$ to be
\begin{align*}
\gamma = \int_{-\infty}^{\infty} 2 \Psi(\overline{\bm{\Phi}}(z)) \dz = \int_{-\infty}^{\infty} 2 \tilde{\Psi}(\phi(z)) \dz = \int_{-1}^{1} \sqrt{2 \tilde{\Psi}(s)} \ds = \frac{2 \sqrt{2}}{3}.
\end{align*}
We observe that the jump conditions \eqref{jump:bmNsigma} and \eqref{Velo:sigma} become
\begin{align*}
0 = \jump{N_{,\sigma}}{H}{T} \, \Rightarrow \, \jump{\sigma}{H}{T} = 2 \frac{\chi_{\varphi}}{\chi_{\sigma}}, \quad (-\mathcal{V} - \vec{v} \cdot \vec{\nu}) \jump{\sigma}{H}{T} = 2 \frac{\chi_{\varphi}}{\chi_{\sigma}} (-\mathcal{V} + \vec{v} \cdot \vec{\nu}) = \chi_{\sigma} \jump{n(\tilde\varphi) \nabla \sigma}{H}{T} \cdot \vec{\nu},
\end{align*}
respectively, and a short computation shows that
\begin{align*}
\jump{\Proj{\bm{\mu}} \cdot \bm{\varphi}}{H}{T}  = \mu_{2} - \mu_{1} = 2 \tilde{\mu}, \quad \jump{ - N(\bm{\varphi}, \sigma) + \sigma N_{,\sigma}(\bm{\varphi}, \sigma)}{H}{T}  = \frac{\chi_{\sigma}}{2}\jump{\abs{\sigma}^{2}}{H}{T},
\end{align*}
so that \eqref{SI:solvabilitycond} becomes $\beta \gamma \kappa = 2 \tilde{\mu} + \frac{\chi_{\sigma}}{2} \jump{\abs{\sigma}^{2}}{H}{T}$.  Furthermore, taking the difference between the second and first components of the free boundary conditions \eqref{jump:bmmu} and \eqref{Velo:mu}, the free boundary conditions on $\Gamma = \Gamma_{12}$ translates to
\begin{subequations}
\begin{align}
\jump{\vec{v}}{H}{T} \cdot \vec{\nu} = 0, \quad
\jump{\tilde{\mu}}{H}{T} = 0, \quad
\jump{p}{H}{T} = \beta \gamma \kappa, \quad
\jump{\sigma}{H}{T} = 2\frac{\chi_{\varphi}}{\chi_{\sigma}}, \quad \beta \gamma \kappa  & = 2 \tilde{\mu} + \frac{\chi_{\sigma}}{2} \jump{\abs{\sigma}^{2}}{H}{T}, \\
2(-\mathcal{V} + \vec{v} \cdot \vec{\nu}) = \jump{m(\tilde\varphi) \nabla \mu}{H}{T} \cdot \vec{\nu}, \quad (-\mathcal{V} + \vec{v} \cdot \vec{\nu}) \jump{\sigma}{H}{T} &  = \chi_{\sigma} \jump{n(\tilde\varphi) \nabla \sigma}{H}{T} \cdot \vec{\nu}.
\end{align}
\end{subequations}
The resulting sharp interface model coincides with \cite[Equation (3.49)]{GLSS}.

\subsection{Sharp interface limit for \eqref{Intro:simplify:3compo}}\label{sec:SIM:3component}
In this section, we derive the sharp interface limit of \eqref{Intro:simplify:3compo}, so that $L = 3$ and $M = 1$, from the general sharp interface model \eqref{Sharpinterface:bulk}-\eqref{SharpInterface:freebdy}, where we again drop the subscript $0$. 
Choosing \eqref{Intro:nutrient:eg} with $\chi_{\sigma} = 1$, $\chi_{\varphi} = \chi_{n} = 0$ and \eqref{MobilityChoice:Symmetric} with 
$m_{i}(s) = 1$, $1 \leq i \leq 3$, we have
\begin{align*}
N(\bm{\varphi}, \sigma) = \frac{1}{2} \abs{\sigma}^{2}, \quad \TT{C}(\bvarphi) = \left(\delta_{ij} - \frac{1}{3}\right)_{i,j=1}^3, \quad S(\bvarphi,\sigma) = - \mathcal{C} \varphi_{2} \sigma,
\end{align*}
so that $N_{,\sigma}(\bm{\varphi}, \sigma) = \sigma$ and $\bm N_{,\bm\varphi}(\bm{\varphi}, \sigma) = \bm 0$. Moreover, defining $\Omega_{H} = \{ \bm{\varphi} = \bm{e}_{1} \}$, $\Omega_{P} = \{ \bm{\varphi} = \bm{e}_{2}\}$ and $\Omega_{N} = \{ \bm{\varphi} = \bm{e}_{3}\}$, with interfaces $\Gamma_{HP} = \pd \Omega_{H} \cap \pd \Omega_{P}$, $\Gamma_{PN} = \pd \Omega_{P} \cap \pd \Omega_{N}$, we assume that $\pd \Omega_{H} \cap \pd \Omega_{N} = \emptyset$.  One can compute that $\TT{C}(\bvarphi) \nabla \Proj{\bm{\mu}} = \TT{C}(\bvarphi) \nabla \bm{\mu}$, and thus upon setting
\begin{align*}
y = \frac{1}{3} \left ( 2 \mu_{1} - \mu_{2} - \mu_{3} \right ), \quad z= \frac{1}{3} \left (-\mu_{1} + 2 \mu_{2} - \mu_{3} \right ),
\end{align*}
from \eqref{SI:mu} and \eqref{SI:sigma} (recalling that $\sigma$ evolves quasi-statically) we obtain the following outer equations:
\begin{align*}
\Laplace \sigma = \begin{cases}
0 & \text{ in } \Omega_{H} \cup \Omega_{N}, \\
\mathcal{C} \sigma & \text{ in } \Omega_{P}.
\end{cases}, \quad \begin{cases}
- \Laplace y = U_{1}(\bm{\varphi}, \sigma) - (\unit\cdot \bm{U}(\bm{\varphi}, \sigma)) \varphi_{1}, \\
- \Laplace z = U_{2}(\bm{\varphi}, \sigma) - (\unit\cdot \bm{U}(\bm{\varphi}, \sigma)) \varphi_{2}, \\
-K \Laplace p = (\unit\cdot\bm{U}(\bm{\varphi}, \sigma)), 
\end{cases} \text{ in } \Omega_{H} \cup \Omega_{P} \cup \Omega_{N}.
\end{align*}
Meanwhile, due to the fact that $N_{,\sigma}(\bm{\varphi}, \sigma) = \sigma$ and $\Proj{\bm{\mu}} = (y, z, -(y+z))^{\top}$, we obtain from \eqref{jump:p},  \eqref{jump:bmmu}, \eqref{jump:bmNsigma} and \eqref{Velo:sigma} that
\begin{align*}
[\nabla p] \cdot \vec{\nu} = [y] = [z] = [\sigma] = [\nabla \sigma] \cdot \vec{\nu} = 0 \text{ on } \Gamma_{PN} \cup \Gamma_{HP}.
\end{align*}
Furthermore, \eqref{SI:solvabilitycond} and \eqref{Velo:mu} simplify to
\begin{align*}
\jump{p}{N}{P} = \beta \gamma_{PN} \kappa  = \Proj{\mu_{2}} - \Proj{\mu_{3}} = 2 y - z & \text{ on } \Gamma_{PN}, \\
\jump{p}{P}{H} = \beta \gamma_{HP} \kappa = \Proj{\mu_{1}} - \Proj{\mu_{2}} = y - z & \text{ on } \Gamma_{HP}, \\
-\mathcal{V} + K \nabla p \cdot \vec{\nu} = \jump{\nabla z}{N}{P} \cdot \vec{\nu} , \quad 0 = \jump{\nabla y}{N}{P} \cdot \vec{\nu} & \text{ on } \Gamma_{PN}, \\
- \mathcal{V} + K \nabla p \cdot \vec{\nu} = \jump{\nabla y}{P}{H} \cdot \vec{\nu}, \quad \mathcal{V} - K \nabla p \cdot \vec{\nu} = \jump{\nabla z}{P}{H} \cdot \vec{\nu} & \text{ on } \Gamma_{HP}.
\end{align*}

In the case where source terms of the form \eqref{Intro:source:2} are considered, the asymptotic analysis requires a slight modification, which we will briefly sketch below.  The multi-component system we study is given by
\begin{subequations}\label{Kampmann}
\begin{alignat}{3}
\div \vec{v} & = \eps^{-1}(\unit\cdot \bm{U}(\bm{\varphi}, \sigma)), \label{Kampmann:div} \\
\vec{v} & = - K \nabla p + K (\nabla \bm{\varphi})^{\top} \bm{\mu}, \label{Kampmann:Darcy} \\
\pd_{t} \bm{\varphi} + (\nabla \bm{\varphi}) \vec{v} & = \div (\TT{C}(\bm{\varphi},\sigma) \nabla \Proj{\bm{\mu}}) + \eps^{-1} \left ( \bm{U}(\bm{\varphi}, \sigma) - (\unit\cdot \bm{U}(\bm{\varphi}, \sigma) \bm{\varphi} ) \right ), \label{Kampmann:varphi} \\
\Proj{\bm{\mu}} & = - \beta \eps \Laplace \bm{\varphi} + \beta \eps^{-1} \Proj{\Psi_{,\bm{\varphi}}}(\bvarphi), \label{Kampmann:mu} \\
0 & = \Laplace \sigma + S(\bm{\varphi}, \sigma), \label{Kampmann:sigma}
\end{alignat}
\end{subequations}
where we now consider
\begin{align*}
\bm{U}(\bm{\varphi}, \sigma) = (\overline{\rho}_{1}^{-1} \mathcal{U}_{1}(\varphi_1, \sigma), \dots, \overline{\rho}_{L}^{-1} \mathcal{U}_{L}(\varphi_L, \sigma))^{\top}, \quad \mathcal{U}_{k}(\varphi_k, \sigma) = \varphi_{k}^{2} (1-\varphi_{k})^{2} F_{k}(\sigma)
\end{align*}
with scalar functions $F_{k}$, $1 \leq k \leq L$.  
For example, we may choose $L=3$, $\overline{\rho}_{i} = 1$ for $1 \leq i \leq 3$, and $F_{1}(s) = 0$, $F_{2}(s) = \mathcal{P} s - \mathcal{A}$, $F_{3}(s) = \mathcal{A} - D_{N}$ in order to match with \eqref{Intro:source:2}. 

In the outer expansions, we obtain \eqref{outer:varphi:-1} from $\eqref{Kampmann:mu}_{O}^{-1}$, which implies that $\bm{\varphi}_{0} = \bm{e}_{i}$, $1 \leq i \leq L$.  Then noting that the source term $\eps^{-1}\left ( \bm{U}(\bm{\varphi}, \sigma) - (\unit\cdot \bm{U}(\bm{\varphi}, \sigma))\bm{\varphi} \right )$ will not contribute to leading and first order, we obtain from the zeroth order expansions of \eqref{Kampmann:div}-\eqref{Kampmann:varphi}, \eqref{Kampmann:sigma} the outer equations
\begin{align*}
\div \vec{v}_{0} = 0, \quad \vec{v}_{0} = - K \nabla p_{0}, \quad  \div (\TT{C}(\bm{\varphi}_{0},\sigma_0) \nabla (\Proj{\bm{\mu}_{0}})) = \bm{0}, \quad \Laplace \sigma_{0} = - S(\bm{\varphi}_{0}, \sigma_{0}).
\end{align*}
For the free boundary conditions on the interface $\Gamma_{ij}$, using the inner expansions, we recover \eqref{inner:mu:constant} from $\eqref{Kampmann:varphi}_{I}^{-2}$ and from $\eqref{Kampmann:mu}_{I}^{-1}$ we have similarly that $\bm{\Phi}_{0}$ is a function only in $z$ connecting $\bm{e}_{i}$ to $\bm{e}_{j}$.  While the rest of the analysis is analogous, the only difference lies in $\eqref{Kampmann:div}_{I}^{-1}$ and $\eqref{Kampmann:varphi}_{I}^{-1}$ where now the source term enters.  More precisely, from $\eqref{Kampmann:div}_{I}^{-1}$ and $\eqref{Kampmann:varphi}_{I}^{-1}$ we have
\begin{subequations}
\begin{alignat}{3}
\pd_{z} \vec{V}_{0} \cdot \vec{\nu} & = \unit\cdot \bm{U}(\bm{\Phi}_{0}, \Sigma_{0}), \label{Kampmann:inner:div} \\
(-\mathcal{V} + \vec{V}_{0} \cdot \vec{\nu})\bm{\Phi}_{0}' & = \pd_{z} \left ( \left ( \TT{C}(\bm{\Phi}_{0},\Sigma_{0}) \right ) [(\pd_{z} \Proj{\bm{\Upsilon}_{1}} \otimes \vec{\nu}) + \surf \Proj{\bm{\Upsilon}_{0}} ] \right ) \vec{\nu} \label{Kampmann:inner:velo} \\
\notag & \quad + \bm{U}(\bm{\Phi}_{0}, \Sigma_{0}) - (\unit\cdot \bm{U}(\bm{\Phi}_{0}, \Sigma_{0})) \bm{\Phi}_{0}.
\end{alignat}
\end{subequations}
Using the fact that $\pd_{z} \Sigma_{0} = 0$ (from $\eqref{Kampmann:sigma}_{I}^{-2}$) and introducing the notation $\Phi_{0,k}$ as the $k$th component of the vector $\bm{\Phi}_{0}$, we define for $1 \leq k, l \leq L$,
\begin{align*}
\bm{\delta} & = (\delta_{1}, \dots, \delta_{L})^{\top}, \quad \delta_{k} = \int_{\R} (\Phi_{0,k}(z))^{2}(1 - \Phi_{0,k}(z))^{2} \dz, \\
\TT{G} & \in \R^{L \times L}, \quad \TT{G}_{kl} = \int_{\R} (\Phi_{0,l}(z))^{2} (1- \Phi_{0,l}(z))^{2} \Phi_{0,k}(z) \dz, \\
\bm{F}(\sigma) & = ( \overline{\rho}_{1}^{-1} F_{1}(\sigma), \dots, \overline{\rho}_{L}^{-1} F_{L}(\sigma))^{\top}, \quad \bm{H}(\sigma) = (\delta_{1} F_{1}(\sigma), \dots, \delta_{L} F_{L}(\sigma))^{\top}.
\end{align*}
Then, integrating \eqref{Kampmann:inner:div} and \eqref{Kampmann:inner:velo} in $z$ leads to
\begin{align*}
\jump{\vec{v}_{0}}{i}{j} \cdot \vec{\nu} = \unit\cdot \bm{H}(\sigma_{0}), \quad -\mathcal{V} (\bm{e}_{j} - \bm{e}_{i}) + \jump{(\vec{v}_{0} \cdot \vec{\nu}) \bm{\varphi}_{0}}{i}{j} = \jump{\TT{C}(\bm{\varphi}_{0},\sigma_0) \nabla (\Proj{\bm{\mu}_{0}})}{i}{j} \vec{\nu} + \bm{H}(\sigma_{0}),
\end{align*}
where we used that
\begin{align*}
\int_{\R} (\vec{V}_{0} \cdot \vec{\nu}) \bm{\Phi}_{0}' \dz = \jump{(\vec{v}_{0} \cdot \vec{\nu}) \bm{\varphi}_{0}}{i}{j} - \int_{\R} (\unit\cdot \bm{U}(\bm{\Phi}_{0}, \Sigma_{0}))\bm{\Phi}_{0} \dz = \jump{(\vec{v}_{0} \cdot \vec{\nu}) \bm{\varphi}_{0}}{i}{j}  - \TT{G} \bm{F}(\sigma).
\end{align*}
Hence, the sharp interface limit of \eqref{Kampmann} is
\begin{align*}
-K \Laplace p_{0} = 0, \quad \div (\TT{C}(\bm{\varphi}_{0},\sigma_0) \nabla(\Proj{\bm{\mu}_{0}})) = \bm{0}, \quad -\Laplace \sigma_{0} = S(\bm{\varphi}_{0}, \sigma_{0}) & \text{ in } \Omega_{k}, \\
\jump{\Proj{\bm{\mu}_{0}}}{i}{j} = \bm{0}, \quad \jump{\sigma_{0}}{i}{j} = 0, \quad \jump{\nabla \sigma_{0}}{i}{j} \cdot \vec{\nu} = 0, \quad \beta \gamma_{ij} \kappa = \jump{\Proj{\bm{\mu}_{0}} \cdot \bm{\varphi}_{0}}{i}{j} = \jump{p_{0}}{i}{j} & \text{ on } \Gamma_{ij}, \\
\jump{\vec{v}_{0}}{i}{j} \cdot \vec{\nu} = \unit\cdot\bm{H}(\sigma_{0}), \quad -\mathcal{V} (\bm{e}_{j} - \bm{e}_{i}) + \jump{(\vec{v}_{0} \cdot \vec{\nu}) \bm{\varphi}_{0}}{i}{j} = \jump{\TT{C}(\bm{\varphi}_{0},\sigma_0) \nabla (\Proj{\bm{\mu}_{0}})}{i}{j} \vec{\nu} + \bm{H}(\sigma_{0}) & \text{ on } \Gamma_{ij},
\end{align*}
for $1 \leq k \leq L$ and $1 \leq i < j \leq L$.

\begin{remark}\label{rem:Obstacle:sourceterm}
In our numerical investigations below, we will use an obstacle potential \eqref{Obstacle:Pot}, and the asymptotic analysis for the obstacle potential will yield that the outer expansions $\bm{\varphi}_{i}$ for $i \geq 1$ are all zero.  Hence, it is sufficient to consider source terms of the form \eqref{Intro:source:2} with the prefactor $\varphi_{i}(1-\varphi_{i})$ instead of $\varphi_{i}^{2}(1-\varphi_{i})^{2}$, which will also lead to the same outer equations for the sharp interface limit, but in general, the prefactors $\delta_{k}$ will be different.
\end{remark}

\subsection{Sharp interface limit of a model with degenerate Ginzburg--Landau energy}\label{sec:SIM:degGL}
In this section, we study a particular three-component model consisting of host cells $(\varphi_{1})$, proliferating cells $(\varphi_{2})$ and necrotic cells $(\varphi_{3})$ along with a quasi-static nutrient $(\sigma)$, which is derived from a degenerate Ginzburg--Landau energy, similar to the discussions in Section \ref{sec:DegGL}.  In particular, we have $L = 3$ and $M = 1$.  We consider a total energy of the form
\begin{align}\label{deg:energy}
\mathcal{E}(\bvarphi,\sigma)= 
\int_\Omega
\frac{\beta \eps}{2} \abs{\nabla \varphi_{2}}^{2}
+ \frac{\beta}{\eps} W(\varphi_{2})  + \frac{\chi_{\sigma}}{2} \abs{\sigma}^{2} - \chi_{\varphi} \sigma \varphi_{2} \dx,
\end{align}
where $W  : \R \to \R$ is a potential with minima at $0$ and $1$.  In the context of cellular adhesion, we assume that the host cells and necrotic cells prefer to adhere to each other rather than to the proliferating cells.  Let us consider the bare mobilities $m_{1}(s) = 1-s$, $m_{2}(s) = s$ and $m_{3}(s) = 1-s$, and the second order tensor $\TT{C}(\bm{\varphi}) \in \R^{3 \times 3}$ with entries $\TT{C}_{ij}(\bm{\varphi}) = \frac{3}{2} m_{i}(\varphi_{i}) \left (\delta_{ij} - m_{j}(\varphi_{j}) \left ( \sum_{k=1}^3 m_{k}(\varphi_{k}) \right )^{-1} \right )$, $1 \leq i, j \leq 3$.  Then, for a vector of source terms $\bm{U}$ such that $\unit\cdot \bm{U} = 0$, the model \eqref{model:zeroexcesstotalamass} becomes
\begin{subequations}
\begin{align}
\pd_{t} \bm{\varphi} & = \div (\TT{C}(\bm{\varphi}) \nabla (\Proj{\bm{\mu}}) ) + \bm{U}(\bm{\varphi}, \sigma), \\
\mu_{2} & = - \beta \eps \Laplace \varphi_{2} + \beta \eps^{-1} W'(\varphi_{2}) - \chi_{\varphi} \sigma, \quad \mu_{1} = \mu_{3} = 0, \\
\pd_{t}\sigma & = \div (n(\bm{\varphi},\sigma) \nabla (\sigma - \lambda \varphi_{2})) + S(\bm{\varphi}, \sigma),
\end{align}
\end{subequations}
where $n(\bm{\varphi}, \sigma) = \TT{D}(\bm{\varphi}, \sigma) \chi_{\sigma}$, $\lambda = \chi_{\varphi}/\chi_{\sigma}$, and $\mu_{1} = \mu_{3} = 0$ precisely due to the fact that the energy \eqref{deg:energy} does not depend on $\varphi_{1}$ and $\varphi_{3}$.  Sending $\lambda \to 0$, and neglecting the left-hand side, as the nutrient evolves quasi-statically, and then considering a constant mobility $n(\bvarphi,\sigma) = 1$, leads to the phase field model
\begin{subequations}
\begin{align}
\pd_{t}\varphi_{2} & = \div \left ( \TT{C}_{22}(\bm{\varphi}) \nabla \mu_{2} \right ) + \overline{\rho}_{2}^{-1} \mathcal{U}_{2}(\bm{\varphi}, \sigma), \label{SIM:degGL:2} \\
\pd_{t} \varphi_{i} & = \div \left ( \TT{C}_{i2}(\bm{\varphi}) \nabla \mu_{2} \right ) + \overline{\rho}_{i}^{-1} \mathcal{U}_{i}(\bm{\varphi}, \sigma), \quad i = 1,3, \label{SIM:degGL:1:3} \\
\mu_{2} & = -\beta \eps \Laplace \varphi_{2} + \beta \eps^{-1} W'(\varphi_{2}) - \chi_{\varphi} \sigma, \\
0 & = \Laplace \sigma + S(\bm{\varphi}, \sigma).
\end{align}
\end{subequations}
Note that by the relations $\unit \cdot \bm{U} = \overline{\rho}_{1}^{-1} \mathcal{U}_{1} + \overline{\rho}_{2}^{-1} \mathcal{U}_{2} + \overline{\rho}_{3}^{-1} \mathcal{U}_{3} = 0$, $\varphi_{2} = 1 - \varphi_{1} - \varphi_{3}$, and $\TT{C}_{22}(\bm{\varphi}) = -\TT{C}_{12}(\bm{\varphi}) - \TT{C}_{32}(\bm{\varphi})$, equation \eqref{SIM:degGL:2} can be written as
\begin{align*}
-\pd_{t}(\varphi_{1} + \varphi_{3}) = - \div ((\TT{C}_{12} + \TT{C}_{32})(\bm{\varphi}) \nabla \mu_{2}) - (\overline{\rho}_{1}^{-1} \mathcal{U}_{1} + \overline{\rho}_{3}^{-1} \mathcal{U}_{3})(\bm{\varphi}, \sigma),
\end{align*}
which is the negative of the sum of \eqref{SIM:degGL:1:3}.  As the minima of $W$ are $0$ and $1$, we have that the leading order term $\varphi_{2,0} = 0$ or $1$, which allows us to define the regions $\Omega_{P} = \{ \varphi_{2,0} = 1\} = \{ \varphi_{1,0} + \varphi_{3,0} = 0 \}$ and $\Omega_{P}^{c} = \{ \varphi_{2,0} = 0 \} = \{ \varphi_{1,0} + \varphi_{1,3} = 1\}$.  Then, the following outer equations are derived:
\begin{align*}
0 = \Laplace \sigma_{0} + S(\bm{\varphi}_{0}, \sigma_{0})  & \text{ in } \Omega_{P} \cup \Omega_{P}^{c}, \\
-\Laplace \mu_{2,0} = \overline{\rho}_{2}^{-1} \mathcal{U}_{2}(\bm{\varphi}_{0} , \sigma_{0})  = - (\overline{\rho}_{1}^{-1} \mathcal{U}_{1} + \overline{\rho}_{3}^{-1} \mathcal{U}_{3})(\bm{\varphi}_{0}, \sigma_{0}) & \text{ in } \Omega_{P}, \\
0=
\overline{\rho}_{2}^{-1} \mathcal{U}_{2}(\bm{\varphi}_{0}, \sigma_{0})  = -(\overline{\rho}_{1}^{-1} \mathcal{U}_{1} + \overline{\rho}_{3}^{-1} \mathcal{U}_{3})(\bm{\varphi}_{0}, \sigma_{0}) & \text{ in } \Omega_{P}^{c}.
\end{align*}
Here we used that $\TT{C}_{22}(\bm{\varphi})$ is $1$ in $\Omega_{P}$ and $0$ in $\Omega_{P}^{c}$, while $\TT{C}_{12}(\bm{\varphi})$, $\TT{C}_{32}(\bm{\varphi})$ are equal to $-\frac{1}{2}$ in $\Omega_{P}$ and $0$ in $\Omega_{P}^{c}$.  Let $\Gamma = \pd \Omega_{P}$ denote the interface that is moving with normal velocity $\mathcal{V}$, and let $\vec{\nu}$ and $\kappa$ denote the outward unit normal and mean curvature of $\Gamma$, respectively.  Then, we obtain from the inner expansions the following set of equations
\begin{align*}
[\sigma_{0}] = 0, \quad [\nabla \sigma_{0}] \cdot \vec{\nu} = 0, \quad \mu_{2,0} = \beta \gamma \kappa - \chi_{\varphi} \sigma_{0}, \quad - \mathcal{V} = \nabla \mu_{2,0} \cdot \vec{\nu} \text{ on } \Gamma,
\end{align*}
where $\gamma$ is a positive constant defined by $\gamma = \int_{0}^{1} \sqrt{2 W(s)} \ds$.  In the case of equal densities $\overline{\rho}_{1} = \overline{\rho}_{2} = \overline{\rho}_{3} = 1$, we now choose the source terms to be
\begin{align*}
S(\bm{\varphi}, \sigma) = -\sigma, \quad  \mathcal{U}_{2}(\bm{\varphi}, \sigma) = \mathcal{P} \sigma \varphi_{2}, \quad  \mathcal{U}_{1}(\bm{\varphi}, \sigma) =  \mathcal{U}_{3}(\bm{\varphi}, \sigma) = - \frac{1}{2} \mathcal{P} \sigma \varphi_{2},
\end{align*}
for a positive constant $\mathcal{P}$, and define $p = \mu_{2,0} + \chi_{\varphi} \sigma_{0}$, so that we have
\begin{align*}
\Laplace \sigma_{0} = \sigma_{0}, \quad \Laplace p = \left ( \chi_{\varphi} - \mathcal{P} \right ) \sigma_{0} \text{ in } \Omega_{P}, \quad \Laplace \sigma_{0} = \sigma_{0} \text{ in } \Omega_{P}^{c}, \\
[\sigma_{0}] = 0, \quad [\nabla \sigma_{0}] \cdot \vec{\nu} = 0, \quad p = \beta \gamma \kappa, \quad -\mathcal{V} = \nabla p \cdot \vec{\nu} - \chi_{\varphi} \nabla \sigma_{0} \cdot \vec{\nu} \text{ on } \Gamma,
\end{align*}
which bears some similarities to the free boundary models studied in \cite{CLLW,CLNie,EscherMM,PFCL}.

\section{Numerical approximation}\label{sec:numerics}
In this section we propose a finite element approximation for the 
three-component model \eqref{Intro:3component} and present several numerical 
simulations for it. 
In particular, we have $L = 3$ and $M = 1$ and consider the obstacle potential \eqref{Obstacle:Pot}.  For the mobility tensor $\TT{C}(\bm{\varphi})$ we choose
\eqref{MobilityChoice:Symmetric} with
\begin{equation} \label{eq:ms}
m_{1}(s) = 1 - s + \delta_{C}, \quad m_{2}(s) = s + \delta_{C}, \quad m_{3}(s) = s + \delta_{C},
\end{equation}
where $\delta_{C} = 10^{-6}$ is a regularisation parameter.
Moreover, we consider \eqref{Intro:nutrient:eg} with
$\chi_{\sigma} ,\chi_{\varphi} > 0$ and $\chi_{n} = 0$, so that
\begin{equation}\label{eq:Xi}
\bm{N}_{,\bm{\varphi}}(\sigma) = \left( 0, -\chi_{\varphi} \sigma, 0 \right)^{\top}.
\end{equation}

In order to allow for the case $K=0$, which means that we set the velocity 
to zero, we define
\begin{equation}\label{eq:UK}
\bm{\widehat U}(\bm{\varphi}, \sigma) = \begin{cases}
\bm{U}(\bm{\varphi}, \sigma) & \text{ for } K > 0,\\
\bm{U}(\bm{\varphi}, \sigma) - (\unit\cdot \bm {U}(\bm{\varphi}, \sigma)) \bm{\varphi} & \text{ for } K = 0.
\end{cases}
\end{equation}
Recalling \eqref{Intro:source:sigma}-\eqref{Intro:source:2}, we consider 
\begin{subequations}
\begin{align} 
S(\bm{\varphi}, \sigma) & = -\mathcal{C} \sigma \varphi_{2}, \label{eq:S} \\
\bm{U}_{A}(\bm{\varphi}, \sigma) & = 
(0, \varphi_{2} (\mathcal{P} \sigma - \mathcal{A}),
\mathcal{A}\varphi_{2} - D_{N} \varphi_{3} )^{\top}, \label{eq:U} \\
\bm{U}_{B}(\bm{\varphi}, \sigma) & = (-\varphi_{2}\mathcal{P} \sigma, \varphi_{2} (\mathcal{P} \sigma - \mathcal{A}), \mathcal{A} \varphi_{2} - D_{N}\varphi_{3} )^{\top}, \label{eq:Unew2} \\
\bm{U}_{C}(\bm{\varphi}, \sigma) & = (0, \eps^{-1}\varphi_{2} (1-\varphi_{2}) (\mathcal{P} \sigma - 
\mathcal{A}), \eps^{-1} \varphi_{3} (1-\varphi_{3})(\mathcal{A} - D_{N}))^{\top}. \label{eq:Unew}
\end{align}
\end{subequations}
Here we note that \eqref{eq:Unew} differs from \eqref{Intro:source:2}.  In particular, we observe that for \eqref{eq:Unew} the function $F$ in 
\eqref{3phase:source:2} is chosen as $F(s) = s\,(1-s)$, rather than $F(s) = s^{2}\,(1-s)^{2}$ as for \eqref{Intro:source:2}, which clearly does not
satisfy the conditions stated below \eqref{3phase:source:2}.  However, we 
remark that the asymptotic analysis remains valid, see Remark \ref{rem:Obstacle:sourceterm}.

\subsection{Finite element approximation}

Let $\mathcal{T}_{h}$ be a regular triangulation of $\Omega$ into disjoint open simplices.  Associated with $\mathcal{T}_h$ is the piecewise linear finite element space
\begin{align*}
S_{h} =  \left \{ \chi \in C^{0}(\overline\Omega) \Big| \,  \chi_{|_{o}} \in P_{1}(o) \ \forall o \in \mathcal{T}_{h} \right \} \subset H^{1}(\Omega),
\end{align*}
where we denote by $P_{1}(o)$ the set of all affine linear functions on $o$.  Let $\bm{S}_{h} = [S_{h}]^{L} = S_{h} \times \dots \times S_{h}$, and define
\begin{equation*}
\bm{S}_{h}^{+}= \{ \bm{\chi} \in \bm{S}^{h} :  \bm{\chi} \geq \bm{0}\}.
\end{equation*}
Similarly to \cite{Nurnberg09}, see also \cite{BarrettBG01}, we consider the splitting
\begin{equation*}\label{eq:A+-}
\TT{\mathcal{W}} \equiv \TT{\mathcal{W}}^{+} + \TT{\mathcal{W}}^{-}, \quad \text{ where } \TT{\mathcal{W}}^{+(-)} \text{ is symmetric and positive (negative) 
semi-definite},
\end{equation*}
recall \eqref{Obstacle:Pot}.  Throughout we choose 
$\TT{\mathcal{W}} =  \id - \unit \otimes \unit$, and let 
$\TT{\mathcal{W}}^- = -\frac{2}{3} \unit \otimes \unit$.  We now introduce a finite element approximation of the above described model, 
in which we have taken homogeneous Neumann boundary conditions for $\bm{\varphi}$ and $\bm{\mu}$, and the Dirichlet boundary condition $\sigma = \sigma_{B} \in \R$ on $\pd \Omega$.  To this end, let $S_{h}^{B} = \{ \chi \in S_{h} |~\chi = \sigma_{B} \text{ on } \pd \Omega \}$, as well as
$S_{h}^{0} = \{ \chi \in S_{h} |~\chi = 0 \text{ on } \pd \Omega\}$.  The numerical scheme is defined as follows:  Find
\begin{align*}
(\bm{\varphi}_{h}^{n}, \bm{\mu}_{h}^{n}, \sigma_{h}^{n}, p_{h}^{n}) \in \bm{S}_{h}^{+} \times \bm{S}_{h} \times S_{h}^{B} \times S_{h}^{0}
\end{align*}
such that
\begin{subequations}
\begin{align}
& \frac{1}{\tau} (\bm{\varphi}_{h}^{n} - \bm{\varphi}_{h}^{n-1}, \bm{\eta}_{h})_{h} + \left( \TT{C}(\bm{\varphi}_{h}^{n-1}) \, \nabla \bm{\mu}_{h}^{n}, \nabla \bm{\eta}_{h} \right)_{h} \nonumber \\ & \hspace{1cm}
  = (\bm{\widehat U}(\bm{\varphi}_{h}^{n-1}, \sigma_{h}^{n-1}), \bm{\eta}_{h})_{h} -([\unit\cdot \bm{\widehat U}(\bm{\varphi}_{h}^{n-1}, \sigma_{h}^{n-1})]\,\bm{\varphi}_{h}^{n-1}, \bm{\eta}_{h})_{h}
\nonumber \\ & \hspace{1cm} \qquad
+ K \left( (\nabla \bm{\varphi}_{h}^{n-1}) ( \nabla\,p_{h}^{n-1} - (\nabla \bm{\varphi}_{h}^{n-1})^{\top} (\bm{\mu}_{h}^{n-1} - \bm{N}_{,\bm{\varphi}}(\sigma_{h}^{n-1}))), \bm{\eta}_{h}\right)_{h}, \label{eq:FEAc} \\
& \beta \eps (\nabla \bm{\varphi}_{h}^{n}, \nabla (\bm{\zeta}_{h} - \bm{\varphi}_{h}^{n}))  - (\beta\eps^{-1} \TT{\mathcal{W}}^{-} \bm{\varphi}_{h}^{n} + \bm{\mu}_{h}^{n}, \bm{\zeta}_{h} - \bm{\varphi}_{h}^{n})_{h} \nonumber \\ 
& \hspace{4cm} \geq \beta \eps^{-1} (\TT{\mathcal{W}}^{+} \bm{\varphi}_{h}^{n-1}, \bm{\zeta}_{h} - \bm{\varphi}_{h}^{n})_{h} - (\bm{N}_{,\bm{\varphi}}(\sigma_{h}^{n-1}), \bm{\zeta}_{h} - \bm{\varphi}_{h}^{n})_{h}, \label{eq:FEAd} \\
& \frac{1}{\tau} (\sigma_{h}^{n} - \sigma_{h}^{n-1}, \chi_{h})_{h} 
- K \left( \nabla \sigma_{h}^{n-1} \cdot 
 ( \nabla p_{h}^{n-1} - ( \nabla \bm{\varphi}_{h}^{n})^{\top} (\bm{\mu}_{h}^{n} - \bm{N}_{,\bm{\varphi}}(\sigma_{h}^{n-1}))), \chi_{h} \right)_{h} \nonumber \\ 
 & \hspace{2cm} + ([\unit\cdot\bm{\widehat U}(\bm{\varphi}_{h}^{n-1}, \sigma_{h}^{n-1})] \sigma_{h}^{n-1}, \chi_{h})_{h} + D(\nabla \sigma_{h}^{n}, \nabla \chi_{h}) - \lambda (\nabla \varphi_{2,h}^{n}, \nabla \chi_{h}) \nonumber \\ 
 & \hspace{1cm} = - \mathcal{C} (\sigma_{h}^{n}  \varphi_{2,h}^{n} , \chi_{h})_{h} 
\label{eq:FEAe} \\
& ( \nabla p_{h}^{n}, \nabla \chi_{h}) = ((\nabla \bm{\varphi}_{h}^{n})^{\top} (\bm{\mu}_{h}^{n} -
\bm{N}_{,\bm{\varphi}}(\sigma_{h}^{n})), \nabla \chi_{h})_{h} + \frac{1}{K} (\unit\cdot\bm{\widehat U}(\bm{\varphi}_{h}^{n}, \sigma_{h}^{n}), \chi_{h})_{h},\label{eq:FEAp} 
\end{align}
\end{subequations}
holds for all $(\bm{\zeta}_{h}, \bm{\eta}_{h}, \chi_{h}) \in \bm{S}_{h}^{+} \times \bm{S}_{h} \times S_{h}^{0}$, 
where $\tau$ denotes the time step size, 
$(\cdot,\cdot)$ denotes the $L^{2}$--inner product on $\Omega$,
$(\cdot,\cdot)_{h}$ is the usual mass lumped $L^{2}$--inner product on
$\Omega$, and $\lambda = \frac{\chi_{\varphi}}{\chi_{\sigma}}$, recall
(\ref{eq:lambda}).
In the case $K = 0$ we simply neglect \eqref{eq:FEAp} and do not compute for $p_{h}^{n}$.  A quasi-static variant of the discrete nutrient equation \eqref{eq:FEAe} is given by
\begin{equation} 
 D(\nabla \sigma_{h}^{n}, \nabla \chi_{h})
- \lambda (\nabla \varphi_{2,h}^{n}, \nabla \chi_{h}) = - \mathcal{C} (\sigma_{h}^{n}  \varphi_{2,h}^{n},  \chi_{h})_{h}. 
\label{eq:simpleFEAc}
\end{equation}

We implemented the scheme \eqref{eq:FEAc}-\eqref{eq:FEAp} with the help of the finite element toolbox ALBERTA, see \cite{Alberta}.  To increase computational efficiency, we employ adaptive meshes, which have a finer mesh size $h_{f}$ within the diffuse interfacial regions and a coarser mesh size $h_{c}$ away from them, see \cite{Nurnberg09} for a more detailed description.  Clearly, the system \eqref{eq:FEAc}-\eqref{eq:FEAp} decouples, and so we first solve the variational inequality \eqref{eq:FEAc}-\eqref{eq:FEAd} for $(\bm{\varphi}_{h}^{n}, \bm{\mu}_{h}^{n})$ with 
the projected block Gauss--Seidel algorithm from \cite{Nurnberg09}.  Then we compute $\sigma_{h}^{n}$ from \eqref{eq:FEAe}, or from \eqref{eq:simpleFEAc}, and finally $p_{h}^{n}$ from \eqref{eq:FEAp}, where we employ the direct linear solver UMFPACK, see \cite{Davis04}.  Finally, to increase the efficiency of the numerical computations in this paper, we exploit the symmetry of the problem and performed all computations
only on a quarter of the desired domain $\Omega$.

\subsection{Numerical simulations}

In the following we present several numerical computations in two space dimensions for the scheme \eqref{eq:FEAc}-\eqref{eq:FEAd}, \eqref{eq:FEAp} and \eqref{eq:simpleFEAc}.  We will fix the interfacial parameter to $\eps = 0.05$ throughout, and employ a fine mesh size of $h_{f} = 0.02$, with $h_{c} = 8 h_{f}$.  For the uniform time step size we choose $\tau = 10^{-3}$.  In order to define the initial data, we introduce the following functions.  Given $R_{2}, R_{3} > 0$, $\delta_{2}, \delta_{3} \geq 0$ and $m_{2}, m_{3} \in \N$, we define 
\begin{equation} \label{eq:Rtilde}
\widetilde{R}_{i}(\vec{x}) = R_{i} + \delta_{i}\cos(m_{i} \theta), \quad \text{ with } \theta = \tan^{-1} \left (\frac{x_{2}}{x_{1}} \right ), \quad i = 2, 3.
\end{equation}
Then we set
\begin{equation} 
v_{1}(\vec{x}) = 1, \quad v_{i}(\vec{x}) = \begin{cases} 
1 & \text{ if } r(\vec{x})- \widetilde{R}_{i}(\vec{x}) \leq - \frac{\eps \pi}{2}, \\
 \frac{1}{2} - \frac{1}{2} \sin \left ( \frac{r(\vec{x})- \widetilde{R}_{i}(\vec{x})}
{\eps} \right ) & \text{ if } \abs{r(\vec{x}) -  \widetilde{R}_{i}(\vec{x})} < \frac{\eps \pi}{2}, \\
 0 &  \text{ if } r(\vec{x}) - \widetilde{R}_{i}(\vec{x}) \geq \frac{\eps \pi}{2}, \\
\end{cases} \quad i = 2, 3,
\label{eq:v}
\end{equation}
where $r(\vec{x})=[\sum_{j=1}^{d} \abs{x_{j}}^{2}]^{\frac{1}{2}}$.
In line with the asymptotics of the phase field approach, the interfacial thicknesses for $v_{2}$ and $v_{3}$ are equal to $\eps \pi$, see for example \cite[Equation (3.24)]{GLSS}.  For the initial data $\bm{\varphi}_{h}^{0}$ to \eqref{eq:FEAc}-\eqref{eq:FEAp} we set 
\begin{equation} 
(\bm{\varphi}_{h}^{0})_{i}(\vec{x}) = v_{i}(\vec{x}) \prod_{j=i+1}^{3} (1 - v_{j}(\vec{x})), 
\quad i = 1,2,3, \label{eq:uinit}
\end{equation}
see also \cite[(3.5)]{Nurnberg09}.  Unless otherwise stated, we use $R_{2} = 2$, $R_{3} = 1$ in \eqref{eq:Rtilde} and choose 
\begin{subequations}
\begin{align}
\delta_{2} & = 0.1, \quad m_{2} = 2, \quad \delta_{3} = 0.05, \quad m_{3} = 6, \label{eq:pertsg} \\
\text{or~ }\
\delta_{2} & = 0.1, \quad m_{2} = 6, \quad \delta_{3} = 0.05, \quad m_{3} = 4,
\label{eq:pertsh} \\
\text{or~ }\
\delta_{2} & = 0.1, \quad m_{2} = 2, \quad \delta_{3} = 0.
\label{eq:pertsi} 
\end{align}
\end{subequations}
In addition, we set $p_{h}^{0} = 0$ and $\sigma_{h}^{0} = \sigma_{B}$.  For the graphical
representation of $\bm{\varphi}_{h}^{n}$ we will always plot the scalar quantity $(0, 1, 2)^\top \cdot \bm{\varphi}_{h}^{n}$, which clearly takes on the values $0,\ 1,\ 2$ in the host, proliferating and necrotic phases, respectively.

In a first simulation, we investigate the radial growth of the tumour phases for the source term \eqref{eq:Unew} given sufficient nutrient.  To this end, we let $\Omega = (-5,5)^{2}$ and
\begin{align*}
\mathcal{A} & = 0.5, \quad D = 1, \quad \beta = 0.1 , \quad \mathcal{C} = 2, \quad \mathcal{P} = 0.5, \quad \lambda = \chi_{\varphi} = 0.1, \quad
\mathcal{D}_{N} = 0.
\end{align*}
For the values $\sigma_{B} = 5$ and $K = 0.01$
we start with the perturbed initial profiles
defined by \eqref{eq:pertsg} and \eqref{eq:pertsh}, respectively, and observe that in each case the initial perturbations get smoothed out, leading to a nearly radial growth.  We show the corresponding simulations in Figures~\ref{fig:multifig15}
and \ref{fig:multifig15new}.
\begin{figure}[h]
\center
\mbox{
\includegraphics[angle=-0,width=0.32\textwidth]{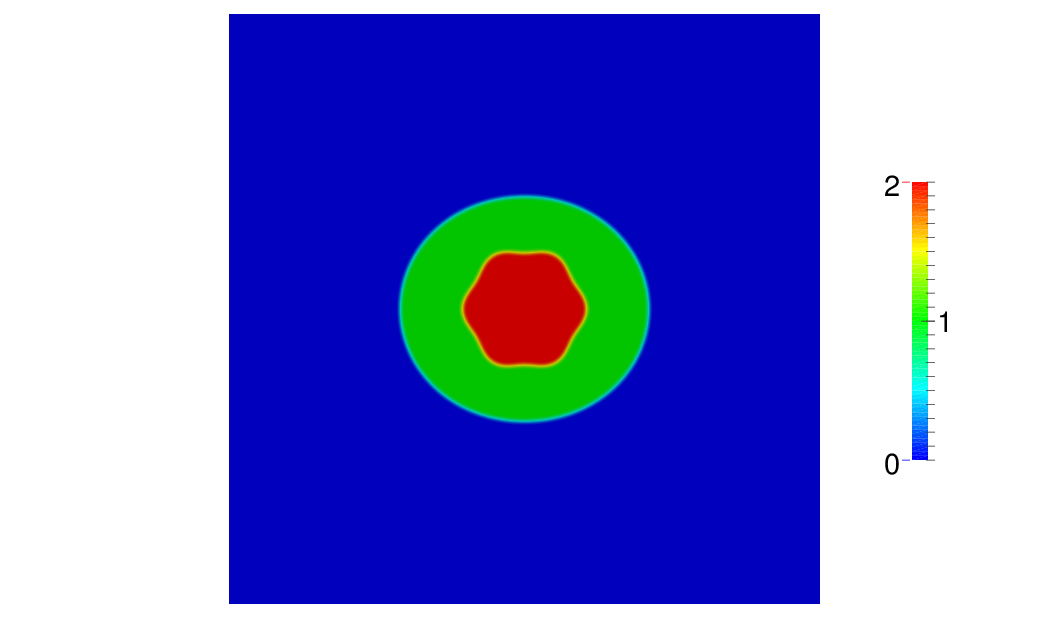}
\includegraphics[angle=-0,width=0.32\textwidth]{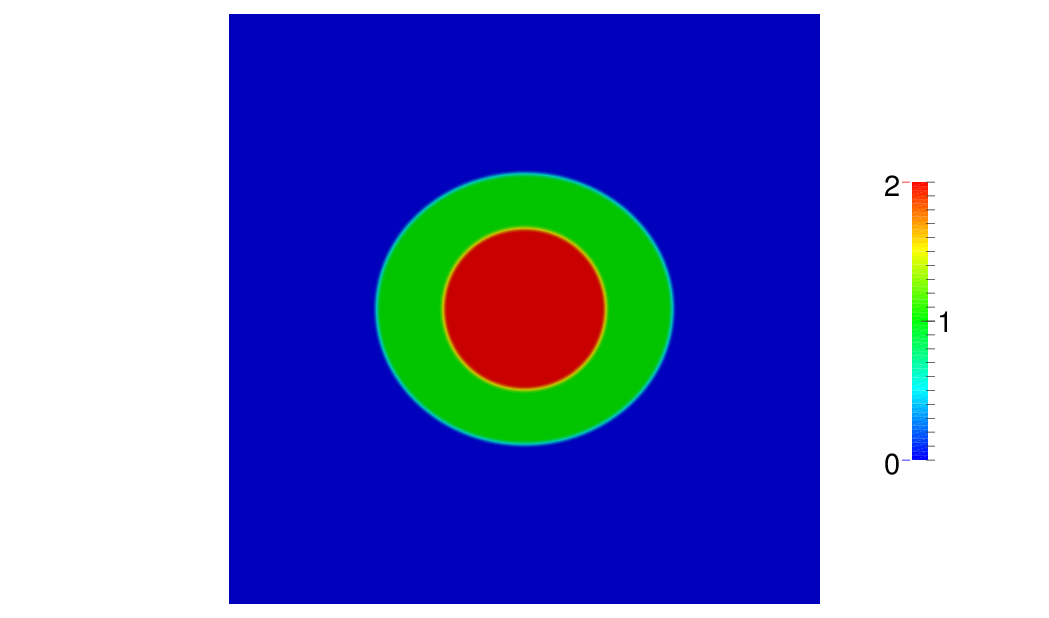}
\includegraphics[angle=-0,width=0.32\textwidth]{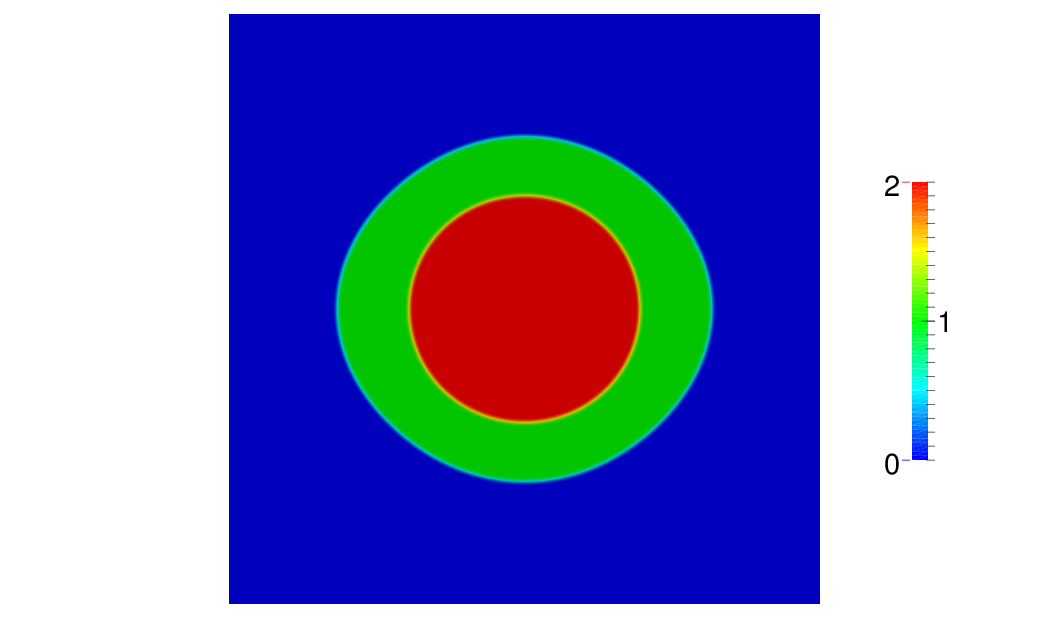}
}
\caption{($\sigma_B=5$, $K = 0.01$)
The solution $\bm{\varphi}_{h}^{n}$ at times $t=0,\ 2,\ 5$ for \eqref{eq:Unew} with initial profile \eqref{eq:pertsg}.}
\label{fig:multifig15}
\end{figure}%
\begin{figure}[h]
\center
\mbox{
\includegraphics[angle=-0,width=0.32\textwidth]{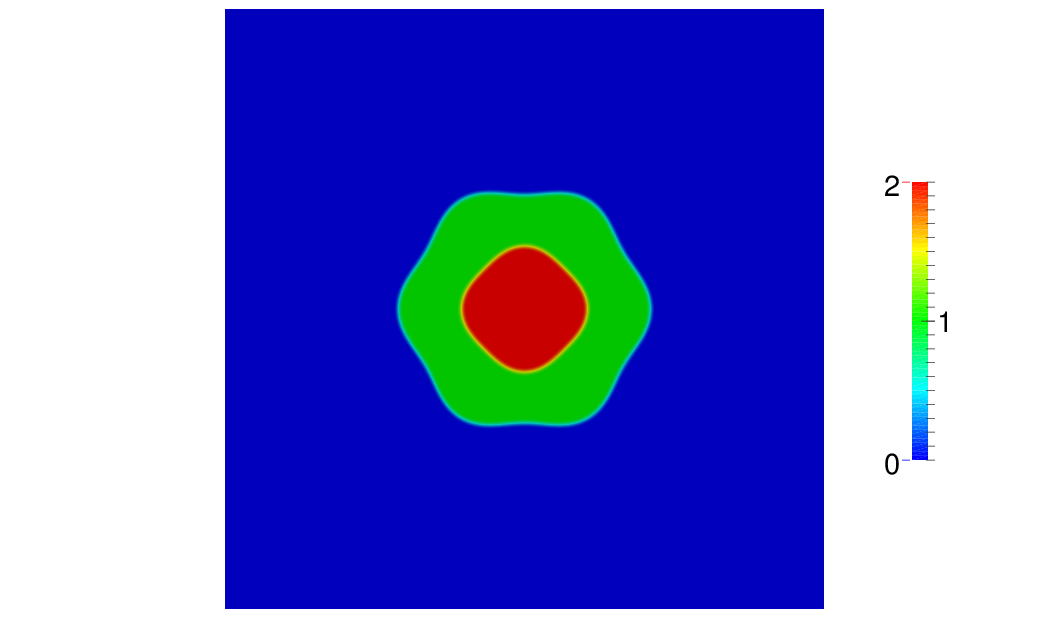} 
\includegraphics[angle=-0,width=0.32\textwidth]{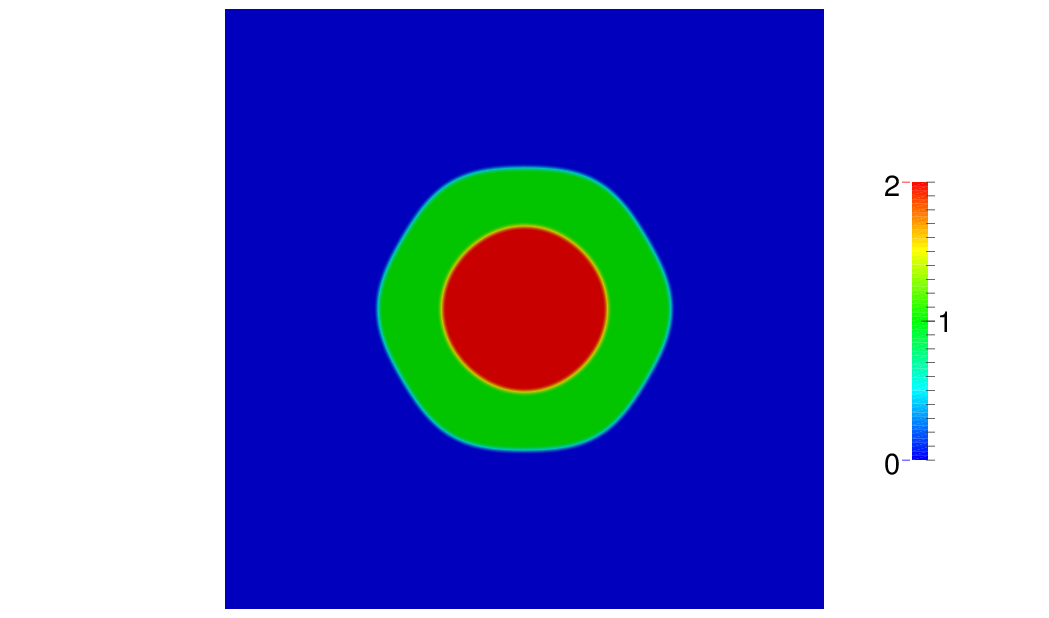} 
\includegraphics[angle=-0,width=0.32\textwidth]{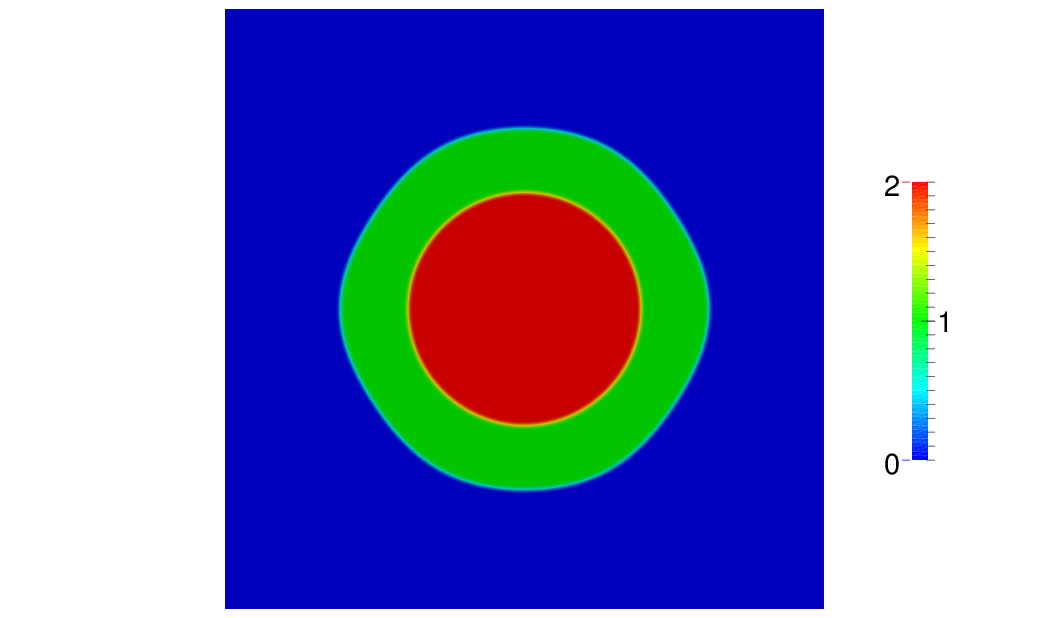} 
}
\caption{($\sigma_B=5$, $K = 0.01$)
The solution $\bm{\varphi}_{h}^{n}$ at times $t=0,\ 2,\ 5$ for \eqref{eq:Unew} with initial profile \eqref{eq:pertsh}.}
\label{fig:multifig15new}
\end{figure}%

In order to investigate the radial growth in more detail, and to study the dependence on the presence of the fluid flow and on the strength of the nutrient source, we repeat the simulations in 
Figures~\ref{fig:multifig15} and \ref{fig:multifig15new} for circular
initial data, and for different values of $\sigma_{B}$ and $K$.  In particular, 
we choose $\delta_{2} = \delta_{3} = 0$ in \eqref{eq:Rtilde} and let $\sigma_{B} \in \{2, 5, 10\}$, with $K = 0$ or $K = 0.01$.  Plots of the radii of the two interfacial layers over time for the different parameters can be seen in
Figure~\ref{fig:multifig15r}.
\begin{figure}[h]
\center
\mbox{
\includegraphics[angle=-90,width=0.32\textwidth]{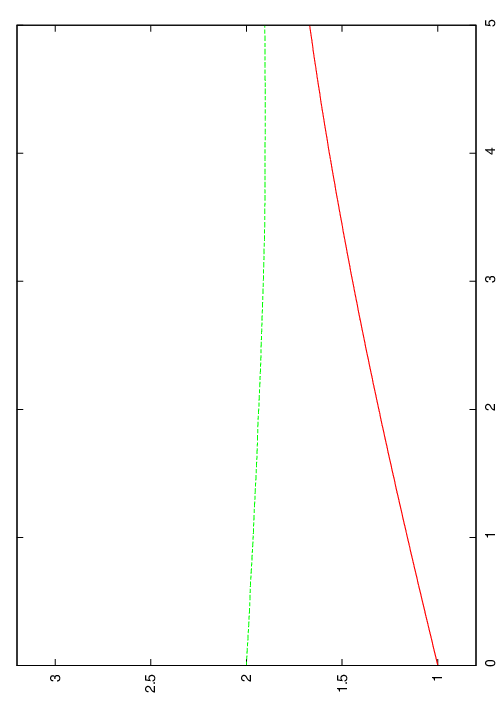} 
\includegraphics[angle=-90,width=0.32\textwidth]{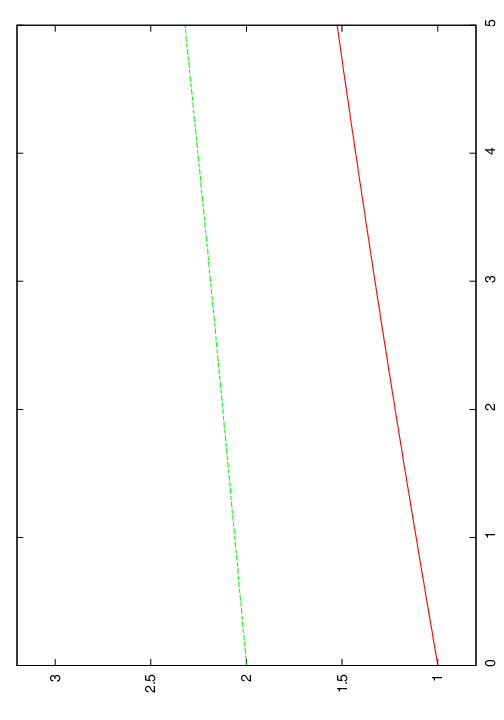} 
\includegraphics[angle=-90,width=0.32\textwidth]{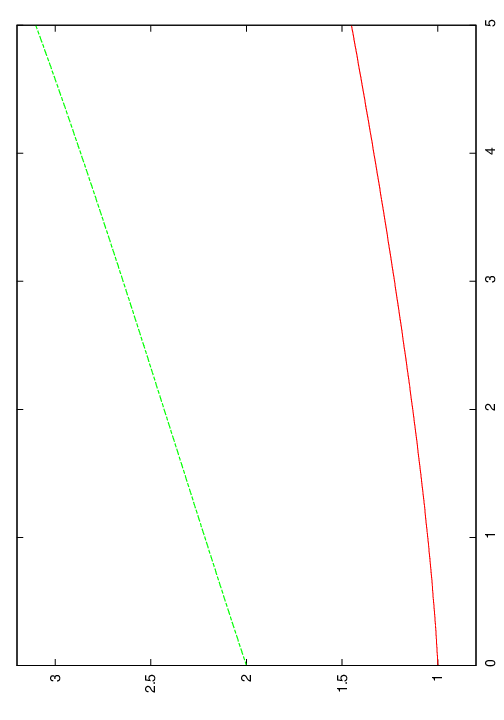} 
}
\mbox{
\includegraphics[angle=-90,width=0.32\textwidth]{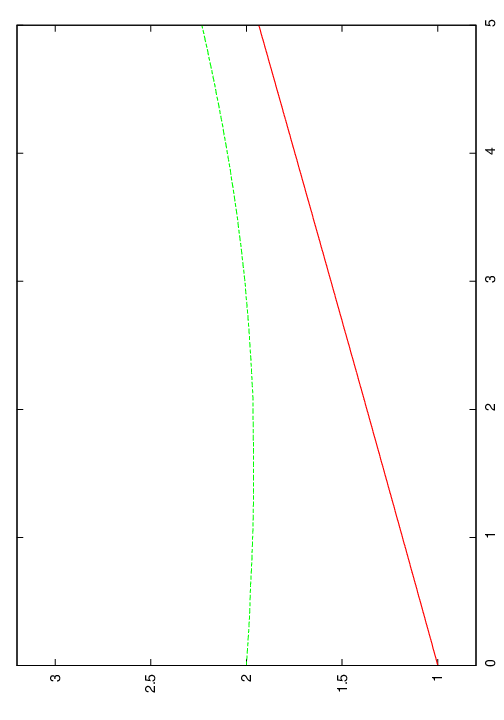} 
\includegraphics[angle=-90,width=0.32\textwidth]{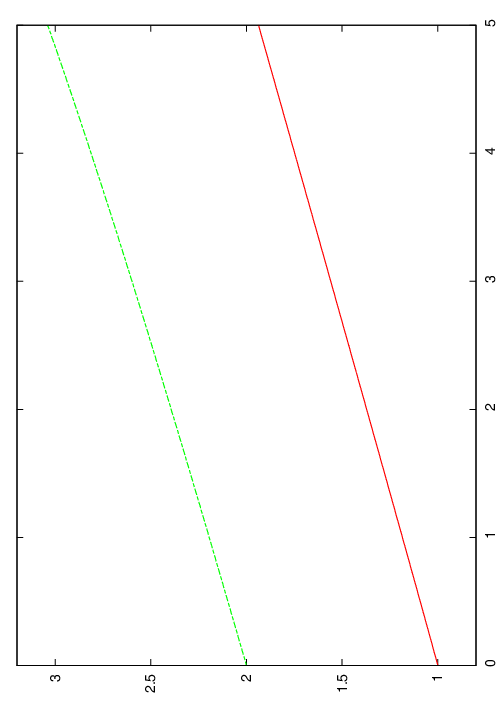} 
\includegraphics[angle=-90,width=0.32\textwidth]{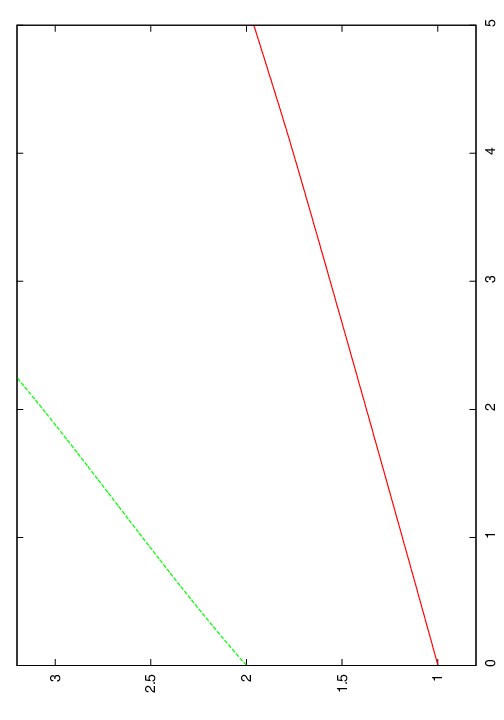} 
}
\caption{A plot of the two radii over time for \eqref{eq:Unew}.
The above plots are without fluid flow, i.e., $K = 0$, 
for $\sigma_{B} = 2,\ 5,\ 10$. Below the same for $K = 0.01$.}
\label{fig:multifig15r}
\end{figure}%

Looking at the results for $\sigma_B=2$ in particular, we also investigate whether the two radii eventually meet.  To this end, we repeat the simulations for a longer time.  As observed in Figure~\ref{fig:multifig15T}, in the absence of Darcy flow the inner radius indeed catches up with the outer radius.  When Darcy flow is present, however, a constant minimum distance between the 
two radii is maintained throughout the evolution.
\begin{figure}[h]
\center
\mbox{
\includegraphics[angle=-90,width=0.46\textwidth]{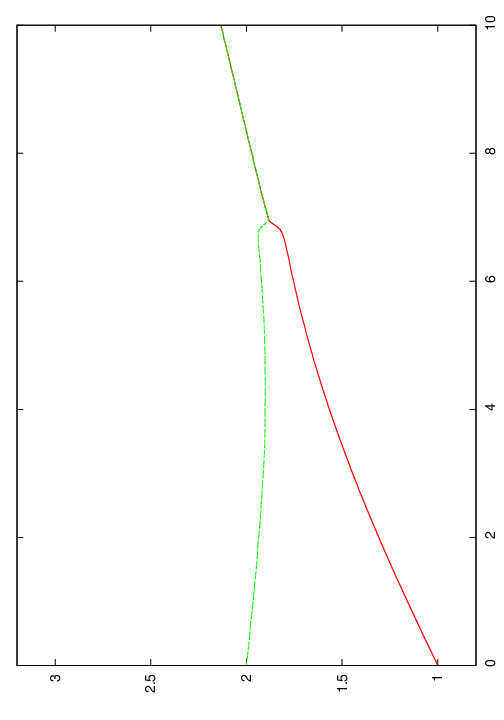} 
\includegraphics[angle=-90,width=0.46\textwidth]{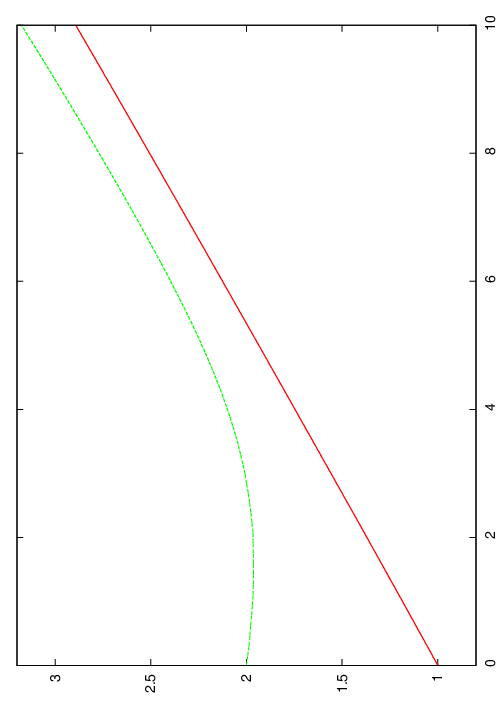} 
}
\caption{A plot of the two radii over time for \eqref{eq:Unew} and $\sigma_{B} = 2$. The left plot is without fluid flow, i.e.,\ $K = 0$, while the right plot is for $K = 0.01$. }
\label{fig:multifig15T}
\end{figure}%
We show some snapshots of the two different evolutions in 
Figure~\ref{fig:multifig15Tphi}.
\begin{figure}[h]
\center
\mbox{
\includegraphics[angle=-0,width=0.32\textwidth]{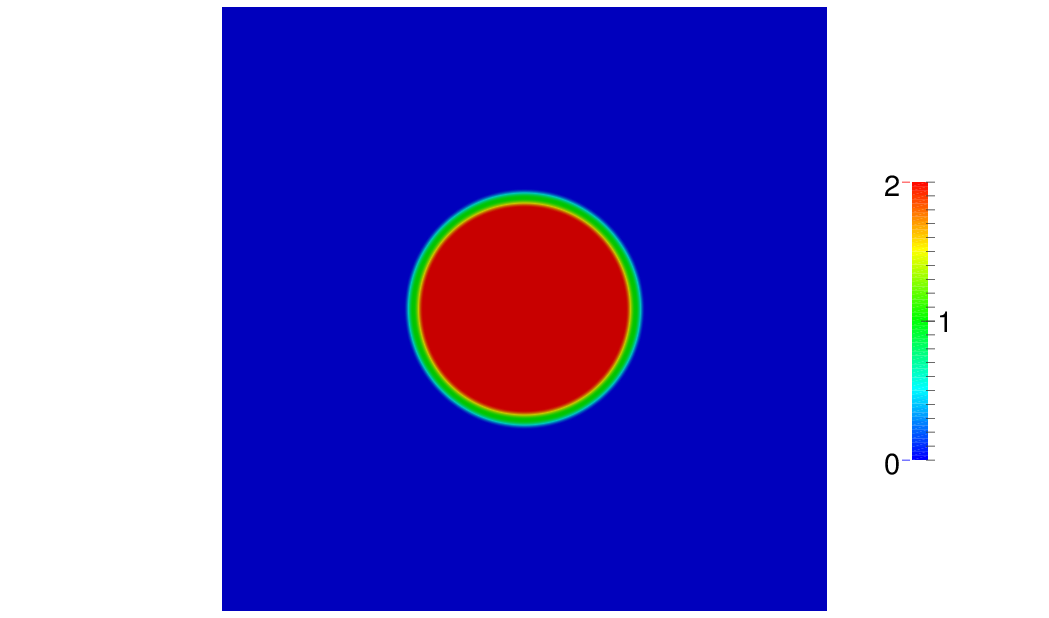} 
\includegraphics[angle=-0,width=0.32\textwidth]{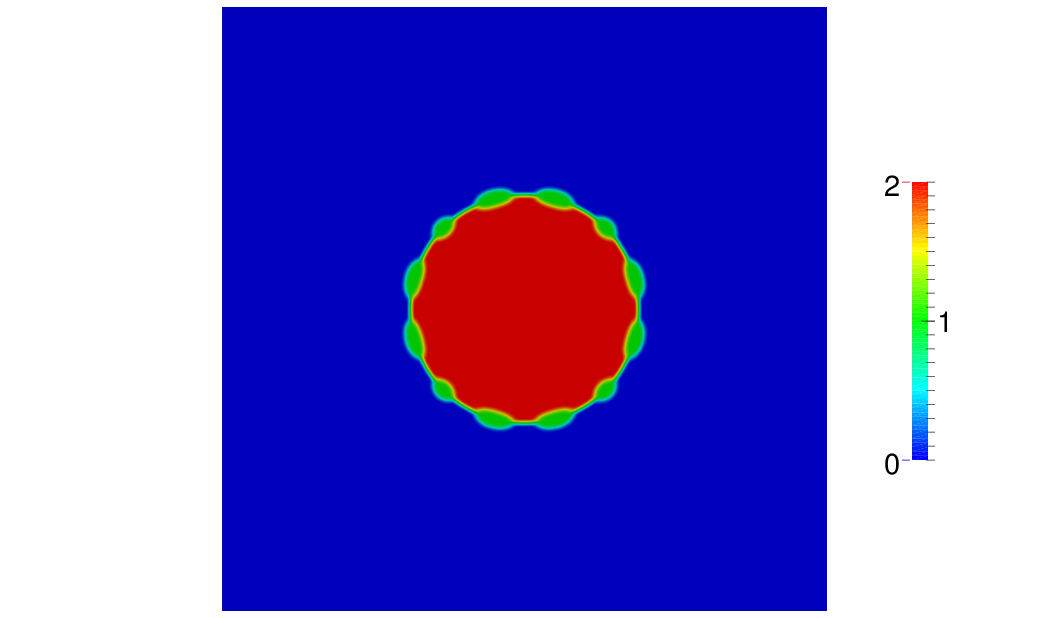} 
\includegraphics[angle=-0,width=0.32\textwidth]{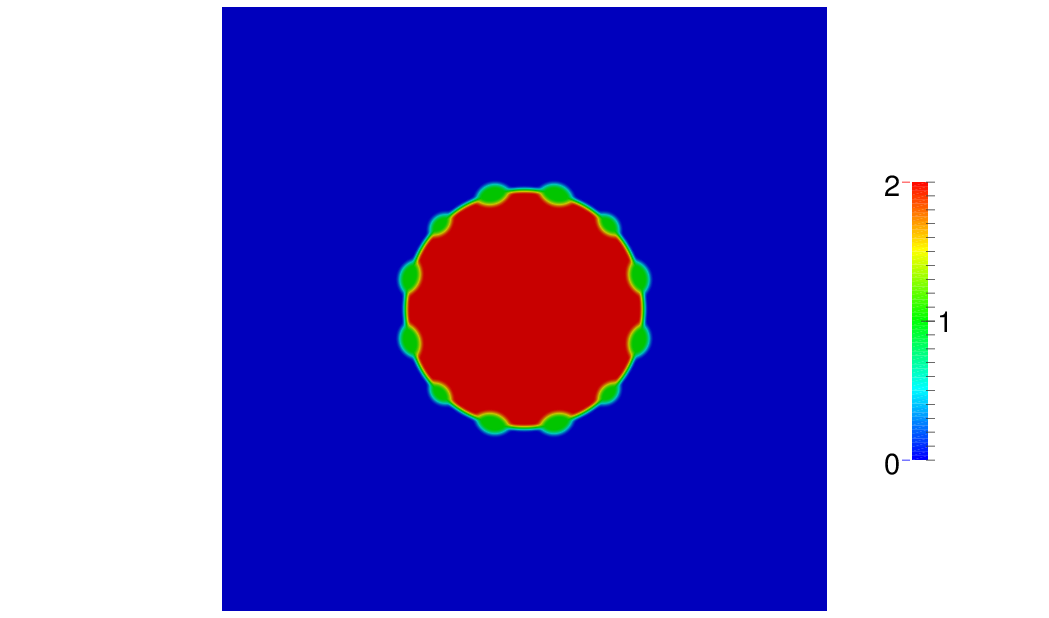} 
}
\mbox{
\includegraphics[angle=-0,width=0.32\textwidth]{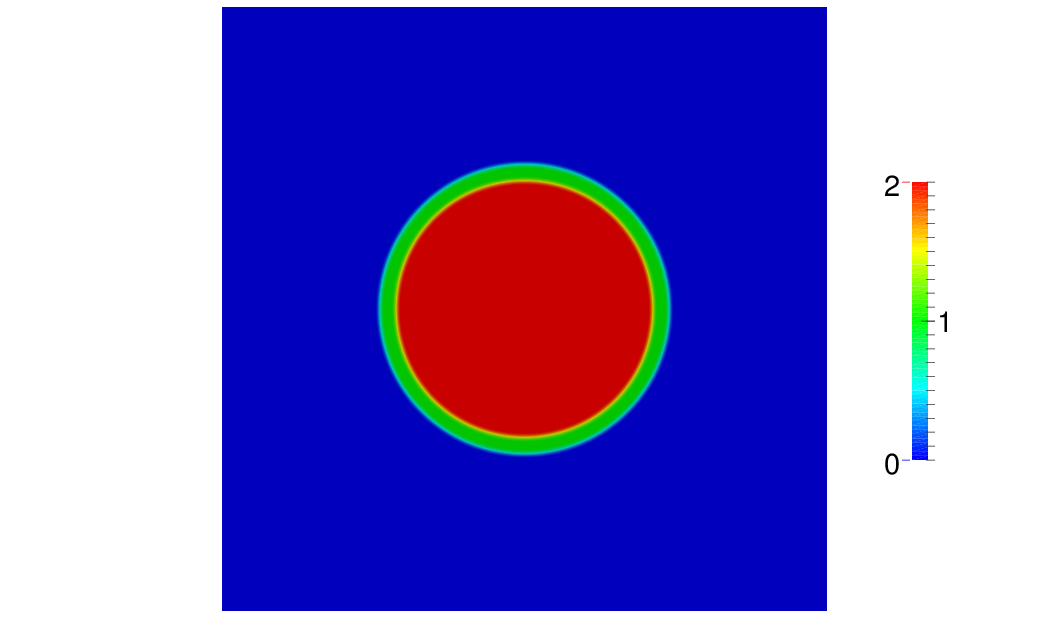} 
\includegraphics[angle=-0,width=0.32\textwidth]{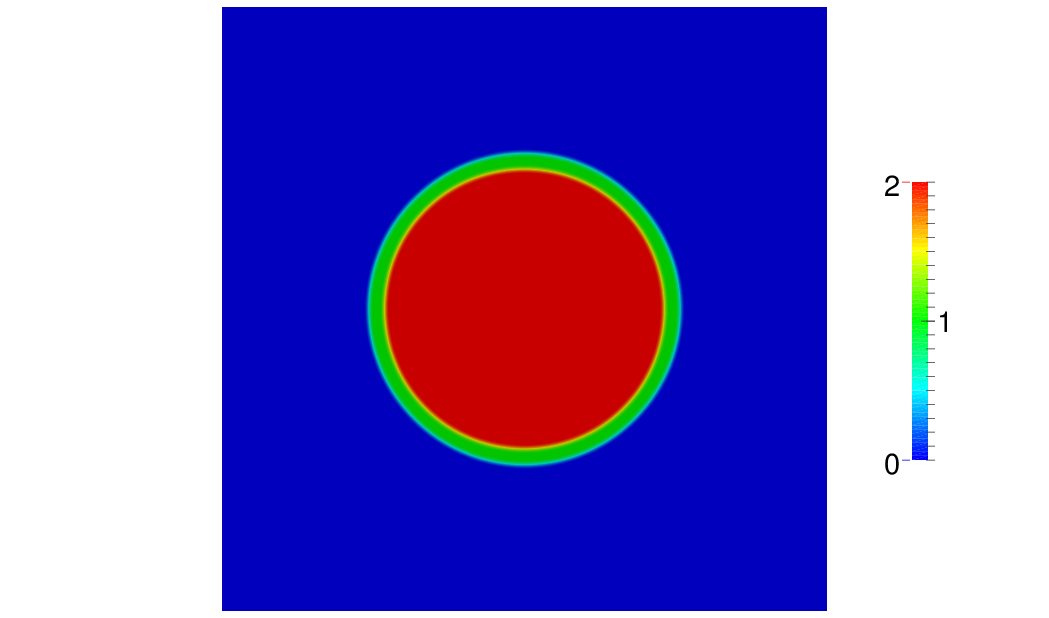} 
\includegraphics[angle=-0,width=0.32\textwidth]{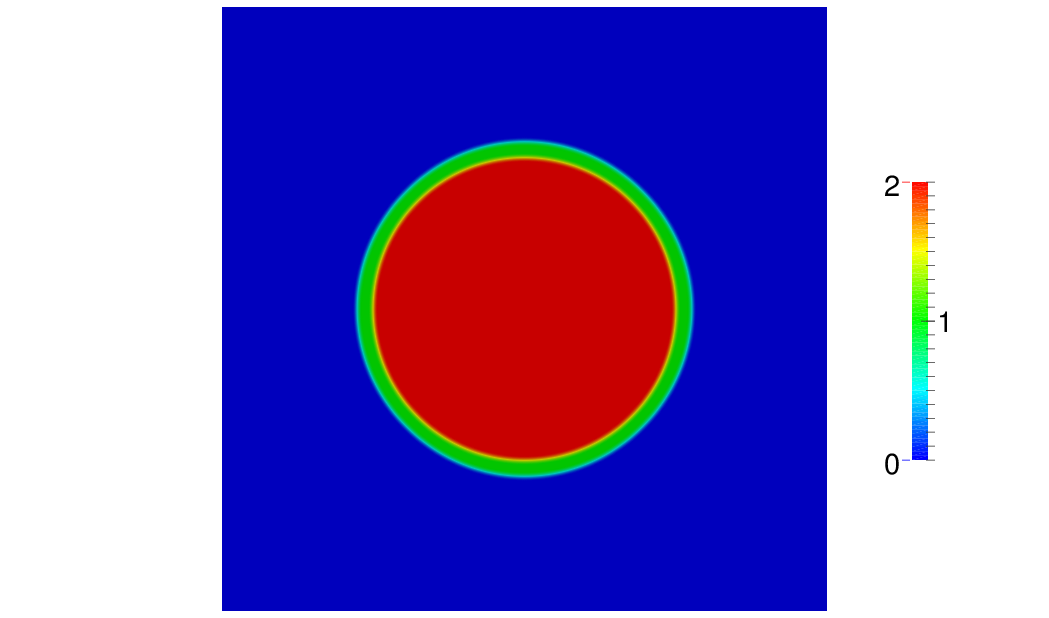} 
}
\caption{
The solution $\bm{\varphi}_{h}^{n}$ at times $t=6,\ 7,\ 8$ for the two evolutions from Figure~\ref{fig:multifig15T}. Above for $K=0$, below
for $K=0.01$.
}
\label{fig:multifig15Tphi}
\end{figure}%

The same simulation as in Figure~\ref{fig:multifig15}, but now for the source
term \eqref{eq:U} can be seen in Figure~\ref{fig:multifig11}.  As a comparison
we also show the evolution without the fluid flow, see Figure~\ref{fig:multifig11nd}.  In this case we observe quite complex nucleation phenomena of the necrotic phase within the proliferating phase.
\begin{figure}[h]
\center
\mbox{
\includegraphics[angle=-0,width=0.25\textwidth]{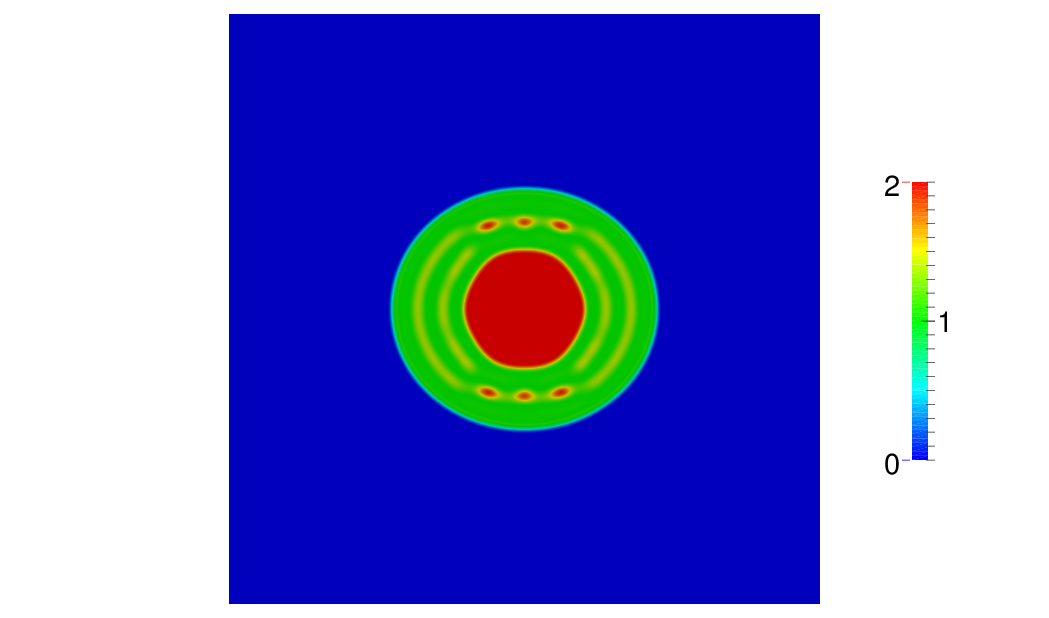} 
\includegraphics[angle=-0,width=0.25\textwidth]{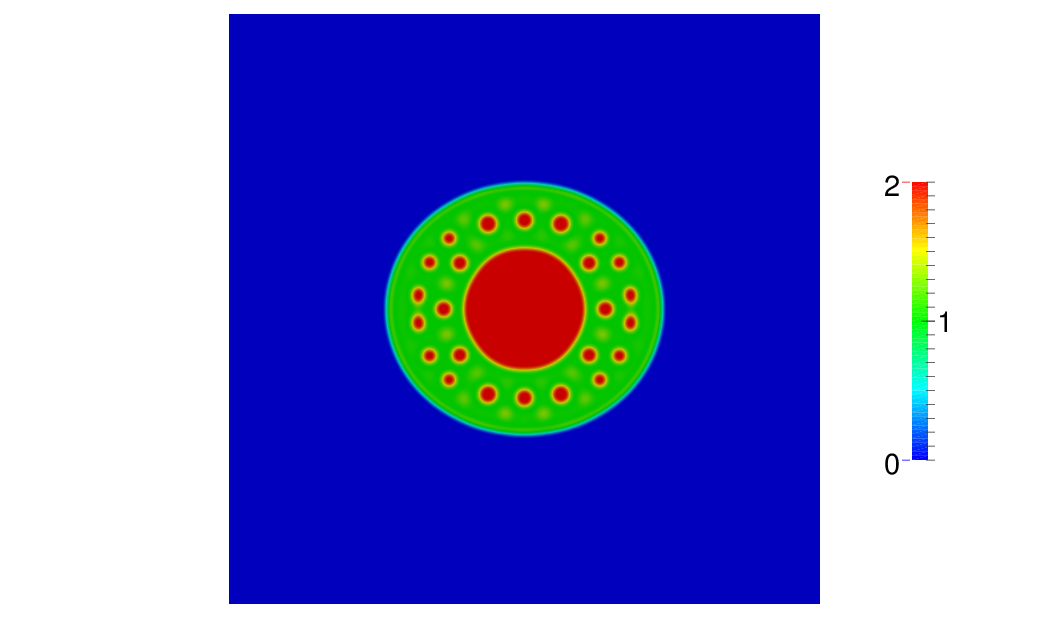} 
\includegraphics[angle=-0,width=0.25\textwidth]{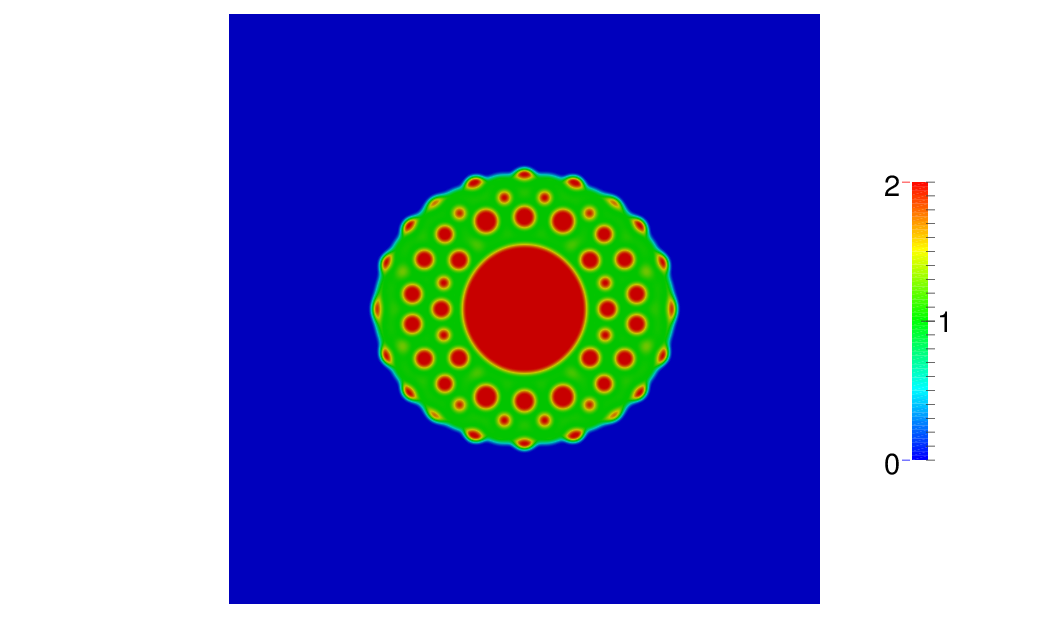} 
\includegraphics[angle=-0,width=0.25\textwidth]{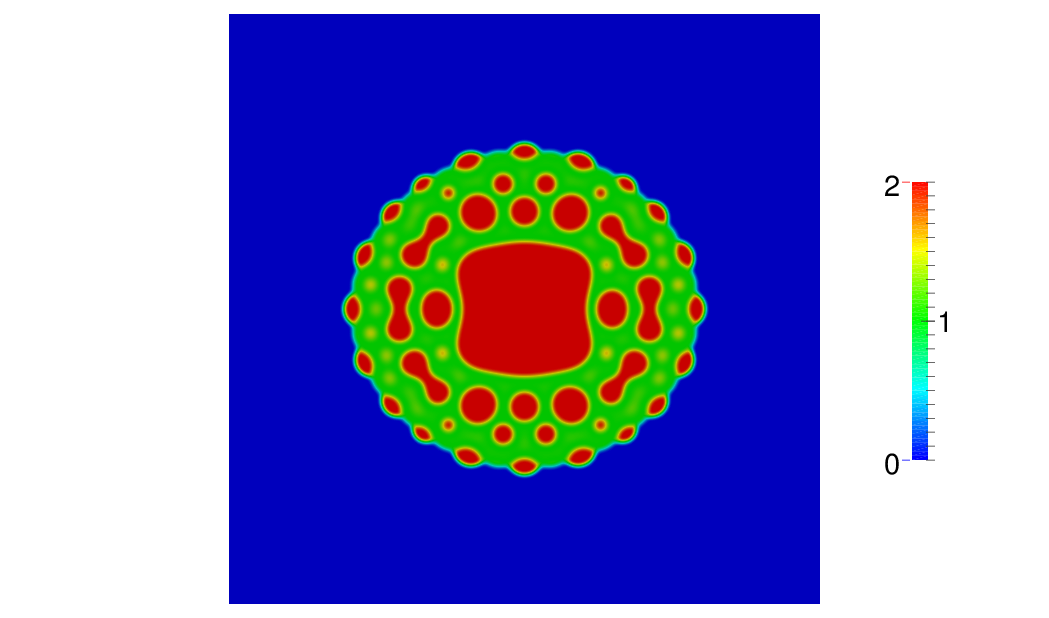} 
}
\mbox{
\includegraphics[angle=-0,width=0.25\textwidth]{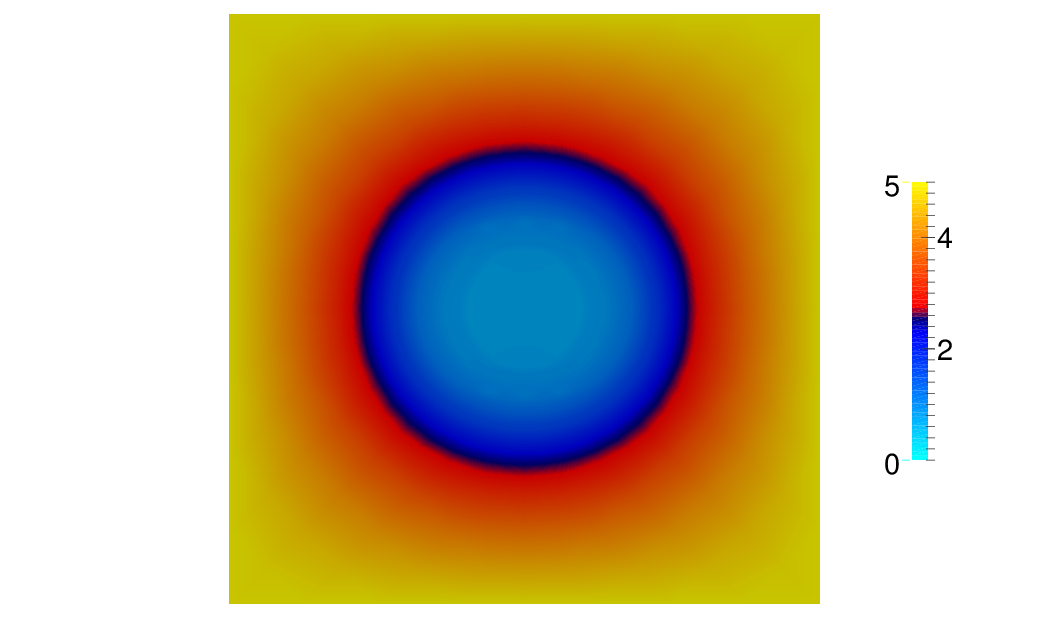} 
\includegraphics[angle=-0,width=0.25\textwidth]{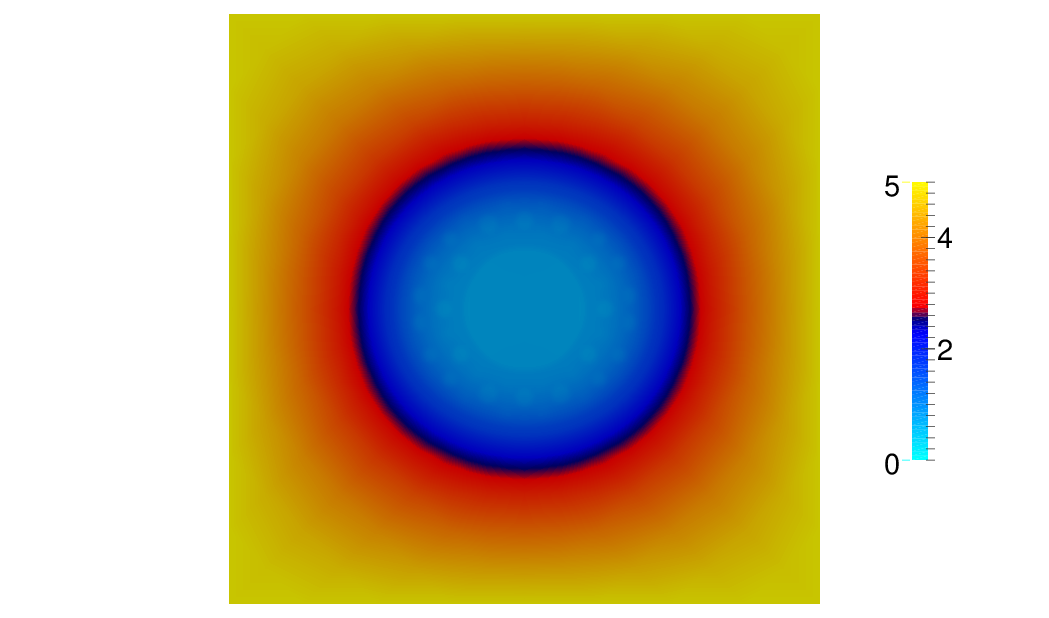} 
\includegraphics[angle=-0,width=0.25\textwidth]{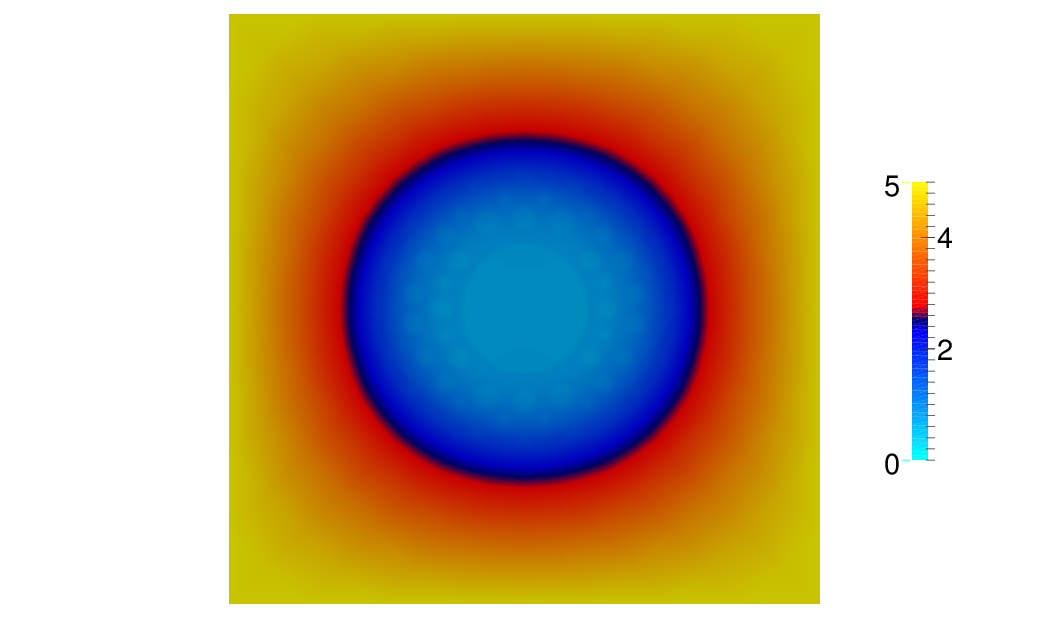} 
\includegraphics[angle=-0,width=0.25\textwidth]{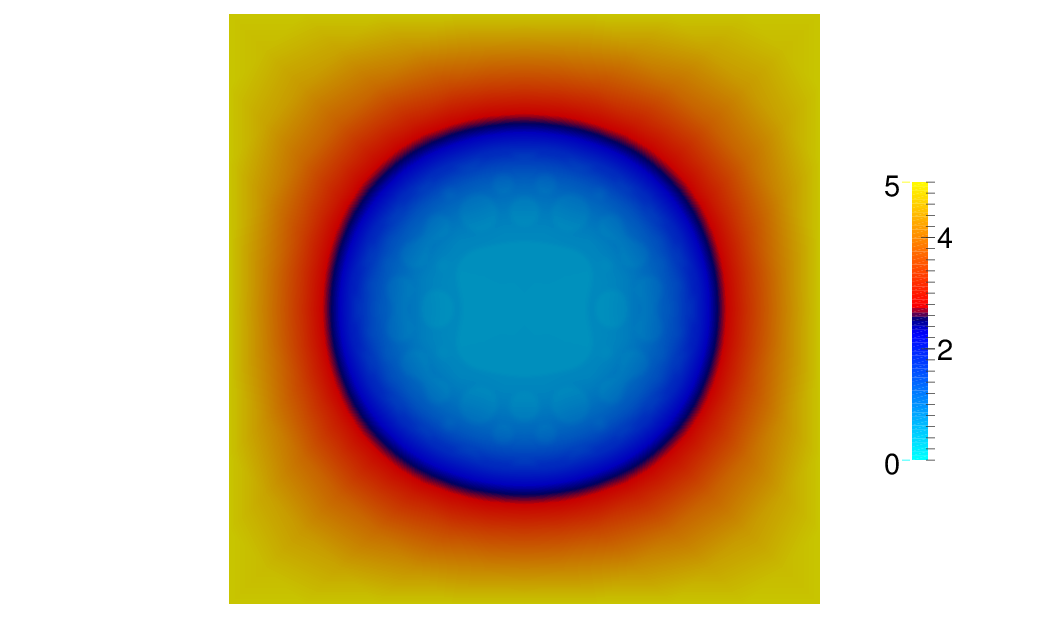} 
}
\caption{($\sigma_{B} = 5$, $K = 0.01$)
The solution $\bm{\varphi}_{h}^{n}$ at times $t=0.3,\ 0.5,\ 1,\ 2$ for \eqref{eq:U}. Below we show plots of $\sigma_{h}^{n}$ at the same times.
}
\label{fig:multifig11}
\end{figure}%
\begin{figure}[h]
\center
\mbox{
\includegraphics[angle=-0,width=0.25\textwidth]{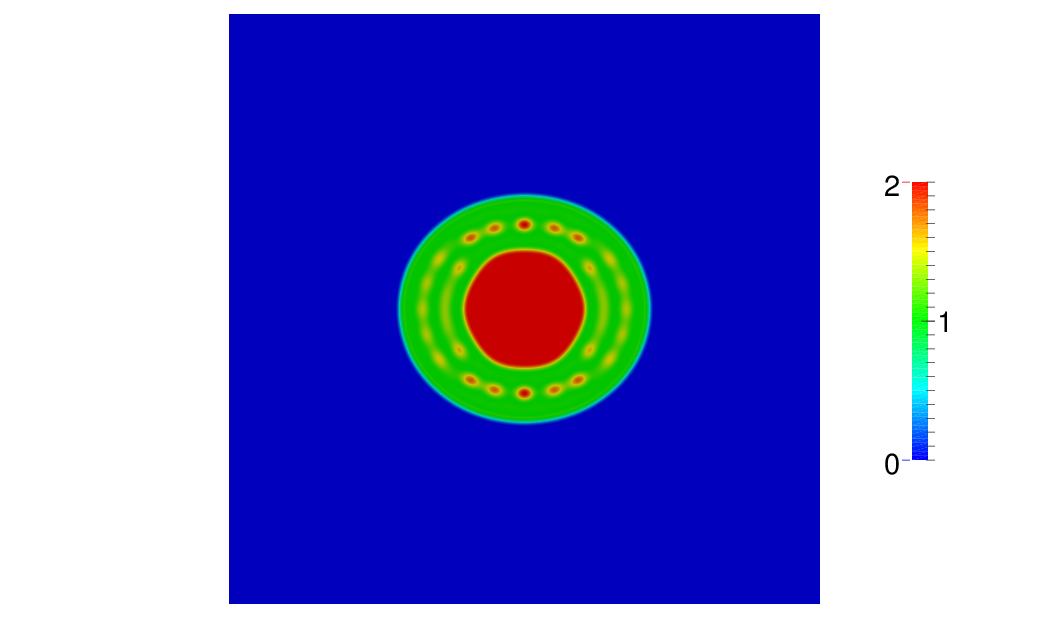} 
\includegraphics[angle=-0,width=0.25\textwidth]{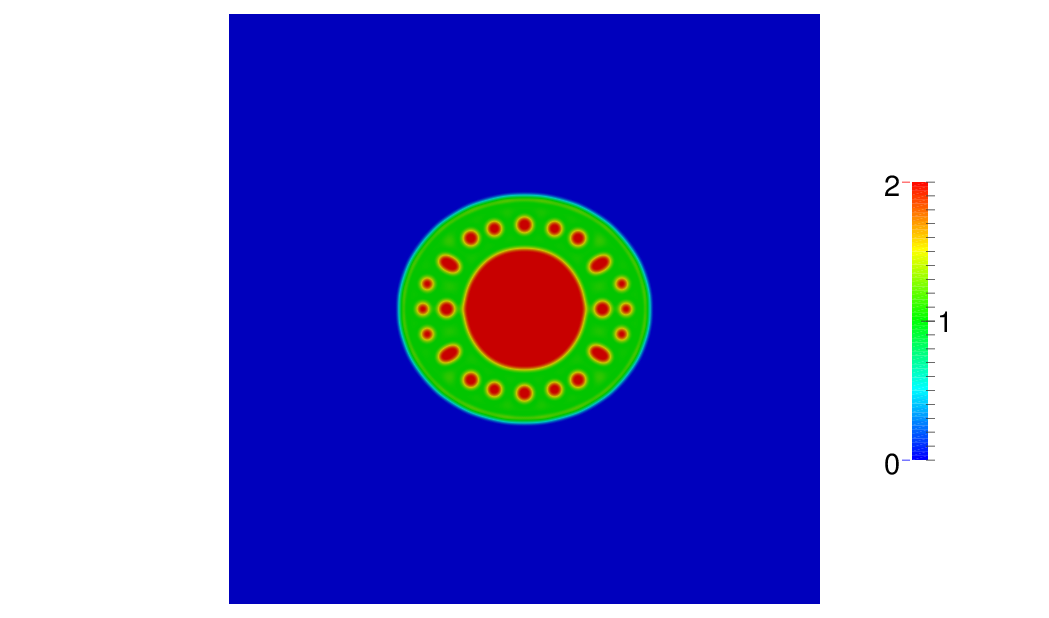} 
\includegraphics[angle=-0,width=0.25\textwidth]{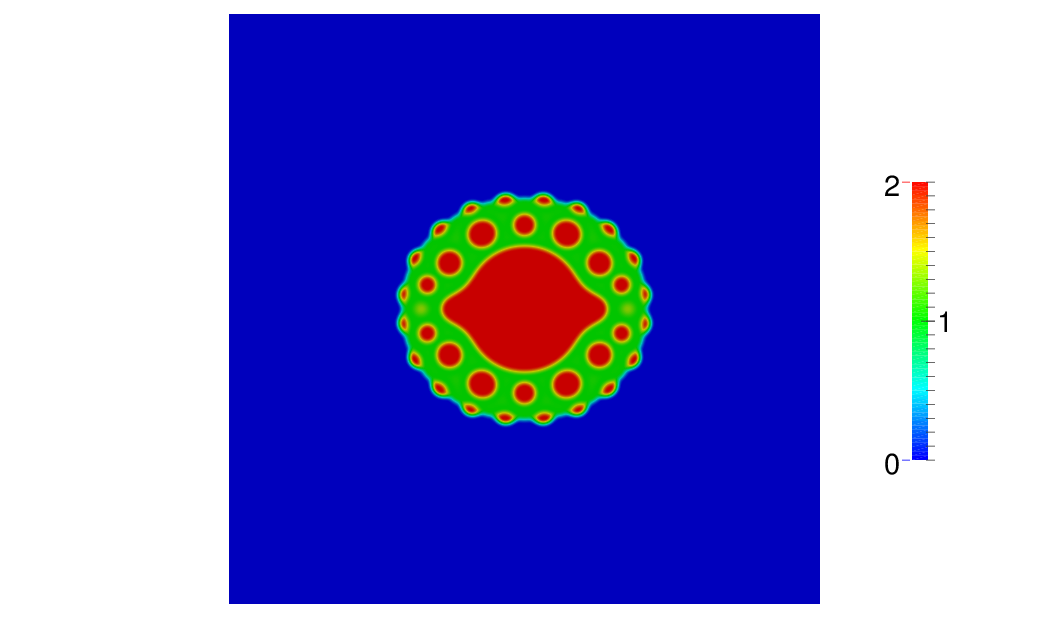} 
\includegraphics[angle=-0,width=0.25\textwidth]{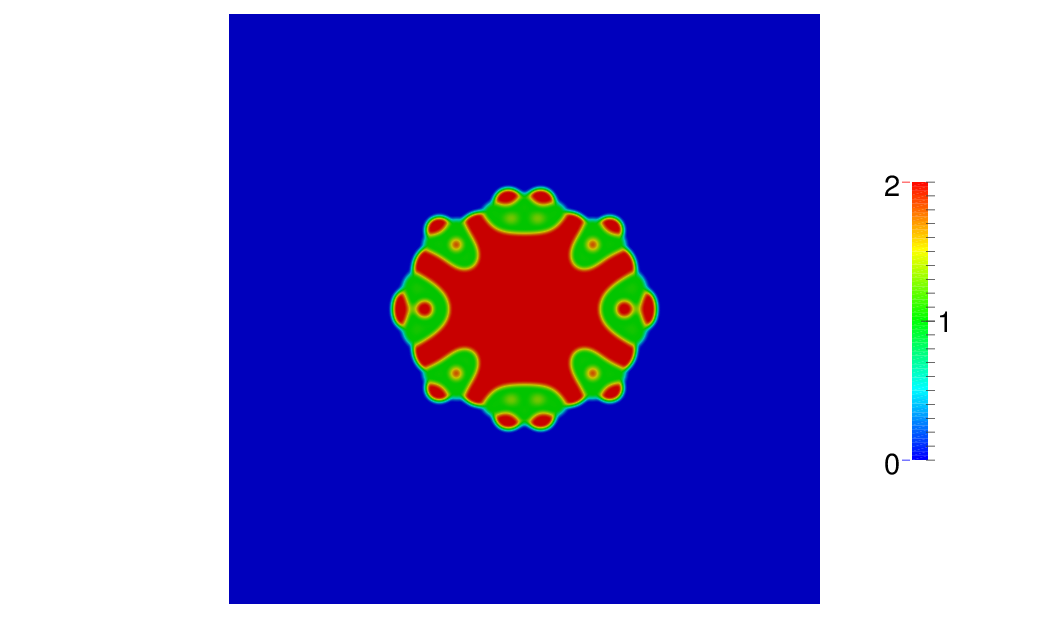} 
}
\mbox{
\includegraphics[angle=-0,width=0.25\textwidth]{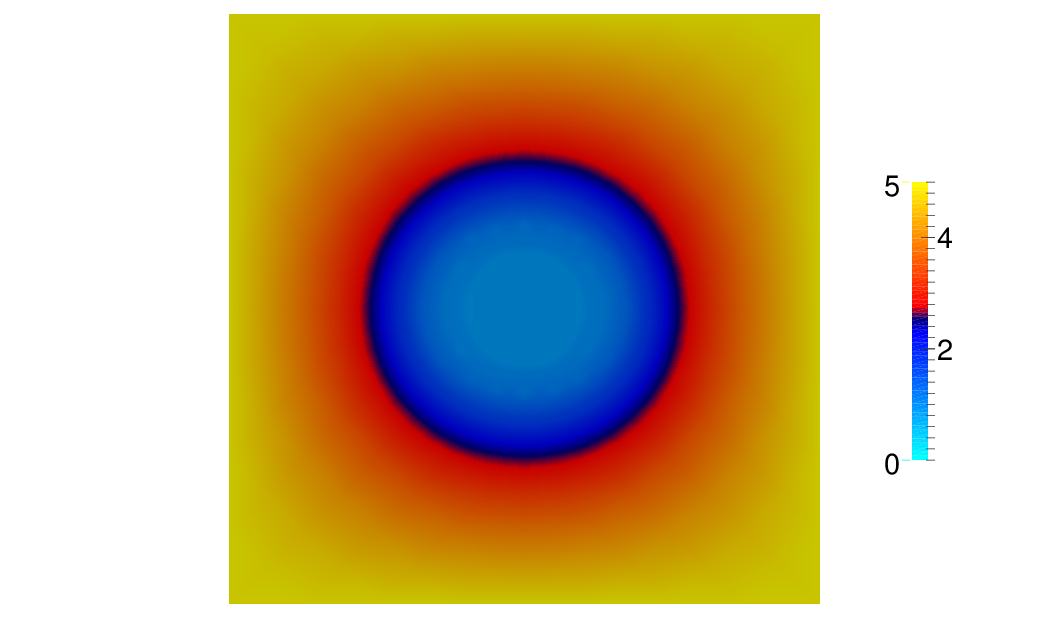} 
\includegraphics[angle=-0,width=0.25\textwidth]{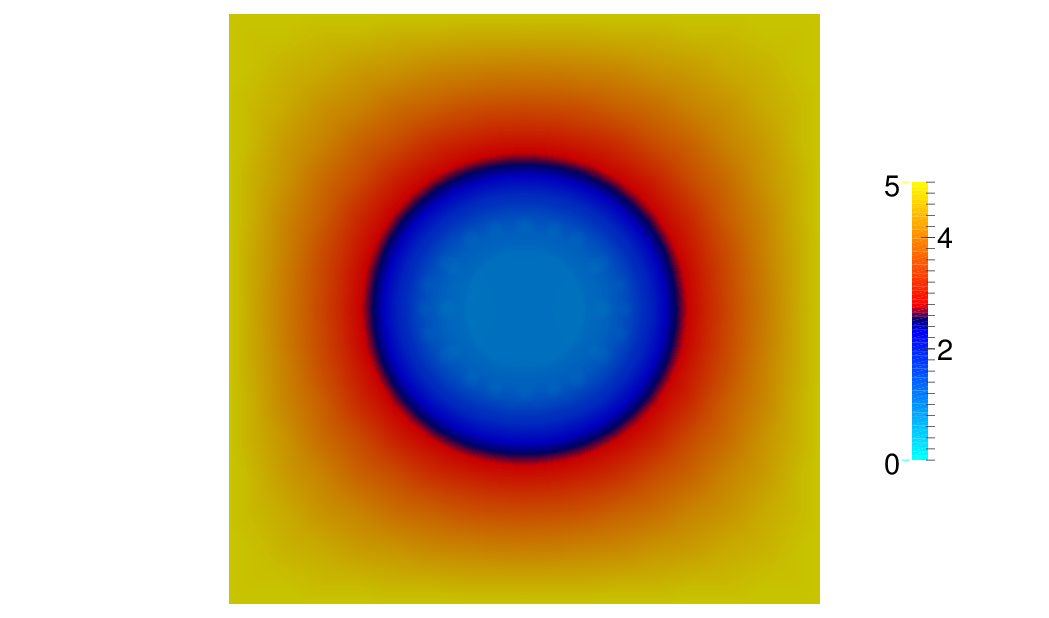} 
\includegraphics[angle=-0,width=0.25\textwidth]{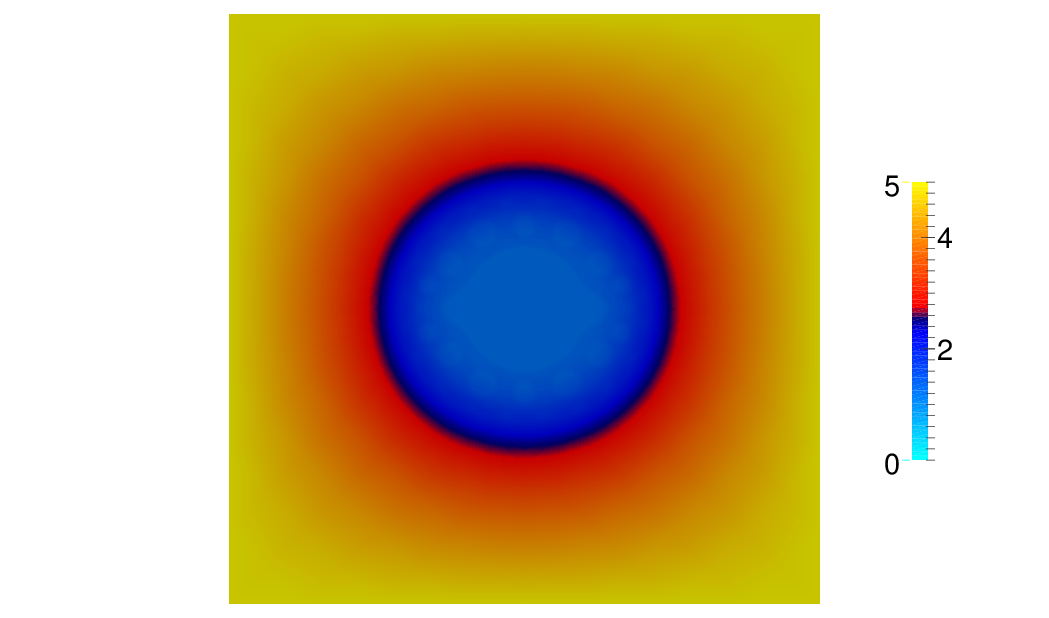} 
\includegraphics[angle=-0,width=0.25\textwidth]{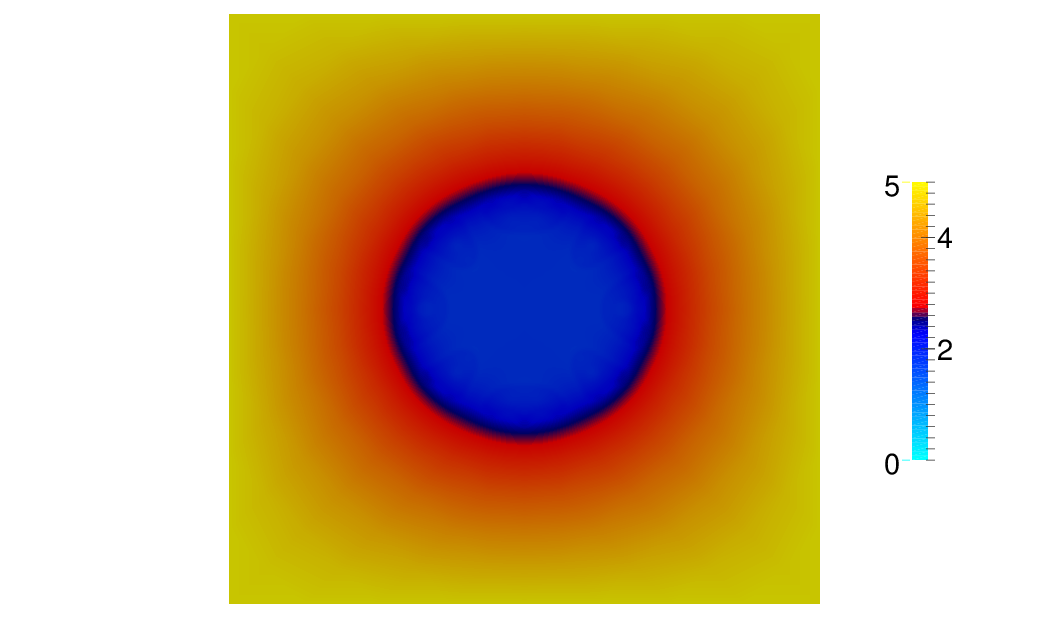} 
}
\caption{($\sigma_{B} = 5$, $K = 0$)
The solution $\bm{\varphi}_{h}^{n}$ at times $t = 0.3,\ 0.5,\ 1,\ 2$ for \eqref{eq:U}. 
Below we show plots of $\sigma_{h}^{n}$ at the same times.
}
\label{fig:multifig11nd}
\end{figure}%

Finally, we consider a numerical simulation on the larger domain $\Omega = (-10,10)^{2}$ for the source term \eqref{eq:U}, with the parameters
\begin{align*}
\mathcal{A} & = 0, \, D = 1, \, \beta = 0.1, \, 
\mathcal{C} = 2, \, \mathcal{P} = 0.1, \, \lambda = 0.02, \, \chi_{\varphi} = 5, \, \mathcal{D}_{N} = 0, \, \sigma_{B} = 1, \, K = 0.01.
\end{align*}
The evolution of the three phases is shown in Figure~\ref{fig:nnlarge1}, where we chose the initial radius $R_{3} = 1$. It can be seen that both tumour phases grow, with some instabilities developing at the tumour/host cell interface. However, if the initial necrotic phase is smaller, it vanishes and the perturbations become more pronounced, see Figure~\ref{fig:nnlarge05}.  In fact, towards the end the evolution in Figure~\ref{fig:nnlarge05} becomes similar to \cite[Fig.\ 5]{GLSS}.
\begin{figure}[h]
\center
\mbox{
\includegraphics[angle=-0,width=0.25\textwidth]{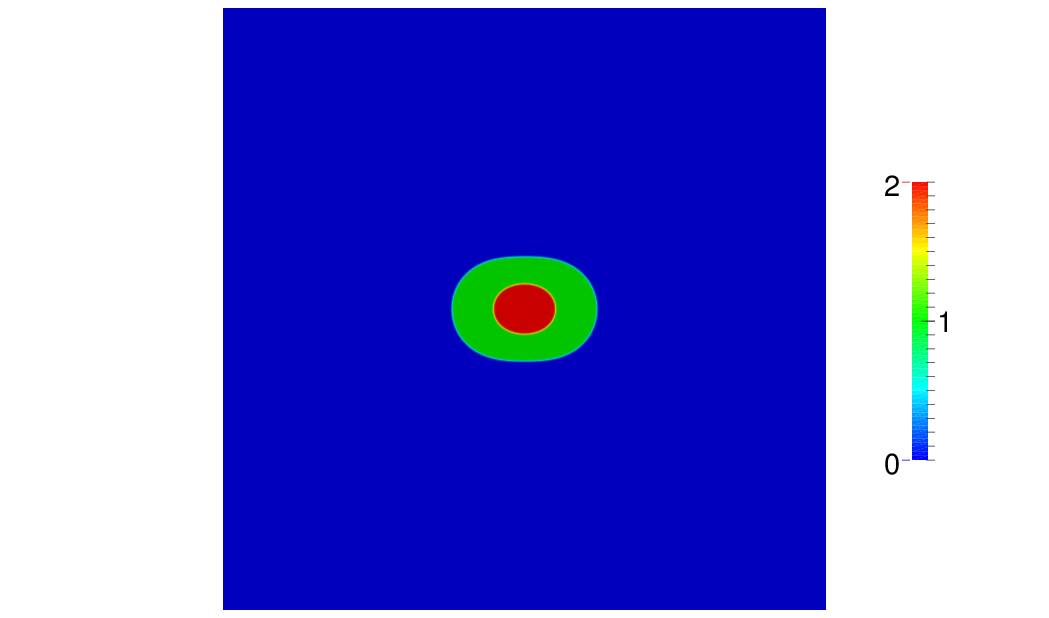} 
\includegraphics[angle=-0,width=0.25\textwidth]{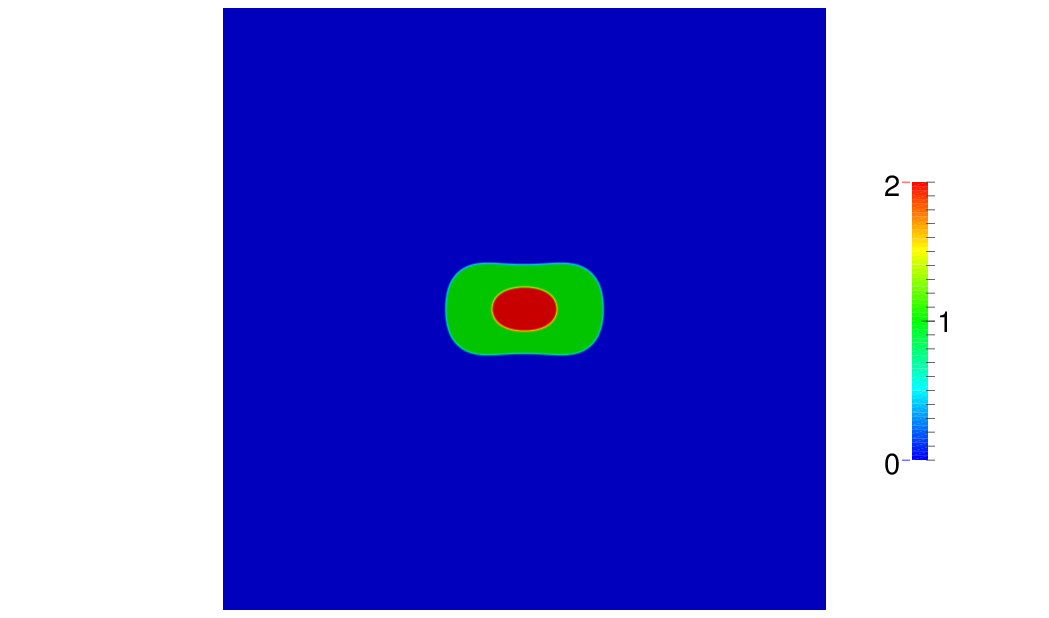} 
\includegraphics[angle=-0,width=0.25\textwidth]{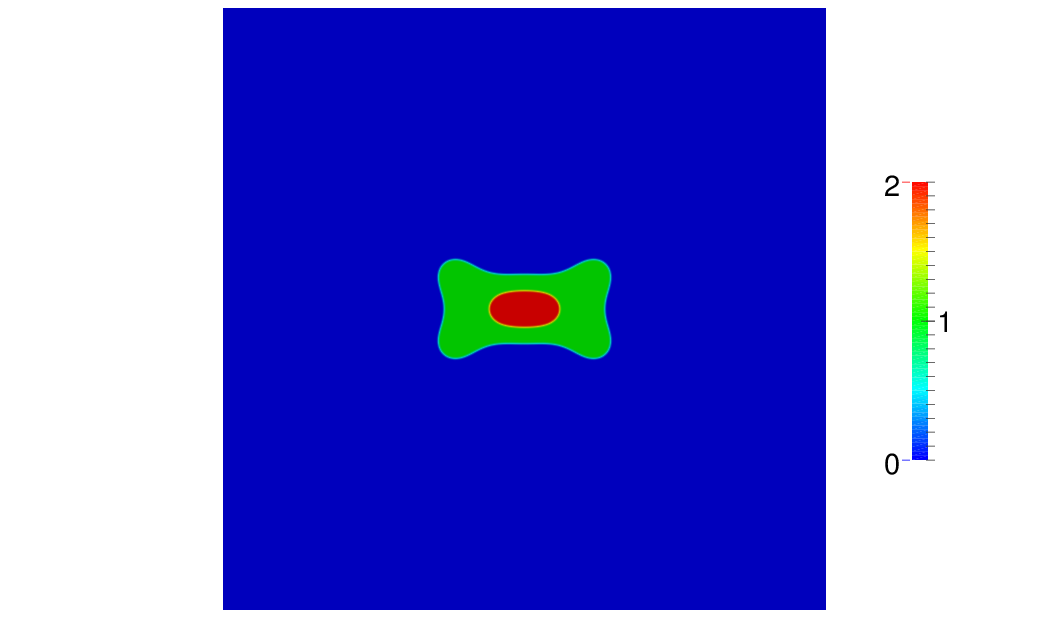} 
\includegraphics[angle=-0,width=0.25\textwidth]{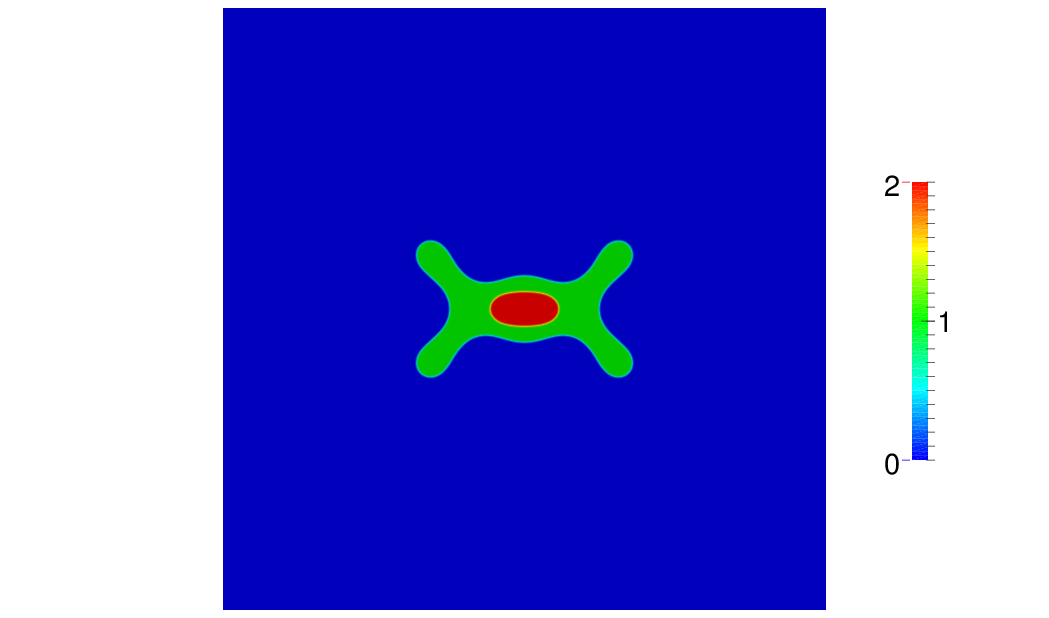} 
}
\caption{
The solution $\bm{\varphi}_{h}^{n}$ at times $t=10,\ 15,\ 20,\ 25$ for \eqref{eq:U} with the perturbed initial data \eqref{eq:pertsi} and $R_{3} = 1$.
}
\label{fig:nnlarge1}
\end{figure}%
\begin{figure}
\center
\mbox{
\includegraphics[angle=-0,width=0.25\textwidth]{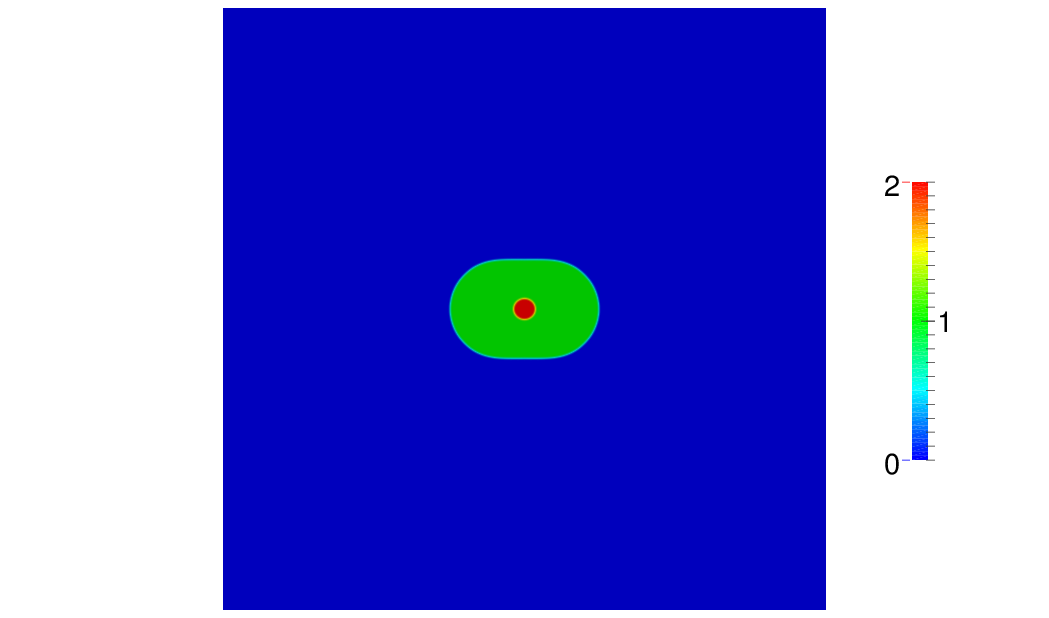} 
\includegraphics[angle=-0,width=0.25\textwidth]{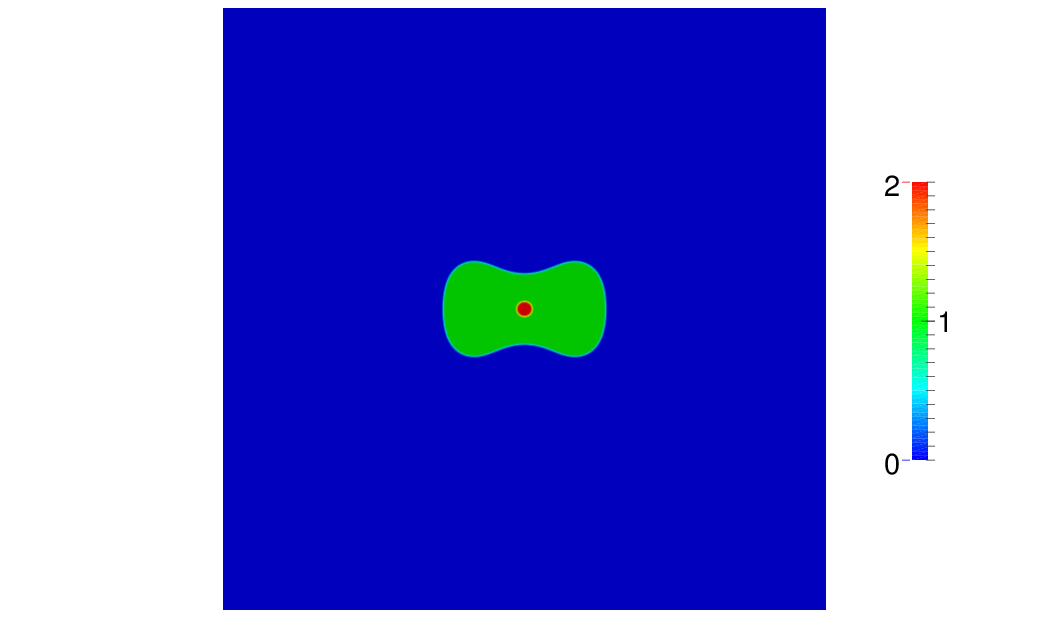} 
\includegraphics[angle=-0,width=0.25\textwidth]{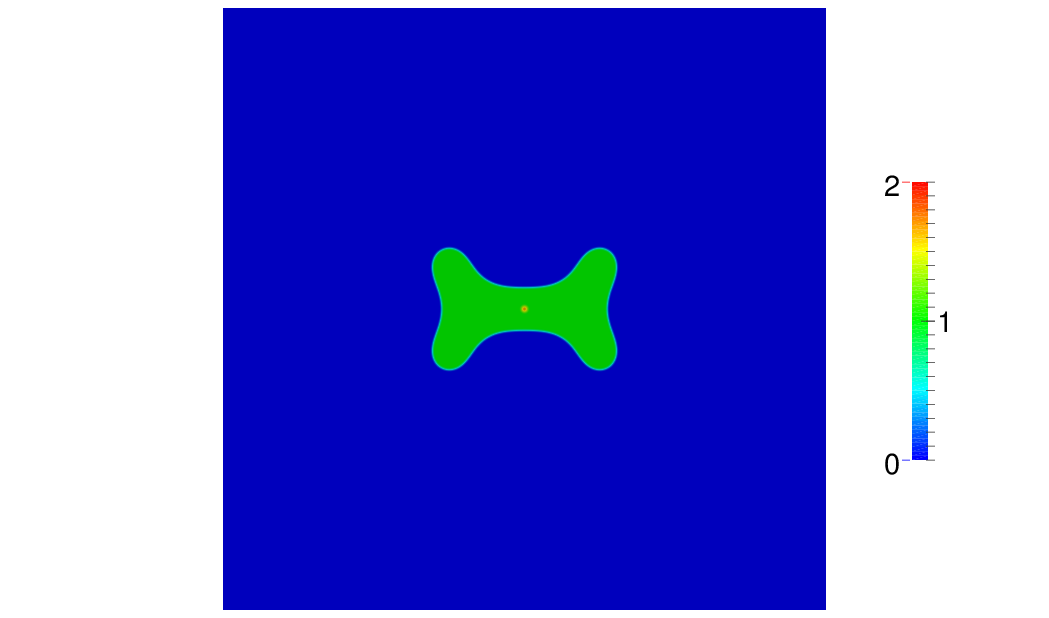} 
\includegraphics[angle=-0,width=0.25\textwidth]{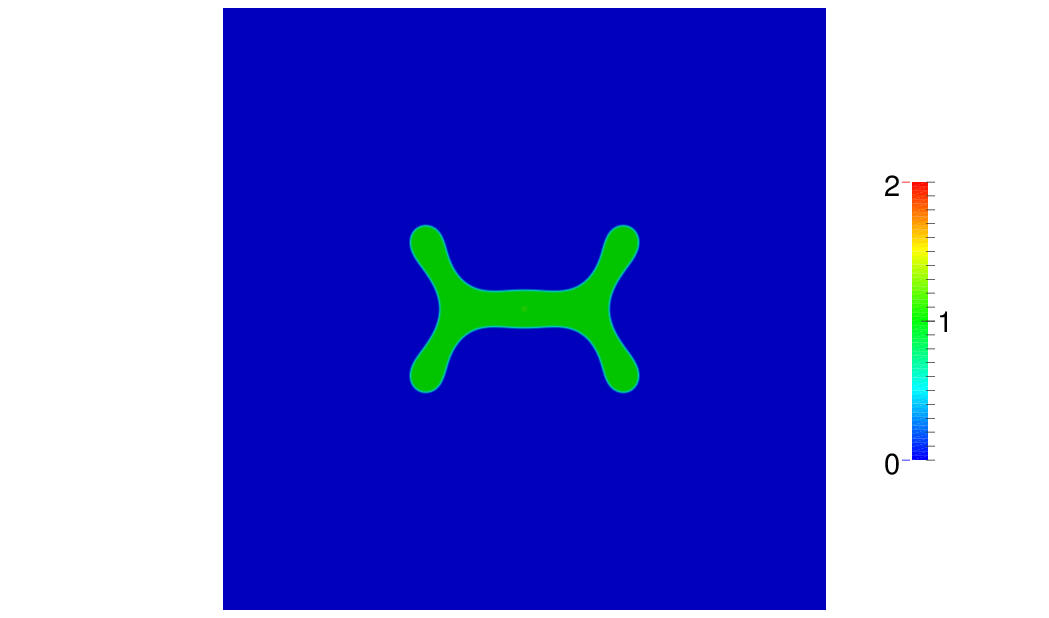} 
}
\caption{
The solution $\bm{\varphi}_{h}^{n}$ at times $t = 10,\ 15,\ 20,\ 25$ for \eqref{eq:U}
with the perturbed initial data \eqref{eq:pertsi} and $R_{3} = 0.5$.
}
\label{fig:nnlarge05}
\end{figure}%
Let us also point out that the numerical simulations with the source term \eqref{eq:Unew2} are almost identical to Figures~\ref{fig:nnlarge1} and \ref{fig:nnlarge05}, and so we omit the results.

We also investigate the effects of a larger initial necrotic core on the evolution of the tumour.  To this end, we repeat the computation in Figure~\ref{fig:nnlarge1} for the initial radius $R_{3}=1.5$ for $\mathcal{D}_{N} \in \{0, 1, 5\}$.  The three different evolutions can be seen in Figure~\ref{fig:nnlarge15}, where we observe that for positive values of $\mathcal{D}_{N}$, the necrotic core slowly disappears, and the subsequent evolution of the tumour is similar to that observed in Figure~\ref{fig:nnlarge05}.  Meanwhile, in the case where the necrotic core does not degrade, upon comparing to Figure~\ref{fig:nnlarge1}, we can conclude that a large necrotic core seems to suppress or delay the development of protrusions and leads to a more compact growth.
\begin{figure}[h]
\center
\mbox{
\includegraphics[angle=-0,width=0.25\textwidth]{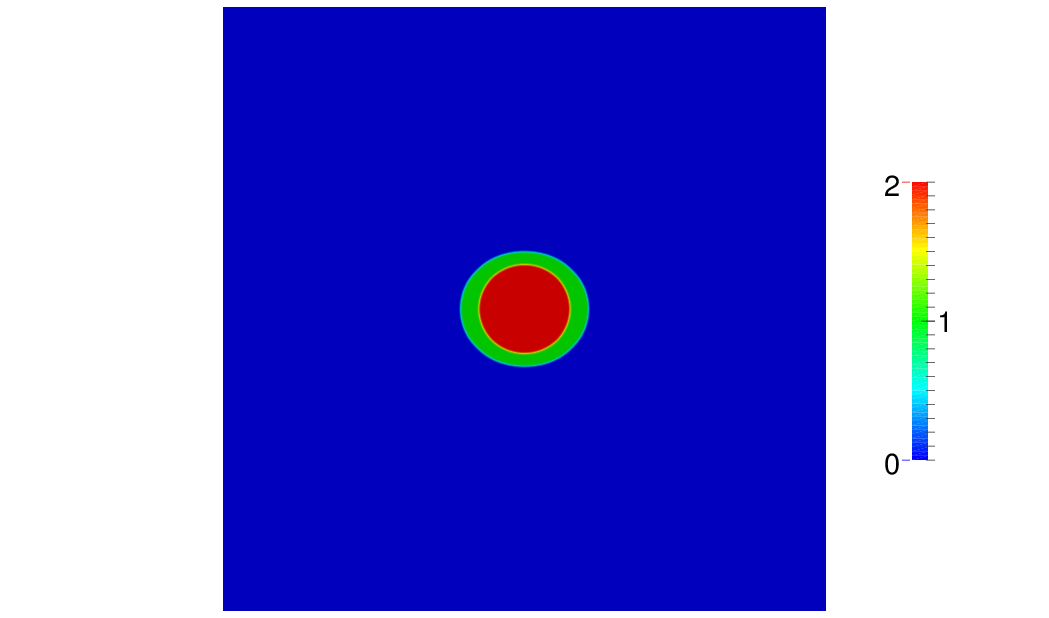} 
\includegraphics[angle=-0,width=0.25\textwidth]{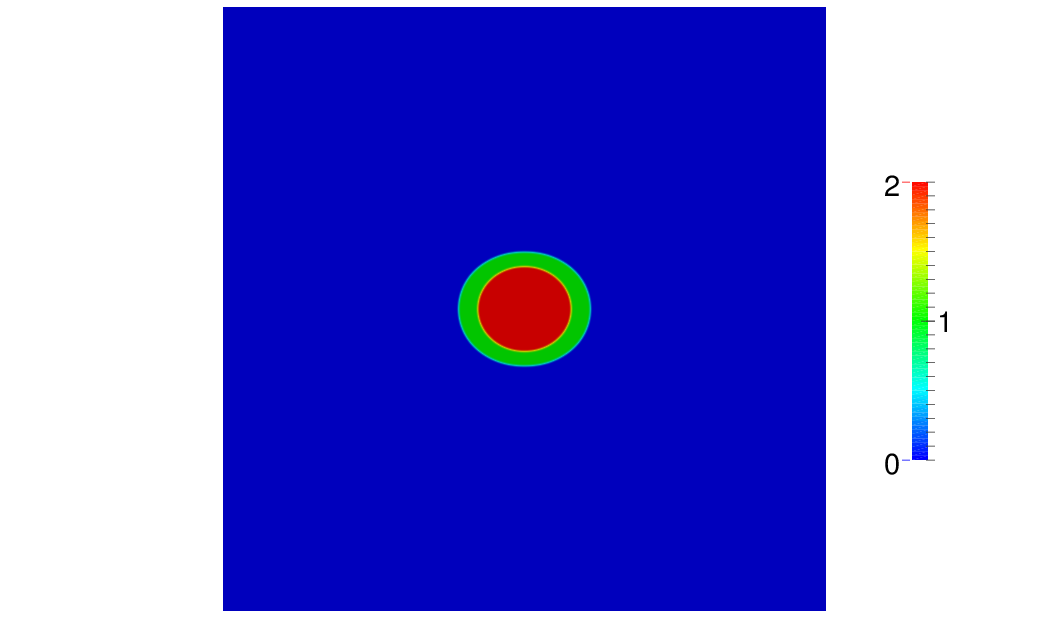} 
\includegraphics[angle=-0,width=0.25\textwidth]{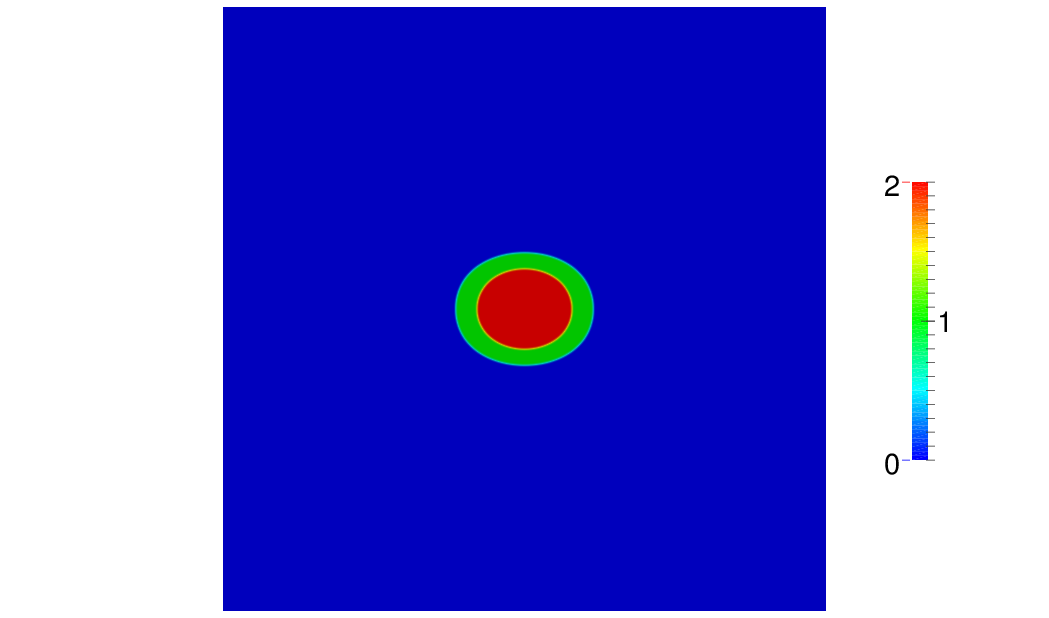} 
\includegraphics[angle=-0,width=0.25\textwidth]{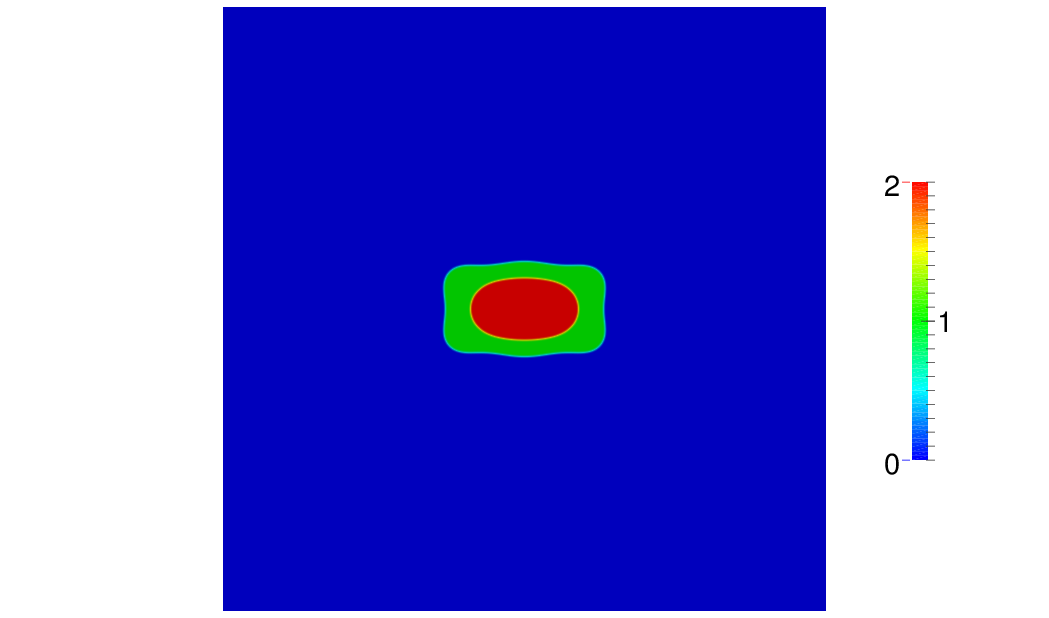} 
}
\mbox{
\includegraphics[angle=-0,width=0.25\textwidth]{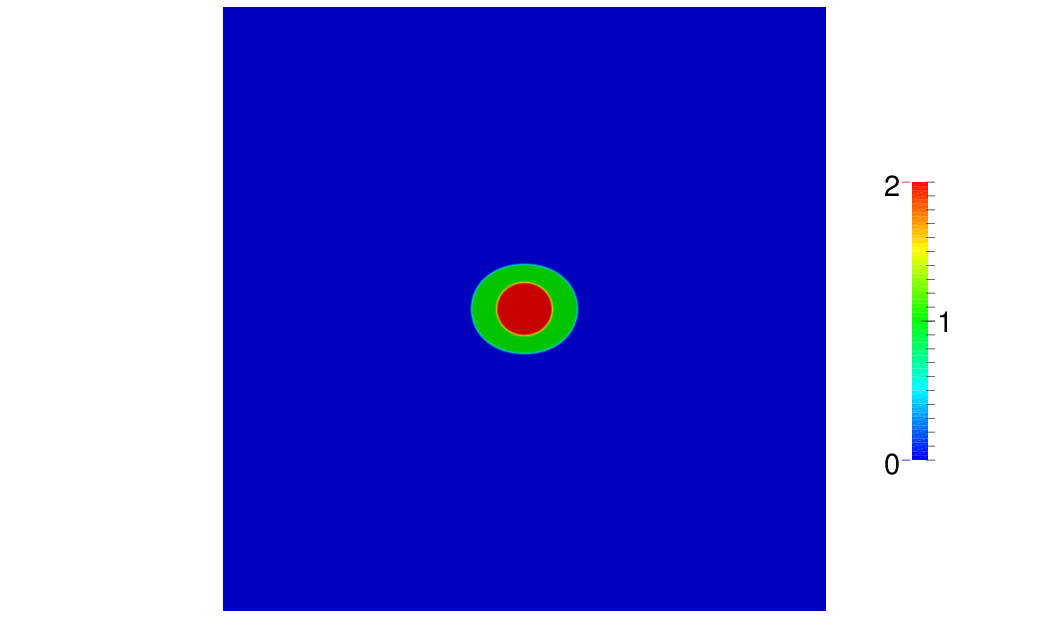} 
\includegraphics[angle=-0,width=0.25\textwidth]{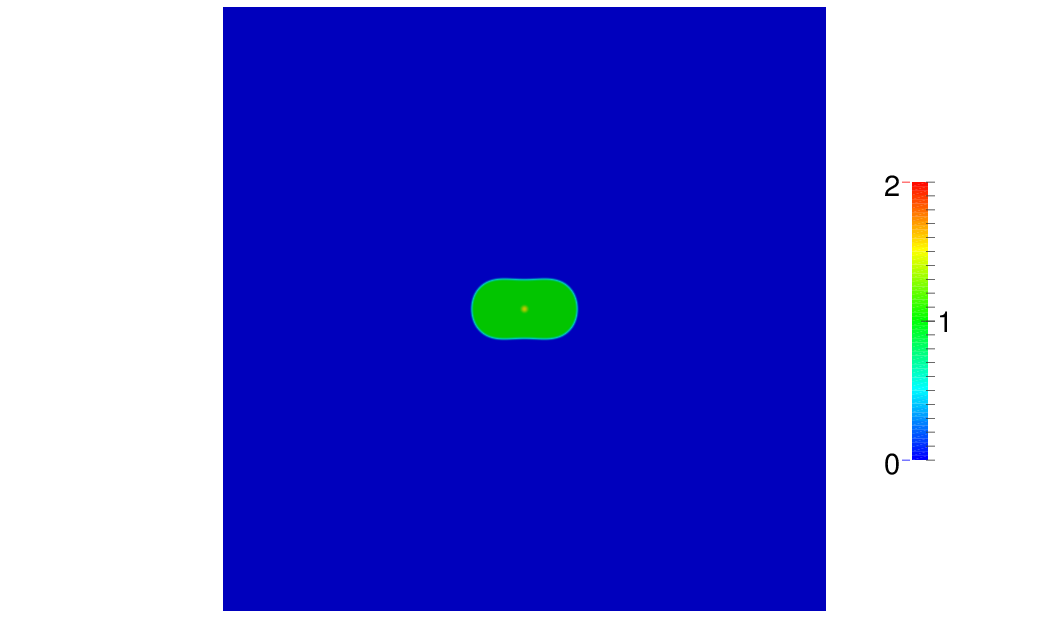} 
\includegraphics[angle=-0,width=0.25\textwidth]{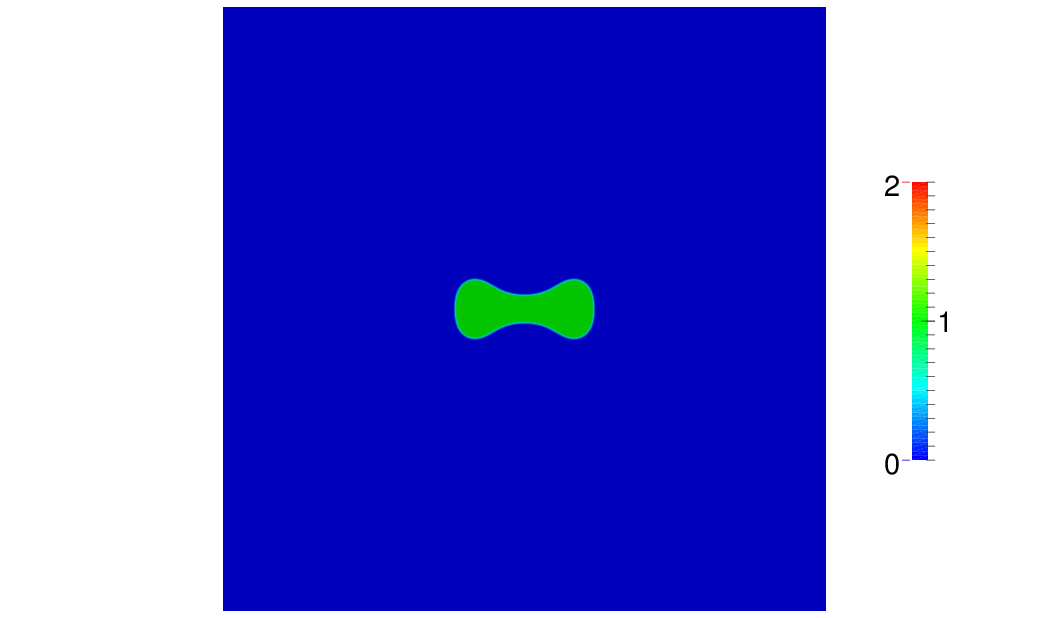} 
\includegraphics[angle=-0,width=0.25\textwidth]{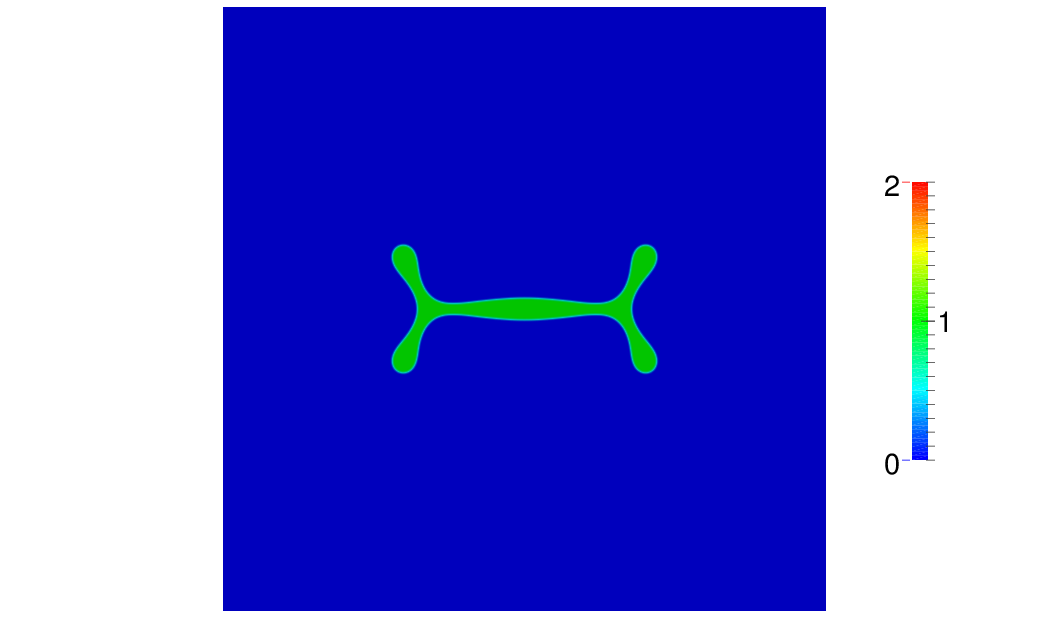} 
}
\mbox{
\includegraphics[angle=-0,width=0.25\textwidth]{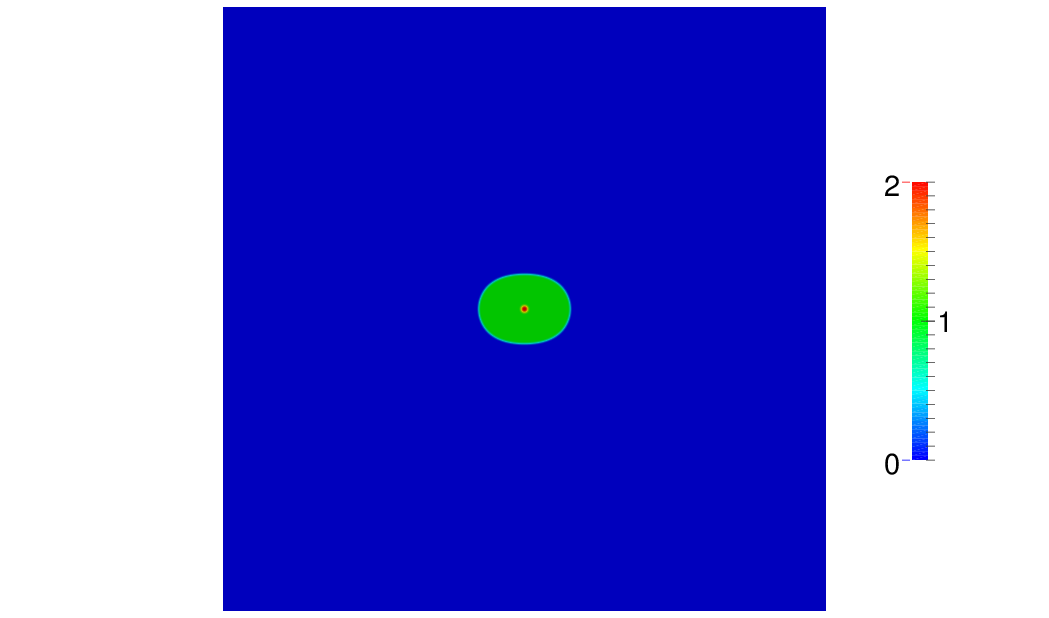} 
\includegraphics[angle=-0,width=0.25\textwidth]{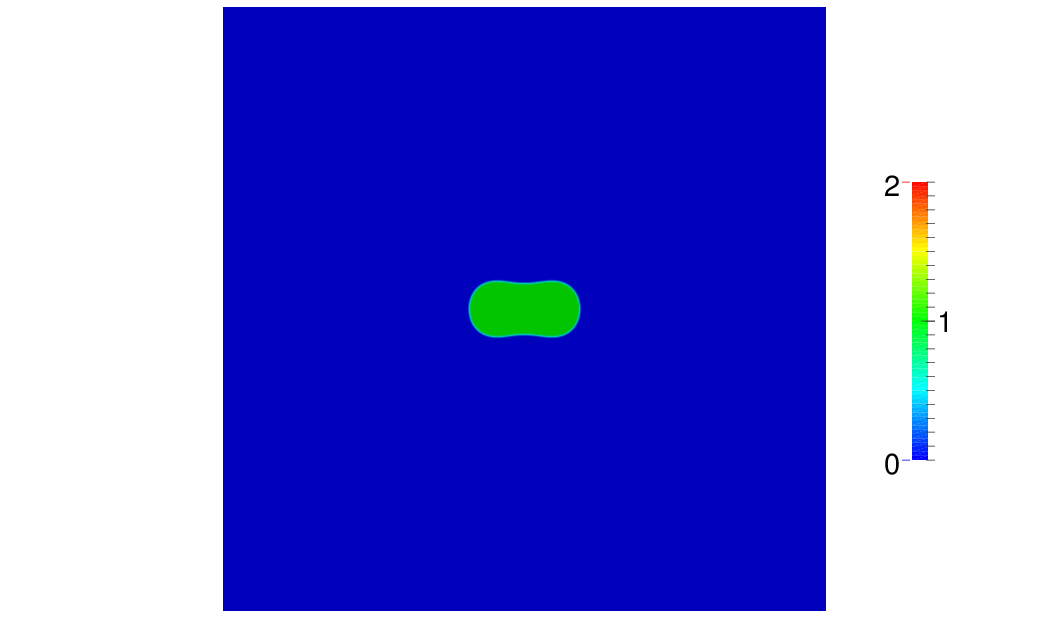} 
\includegraphics[angle=-0,width=0.25\textwidth]{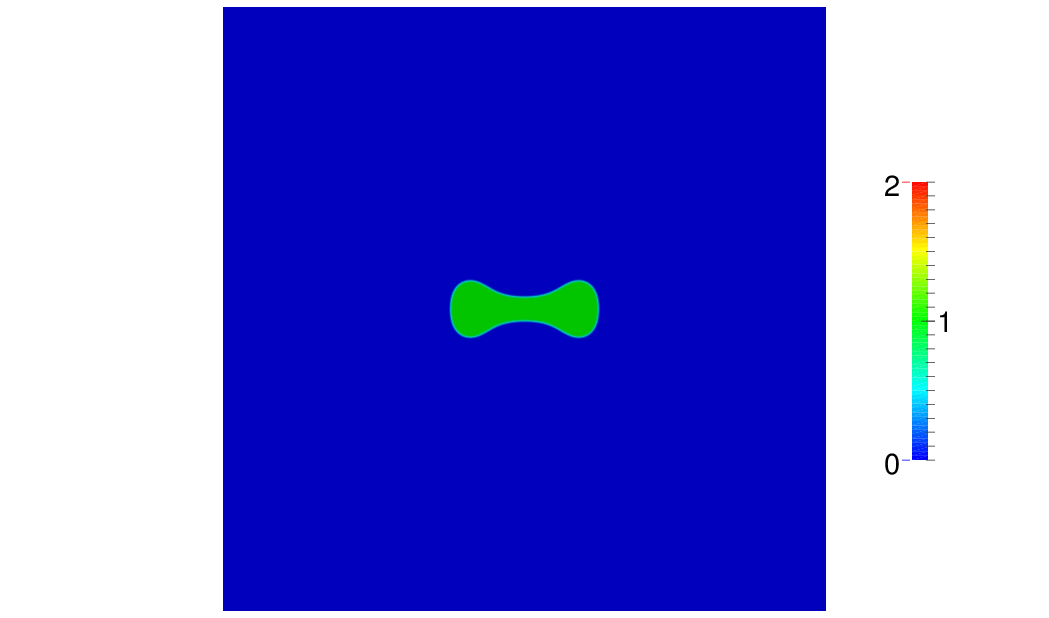} 
\includegraphics[angle=-0,width=0.25\textwidth]{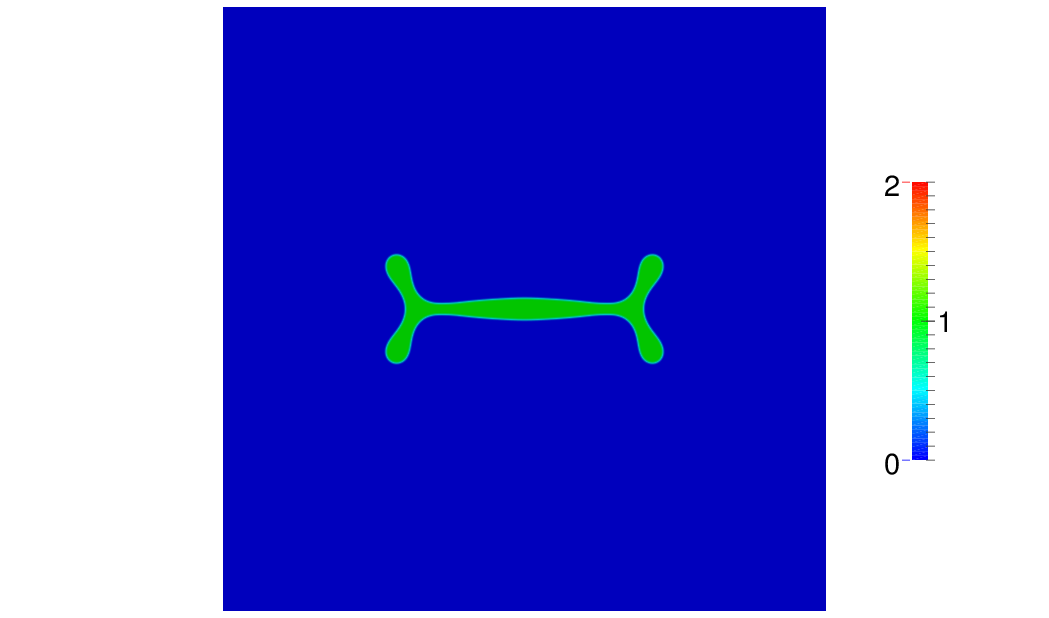} 
}
\caption{
The solution $\bm{\varphi}_{h}^{n}$ at times $t=1,\ 5,\ 10,\ 25$ for 
\eqref{eq:U} with the perturbed initial data \eqref{eq:pertsi} and 
$R_{3} = 1.5$. The top row is for $\mathcal{D}_{N} = 0$, the middle row is
for $\mathcal{D}_{N} = 1$, and the bottom row is for $\mathcal{D}_{N} = 5$.
}
\label{fig:nnlarge15}
\end{figure}%

\FloatBarrier

\section*{Acknowledgments}
The authors gratefully acknowledge the support of the Regensburger Universit\"{a}tsstiftung Hans Vielberth.

\bibliographystyle{plain}
\bibliography{Multiphase}
\end{document}